\documentclass[10pt]{article}

\usepackage{iftex}
\ifXeTeX
  
\fi

\usepackage[utf8]{inputenc}
\usepackage{CJKutf8}

\usepackage{amsmath}
\usepackage{amssymb}
\usepackage{amsfonts}
\usepackage{amsthm}

\usepackage{amscd}
\usepackage{mathrsfs}
\usepackage{bm}
\usepackage{latexsym}

\usepackage{algorithm}
\usepackage{algpseudocode}
\usepackage{cases}

\algnewcommand\algorithmicinput{\textbf{Input:}}
\algnewcommand\algorithmicoutput{\textbf{Output:}}
\algnewcommand\Input{\item[\algorithmicinput]}
\algnewcommand\Output{\item[\algorithmicoutput]}

\usepackage{indentfirst}
\usepackage{lineno}

\usepackage{booktabs}
\usepackage{multirow}
\usepackage[flushleft]{threeparttable}
\usepackage{enumitem}

\usepackage{graphicx}
\usepackage{float}
\usepackage{subcaption}
\usepackage{grffile}
\usepackage{tikz}
\usetikzlibrary{positioning, arrows.meta, calc, shapes.geometric}

\usepackage[multiple]{footmisc}

\usepackage{cite}
\usepackage[colorlinks, citecolor=blue, urlcolor=blue]{hyperref}

\usepackage{authblk}

\usepackage{geometry}
\usepackage{fancyhdr}
\fancypagestyle{plain}{
    \fancyhf{}
    \fancyfoot[C]{\thepage}

}
\newtheorem{theorem}{Theorem}
\newtheorem{remark}{Remark}
\newtheorem{lemma}{Lemma}

\newtheorem{corollary}{Corollary}
\newtheorem{assumption}{Assumption}

\numberwithin{equation}{section}
\numberwithin{theorem}{section}
\numberwithin{definition}{section}
\numberwithin{lemma}{section}
\numberwithin{remark}{section}
\numberwithin{example}{section}
\numberwithin{corollary}{section}
\numberwithin{assumption}{section}
\numberwithin{proposition}{section}

\allowdisplaybreaks[4]
\date{}

\usepackage{etoolbox}

\newcommand{\keywords}[1]{\vspace{0.2cm}\noindent{\bf Keywords:} {#1}\par}
\newcommand{\ccode}[1]{\smallskip\noindent{#1}\par}

\title{Deep Truncated FBSDE Method for High-Dimensional Nonlinear PDEs and Fully Coupled FBSDEs}

\author{
Xinyu Cheng\thanks{Research Institute of Intelligent Complex Systems, Fudan University, Shanghai 200433, China. xycheng@fudan.edu.cn.}
\quad
Yunzhang Li\thanks{Research Institute of Intelligent Complex Systems, Fudan University, Shanghai 200433, China; Department of Mathematics, Friedrich-Alexander-Universit\"at Erlangen--N\"urnberg, 91058 Erlangen, Germany. li\_yunzhang@fudan.edu.cn.}
\quad
Weiran Xiong\thanks{Corresponding author. Research Institute of Intelligent Complex Systems, Fudan University, Shanghai 200433, China. wrxiong26@m.fudan.edu.cn.}
}

\begin{document}

\maketitle

\begin{abstract}
In this paper, we introduce a deep truncated forward--backward stochastic differential equation (FBSDE) method for high-dimensional partial differential equations (PDEs) and fully coupled FBSDEs.
The proposed method combines gradient-truncated iterative decoupling with fictitious-play averaging to separate the forward and backward processes in a coupled framework.
This preserves the coupled dynamics while reducing the unstable feedback induced by parameter-dependent forward paths during optimization.
Furthermore, we incorporate a pathwise consistency term to create explicit local gradient shortcuts, thereby providing a structural mechanism that may mitigate gradient vanishing.
We also derive a residual-based error estimate and establish conditional convergence of the fully discrete numerical approximations under suitable conditions, in which the pathwise consistency loss is not required.
Our approach is particularly effective for convection-dominated equations, where the coupled formulation provides a stable representation of nonlinear transport without introducing singular terms into the BSDE.
Numerical experiments demonstrate improved accuracy and stability in both low- and high-dimensional problems and robust performance for strongly coupled problems.
\end{abstract}


\ccode{\textbf{AMS Subject Classification:} 65M75, 65C30, 60H35, 60H10, 68T07}

\section{Introduction}

Partial differential equations (PDEs) and forward--backward stochastic differential equations (FBSDEs) are linked through the nonlinear Feynman--Kac formula.
In particular, semilinear parabolic PDEs are associated with decoupled FBSDEs~\cite{MR1176785}, whereas quasilinear parabolic PDEs correspond to fully coupled FBSDEs~\cite{MR1701517}.
This probabilistic connection provides the basis for developing stochastic numerical methods for high-dimensional PDEs.
However, many classical numerical methods for PDEs and FBSDEs suffer from the curse of dimensionality.

The deep backward stochastic differential equation (BSDE) method proposed by Han, Jentzen, and E~\cite{MR3847747} provides a seminal framework for high-dimensional semilinear parabolic PDEs.
This method reformulates the PDE as a decoupled FBSDE and employs deep composite neural networks to solve the FBSDE.
A major difficulty arises for fully coupled FBSDEs, where the dependence of forward coefficients on backward variables makes the simulated forward trajectories parameter-dependent.
This can create unstable feedback between the forward simulation and backward approximation, particularly under strong coupling or over long time horizons.

In this article, we propose a new deep learning framework, termed the deep truncated FBSDE method, for high-dimensional quasilinear parabolic PDEs and the corresponding fully coupled FBSDEs.
We consider the following quasilinear parabolic PDE
\begin{equation}\label{eq:quasi_para_PDE}
\left\{
\begin{aligned}
&\partial_t u
+\tfrac12\operatorname{Tr}\bigl(
\sigma\sigma^\top(\cdot,\cdot,u)D_x^2u
\bigr)
+D_xu\,b(\cdot,\cdot,u,D_xu\,\sigma)
+f(\cdot,\cdot,u,D_xu\,\sigma)=0
\qquad \text{on }[0,T)\times\mathbb R^d,\\
&u(T,\cdot)=g
\qquad \text{on }\mathbb R^d,
\end{aligned}
\right.
\end{equation}
where \(u:[0,T]\times\mathbb{R}^d\to\mathbb{R}^m\), and
\(b:[0,T]\times\mathbb{R}^d\times\mathbb{R}^m\times\mathbb{R}^{m\times d}\to\mathbb{R}^d\),
\(\sigma:[0,T]\times\mathbb{R}^d\times\mathbb{R}^m\to\mathbb{R}^{d\times d}\),
\(f:[0,T]\times\mathbb{R}^d\times\mathbb{R}^m\times\mathbb{R}^{m\times d}\to\mathbb{R}^m\),
\(g:\mathbb{R}^d\to\mathbb{R}^m\) are given measurable functions.
The second-order term is understood componentwise, namely,
\(\bigl[\operatorname{Tr}(\sigma\sigma^\top D_x^2 u)\bigr]_j := \operatorname{Tr}\big(\sigma\sigma^\top D_x^2 u^j\big)\) for \(j=1,\dots,m\),
where \(D_x^2 u^j\) denotes the Hessian of the \(j\)-th component \(u^j\).
Associated with the PDE~\eqref{eq:quasi_para_PDE} is the \textit{fully coupled} FBSDE
\begin{equation}\label{eq:fully_coupled_FBSDE}
\left\{
\begin{aligned}
X_t &= x_0
    + \int_0^t b(s,X_s,Y_s,Z_s)\,\mathrm{d}s
    + \int_0^t \sigma(s,X_s,Y_s)\,\mathrm{d}W_s,
\\
Y_t &= g(X_T)
    + \int_t^T f(s,X_s,Y_s,Z_s)\,\mathrm{d}s
    - \int_t^T Z_s\,\mathrm{d}W_s.
\end{aligned}
\right.
\end{equation}
Assume the PDE~\eqref{eq:quasi_para_PDE} admits a unique classical solution \(u\in C^{1,2}([0,T]\times\mathbb{R}^d;\mathbb{R}^m)\), and the FBSDE~\eqref{eq:fully_coupled_FBSDE} admits a unique adapted solution \((X,Y,Z) \).
Then, by the nonlinear Feynman--Kac formula~\cite{MR1701517}, the solution \((X,Y,Z)\) of the FBSDE~\eqref{eq:fully_coupled_FBSDE} satisfies
\begin{equation}\label{eq:nonlinear_feynman_kac}
    Y_t = u(t,X_t),
    \qquad
    Z_t = D_x u(t,X_t)\,\sigma\big(t,X_t,u(t,X_t)\big),
    \qquad \forall \,t\in[0,T].
\end{equation}

\subsection{Main results and contributions}\label{subsec:main_results}

Our main contribution is a gradient-truncated iterative decoupling framework for the fully coupled FBSDE~\eqref{eq:fully_coupled_FBSDE} that preserves the forward--backward coupling in the simulated dynamics while removing gradient feedback through parameter-dependent forward trajectories during optimization.
At each iteration, a stop-gradient operator is applied to \(Y\) and \(Z\) in the forward SDE, retaining their numerical influence on the dynamics while preventing gradient propagation through the simulated forward path.
The BSDE is evaluated along a separately initialized reference path, which is subsequently updated through fictitious-play averaging~\cite{MR4250284}.
Along this path, we reconstruct \(Z\) from the spatial derivative of a neural reference process \(\widehat Y\) using automatic differentiation and the nonlinear Feynman--Kac formula~\eqref{eq:nonlinear_feynman_kac}.
The training objective combines terminal matching against a target evaluated along the current forward trajectory with time-averaged pathwise consistency between the simulated and neural reference processes, providing intermediate-time constraints.

We further establish theoretical properties of the framework.
Theorem~\ref{thm:gradient_mitigation} characterizes the propagation of terminal-loss gradients and identifies local contributions from pathwise consistency at neighboring time levels that bypass the full temporal propagation chain.
These shortcuts provide a structural mechanism that may mitigate gradient vanishing over long time horizons.
Theorem~\ref{thm:error_estimate} provides a residual-based error estimate for the fully discrete approximation in terms of time discretization, terminal loss, terminal reference mismatch, reference-path discrepancy, and decoupling error.
Under the assumptions of Section~\ref{sec:error_analysis},
Theorem~\ref{thm:error_convergence} establishes convergence of the fully discrete approximation to the original coupled FBSDE solution as the time mesh is refined and the number of decoupling steps increases.
Neither the error estimate nor the convergence result requires the pathwise consistency term.

On the numerical side, we present numerical experiments in both low and high dimensions and compare the proposed method with the deep BSDE method and several related approaches.
The experiments show that the proposed method achieves improved accuracy and stability across the numerical comparisons, while maintaining competitive computational cost, and the ablation results support the stabilizing role of gradient truncation.
In particular, for degenerate convection-dominated equations, the coupled FBSDE formulation avoids the singular generator terms arising in the decoupled formulation and enables stable approximation of nonlinear transport.

\subsection{Related works}

Following the breakthrough of the deep BSDE method~\cite{MR3736669, MR3847747}, several related numerical schemes have been developed.
The backward deep BSDE method~\cite{wang2018deeplearningbasedbsdesolver} was designed for optimal stopping problems by solving BSDEs in a backward manner, with its convergence established in~\cite{MR4673346}.
From the perspective of dynamic programming, the deep backward dynamic programming (DBDP) methods~\cite{MR4081911} decompose the global problem into local optimization problems solved sequentially by backward induction.
Since the forward trajectories are generated independently in advance and the sequence of local training problems is already computationally demanding, these schemes are primarily designed for decoupled FBSDEs and do not extend directly to the fully coupled setting; see also the extension to fully nonlinear PDEs in~\cite{MR4322044}.
Related stochastic approaches have also been extended to other classes of PDEs and stochastic systems.
Notable extensions include fully nonlinear parabolic PDEs through second-order BSDEs~\cite{MR3993178}, semilinear elliptic problems~\cite{MR4808368}, forward--backward doubly SDEs~\cite{MR4778774}, BSDEs driven by jump processes~\cite{MR4945146}, partial integro-differential equations and FBSDEs with jumps~\cite{MR4917761, wang2025deepforwardbackwarddynamicprogramming}, and PDEs with fractional Laplacian associated with fat-tailed L\'evy measures~\cite{MR4954338}.

Despite these advances, extending deep learning methods to fully coupled FBSDEs remains challenging.
Han and Long~\cite{MR4122227} extended the deep BSDE method to coupled FBSDEs whose drift and diffusion coefficients are both independent of $Z$, and established a posteriori error estimates.
Jiang and Li~\cite{MR4399896} further extended the a posteriori error analysis to coupled FBSDEs with non-Lipschitz diffusion coefficients.
The forward--backward stochastic neural network (FBSNN)~\cite{raissi2024forward} approximates the global solution using a single neural network.
Zhang and Cai~\cite{MR4476007} further developed single-network FBSDE schemes by minimizing pathwise differences between two stochastic processes and handling the forward--backward coupling directly, without the iterative decoupling and reference-path construction used in our multi-network framework.
Andersson et al.~\cite{MR4551869} introduced a robust deep FBSDE method for stochastic control problems together with an error analysis, and subsequently proposed the deep multi-FBSDE method~\cite{MR5028798}, a two-phase approach for coupled FBSDEs whose diffusion coefficient is independent of both $Y$ and $Z$.
This two-phase method first approximates the common initial value of a family of equivalent FBSDEs and then solves the original system with this value fixed, reducing the difficulty of coupled optimization but requiring additional choices concerning the FBSDE family and computational allocation.
Ji et al.~\cite{8982030} proposed three algorithms for solving fully coupled FBSDEs, while Ji et al.~\cite{MR4932651} developed a control method for a high-dimensional stochastic Hamiltonian system, which is essentially a fully coupled FBSDE.
These approaches differ in how they handle forward--backward coupling in the numerical scheme and during optimization, a distinction central to the design of our method.

\subsection{Outline of the paper}

The paper is organized as follows.
Section~\ref{sec:existing_methods_and_motivation} reviews existing methods and explains our motivation, while Section~\ref{sec:deep_truncated_fbsde} presents our method.
Section~\ref{sec:analysis_gradient} analyzes the gradient structure of the discrete scheme.
Section~\ref{sec:error_analysis} provides residual-based error estimates and a conditional convergence result.
Section~\ref{sec:numerical_experiments} presents numerical experiments in low and high dimensions, and Section~\ref{sec:conclusions} concludes the paper.

\section{Existing methods and motivation}\label{sec:existing_methods_and_motivation}

We first introduce notation, then review related deep FBSDE methods and explain our motivation.

\subsection{Notation}\label{subsec:notation}

Let $(\Omega, \mathcal{F}, \mathbb{F}, \mathbb{P})$ be a complete filtered probability space on which a standard $d$-dimensional Brownian motion $W = \{W_t\}_{t \in [0,T]}$ is defined, with $\mathbb{F} \equiv \{\mathcal{F}_t\}_{t \in [0,T]}$ being its natural filtration.
Letting \(\mathbb{H} = \mathbb{R}^d, \mathbb{R}^m, \mathbb{R}^{m \times d}\), etc., with \(| \cdot| \) being its norm, we define the following spaces
\begin{align*}
    \mathcal{S}_{\mathbb{F}}^2(0,T; \mathbb{H})
    &:= \Bigl\{ X \colon [0,T] \times \Omega \to \mathbb{H} \;\bigm|\; X \text{ is } \mathbb{F}\text{-adapted, a.s.\ continuous, and } \mathbb{E} \bigl[ \sup_{0 \leq t \leq T} |X_t|^2 \bigr] < +\infty \Bigr\}, \\
    \mathcal{L}_\mathbb{F}^{2}(0, T; \mathbb{H})
    &:= \Bigl\{ X \colon [0, T] \times \Omega \to \mathbb{H} \;\bigm|\; X \text{ is } \mathbb{F}\text{-adapted and } \mathbb{E} \bigl[ \int_0^T |X_t|^2 \,\mathrm{d}t \bigr] < +\infty \Bigr\}.
\end{align*}
An adapted solution of~\eqref{eq:fully_coupled_FBSDE} belongs to $\mathcal S_{\mathbb F}^2(0,T;\mathbb R^d)\times\mathcal S_{\mathbb F}^2(0,T;\mathbb R^m)\times\mathcal L_{\mathbb F}^2(0,T;\mathbb R^{m\times d})$ and satisfies the equations in the usual It\^o sense.
Existence and uniqueness results for coupled FBSDEs have been established under various conditions, including sufficiently small time horizons, monotonicity conditions, and weak coupling conditions; see, e.g.,~\cite{MR1233625,MR1355060,MR1675098,MR1701517}.
Throughout this paper, $|\cdot|$ and $\|\cdot\|$ denote the Euclidean norm and the Frobenius norm, respectively.
In addition, $\|A\|_2:=\sup_{x\in\mathbb R^d\setminus\{0\}}|Ax|/|x|$ denotes the induced matrix $2$-norm for $A\in\mathbb R^{m\times d}$, and satisfies $\|A\|_2\le\|A\|$.
For a partition $0=t_0<t_1<\cdots<t_N=T$, set $\Delta t_n:=t_{n+1}-t_n$ and $\Delta W_n:=W_{t_{n+1}}-W_{t_n}$.
We denote the discrete approximations by $\widetilde X_n\approx X_{t_n}$ and $\widetilde Y_n\approx Y_{t_n}$ for $n=0,\ldots,N$, and $\widetilde Z_n\approx Z_{t_n}$ for $n=0,\ldots,N-1$.

\subsection{Deep FBSDE methods}

For a decoupled FBSDE, the deep BSDE method~\cite{MR3736669,MR3847747} uses Euler discretization with $\widetilde Y_0=\theta_{u_0}$ and $\widetilde Z_n=\phi_n(\widetilde X_n;\theta_n)$, where $\theta_{u_0}$ is the trainable initial value and $\phi_n(\cdot;\theta_n)$ is a neural network with parameters $\theta_n$.
The parameters $\Theta=\{\theta_{u_0},\theta_0,\ldots,\theta_{N-1}\}$ minimize the terminal mismatch
\begin{equation}\label{eq:DeepBSDE_Loss}
    \Theta^*=\operatorname*{argmin}_{\Theta}\mathbb E\bigl[|g(\widetilde X_N)-\widetilde Y_N(\Theta)|^2\bigr].
\end{equation}
The forward trajectory $\{\widetilde X_n\}_{n=0}^N$ is independent of $\Theta$, so parameter updates act only through the backward recursion.

For a fully coupled FBSDE, this separation is no longer available, since the forward trajectory depends on the approximations of $Y$ and $Z$.
A typical parametrization takes the form $\widetilde Z_n=\phi_n(\widetilde X_n,\widetilde Y_n;\theta_n)$.
This formulation includes the coupled scheme of Han and Long~\cite{MR4122227} and its extension in Algorithm~1 of Ji et al.~\cite{8982030} to forward coefficients depending on $Z$.
The associated loss still relies on terminal matching
\begin{equation}\label{eq:DeepFBSDE_Loss}
    \Theta^*=\operatorname*{argmin}_{\Theta}\mathbb E\bigl[|g(\widetilde X_N(\Theta))-\widetilde Y_N(\Theta)|^2\bigr].
\end{equation}
Unlike in~\eqref{eq:DeepBSDE_Loss}, the terminal forward state now depends on the parameters, so each parameter update changes the forward trajectory on which the loss is evaluated.

\subsection{Motivation: breaking the vicious cycle}

The parameter dependence described above can lead to recursive error amplification in the joint simulation and optimization of fully coupled FBSDEs.
Since the forward trajectory \(X(\Theta)\) depends on the trainable approximations of \(Y\) and \(Z\), errors in these approximations also affect the simulated forward paths.
As \(\Theta\) changes during training, the distribution of these trajectories varies, potentially making the optimization less stable.
Moreover, since the terminal loss~\eqref{eq:DeepFBSDE_Loss} is evaluated along \(X(\Theta)\), it may decrease by modifying the forward paths rather than improving the approximation of the true backward solution.
Consequently, the terminal condition may be fitted along distorted trajectories without accurately learning the underlying backward dynamics.
These two effects reinforce each other and form a vicious cycle, as illustrated in Figure~\ref{fig:vicious_cycle}.
Inaccurate approximations of \(Y\) and \(Z\) distort the forward trajectories, while optimization along these paths may further amplify the errors.
This mechanism can be particularly pronounced under strong coupling or over long time horizons.

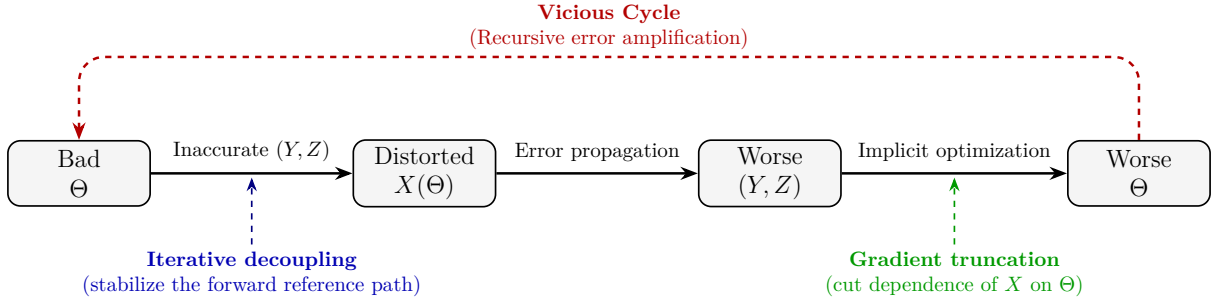
\begin{figure}[!htbp]
    \centering
    \resizebox{\linewidth}{!}{\begin{tikzpicture}[
        >=Stealth,
        every node/.style={font=\large},
        state/.style={
            rectangle,
            rounded corners=6pt,
            draw=black,
            thick,
            minimum width=2.4cm,
            minimum height=1.1cm,
            align=center,
            fill=gray!8
        },
        mainarrow/.style={->, very thick},
        cyclearrow/.style={->, very thick, dashed, draw=red!70!black},
        note/.style={font=\normalsize, align=center}
    ]

    \node[state] (theta1) {Bad\\ $\Theta$};
    \node[state, right=3.4cm of theta1] (x) {Distorted\\ $X(\Theta)$};
    \node[state, right=3.4cm of x] (yz) {Worse\\ $(Y,Z)$};
    \node[state, right=3.8cm of yz] (theta2) {Worse\\ $\Theta$};

    \draw[mainarrow] (theta1) -- node[above=2pt, note] {Inaccurate $(Y,Z)$} coordinate[midway] (mid1) (x);
    \draw[mainarrow] (x) -- node[above=2pt, note] {Error propagation} (yz);
    \draw[mainarrow] (yz) -- node[above=2pt, note] {Implicit optimization} coordinate[midway] (mid2) (theta2);

    \draw[cyclearrow, rounded corners=12pt]
        (theta2.north) -- ++(0, 1.4) coordinate (tr)
        -- node[above, note, text=red!70!black] {\textbf{Vicious Cycle}\\ (Recursive error amplification)} (theta1.north |- tr)
        -- (theta1.north);

    \node[note, text=blue!70!black, below=1.2cm of mid1] (picard_note)
        {\textbf{Iterative decoupling}\\(stabilize the forward reference path)};

    \node[note, text=green!60!black, below=1.2cm of mid2] (grad_note)
        {\textbf{Gradient truncation}\\(truncate gradients through $X$)};

    \draw[->, dashed, blue!50!black, thick, shorten >=2pt] (picard_note.north) -- (mid1);
    \draw[->, dashed, green!60!black, thick, shorten >=2pt] (grad_note.north) -- (mid2);

    \end{tikzpicture}}
    \caption{Illustration of the instability mechanism in fully coupled FBSDE solvers.}
    \label{fig:vicious_cycle}
\end{figure}

To address this vicious cycle, our method combines two complementary strategies.
First, iterative decoupling allows the BSDE to be evaluated along a forward reference path constructed from previous forward iterates.
This reference path is held fixed within each decoupling step and updated through fictitious-play averaging, reducing the influence of changes in the current forward trajectories on the backward optimization.
Second, gradient truncation is applied to the backward variables entering the forward SDE, so that their numerical values continue to influence the forward dynamics while gradient propagation through the simulated forward path is removed.
Although the forward trajectories may change when recomputed after parameter updates, they are treated as fixed when differentiating the loss.
Together, these strategies preserve the forward--backward coupling in the simulated dynamics while reducing the parameter-dependent feedback during optimization.

\section{Deep truncated FBSDE method}\label{sec:deep_truncated_fbsde}

We first introduce the main ideas and key components of the proposed method, then present the fully discrete numerical scheme and its implementation.

\subsection{Gradient-truncated decoupling with fictitious-play averaging}

To initialize the iterative decoupling, we first specify an initial reference path \(\mathcal X^0\).
For \(p\ge1\), given \(\mathcal X^{p-1}\), we consider the auxiliary decoupling system
\begin{equation}\label{eq:fp_step_p}
\left\{
\begin{aligned}
X_t^p &= x_0
    + \int_0^t b\!\left(s,X_s^p,Y_s^p,Z_s^p\right)\,\mathrm{d}s
    + \int_0^t \sigma\!\left(s,X_s^p,Y_s^p\right)\,\mathrm{d}W_s,\\
Y_t^p &= g(\mathcal X_T^{p-1})
    + \int_t^T f\!\left(s,\mathcal X_s^{p-1},Y_s^p,Z_s^p\right)\,\mathrm{d}s
    - \int_t^T Z_s^p\,\mathrm{d}W_s.
\end{aligned}
\right.
\end{equation}
The backward equation is evaluated along the given reference path \(\mathcal X^{p-1}\), while its solution \((Y^p,Z^p)\) is used in the forward equation to generate the current trajectory \(X^p\).
For \(p\ge1\), the reference path is then updated by the running-average rule
\begin{equation}\label{eq:fp_average_continuous}
    \mathcal X_t^{p}
    :=
    \frac{p}{p+1}\mathcal X_t^{p-1}
    +
    \frac{1}{p+1}X_t^{p},
    \qquad t\in[0,T].
\end{equation}
Equivalently, \(\mathcal X_t^p=\frac{1}{p+1}(\mathcal X_t^0+\sum_{j=1}^p X_t^j)\).
This averaging step is inspired by belief averaging in fictitious play~\cite{MR3608094} and weighted-average variants discussed in deep fictitious play~\cite{MR4250284}.
Although other averaging weights are possible, we adopt the running-average rule in~\eqref{eq:fp_average_continuous} for its compatibility with the convergence analysis.
Compared with a Picard update using only the latest forward iterate, this rule can stabilize the reference path used by the neural networks.

In the numerical implementation, we apply gradient truncation to the backward variables entering the forward dynamics, preserving their numerical values while blocking gradient propagation through the simulated forward trajectories.
This complements the reference-path construction, which fixes the spatial arguments of the backward equation within each decoupling step.
Related freezing strategies also appear in deep fictitious play~\cite{pmlr-v107-han20a,MR4484148}, where information from previous stages is treated as fixed during the current update.

\subsection{Feynman--Kac reconstruction along the reference path}

To construct the numerical scheme, we exploit the nonlinear Feynman--Kac formula~\eqref{eq:nonlinear_feynman_kac} to approximate \(Z^p\) from the neural representation of \(Y^p\).
At the \(p\)-th iterative decoupling step, we introduce the neural reference process
\(
\widehat Y_t^p
=
\phi_t(\mathcal X_t^{p-1};\theta_t^{p-1}).
\)
The corresponding approximation of \(Z^p\) is given by
\(
    Z_t^p
    \approx
    D_x\widehat Y_t^p\,
    \sigma\bigl(t,\mathcal X_t^{p-1},\widehat Y_t^p\bigr),
\)
where \(D_x\widehat Y_t^p\) is computed by automatic differentiation (AD).
Consequently, no separate neural parametrization is introduced for \(Z\).
This reduces the number of neural components and avoids an additional source of approximation error associated with separately parameterizing \(Z\), while retaining the derivative-based relation between the approximations of \(Y\) and \(Z\) along the whole trajectory.
Moreover, no additional normalization layer such as batch normalization is introduced, so the forward evaluation along each path does not rely on batch statistics and avoids the associated computational cost.

\subsection{Pathwise consistent optimization over long horizons}\label{subsec:pathwise_optim}

At the \(p\)-th iterative decoupling step, let \(\widetilde X^p\) denote the current forward trajectory, and let \(\widetilde Y^p(\Theta)\) and \(\widehat Y^p(\Theta)\) denote the simulated backward and neural reference processes, respectively, which depend on the trainable parameters \(\Theta\).
With \(\widetilde X^p\) held fixed during differentiation, we consider the minimization problem
\begin{equation}\label{eq:pathwise_loss}
    \Theta^{p,*}
    \in
    \operatorname*{argmin}_{\Theta}
    \mathbb{E}\bigl[
        |g(\widetilde X_T^p)-\widetilde Y_T^p(\Theta)|^2
        +
        \frac1T\int_0^T
        |\widetilde Y_t^p(\Theta)-\widehat Y_t^p(\Theta)|^2\,\mathrm{d}t
    \bigr].
\end{equation}
Although the backward recursion is evaluated along the reference path, the terminal target is evaluated at the current forward state \(\widetilde X_T^p\).
The optimization therefore remains directed toward the terminal condition of the original FBSDE~\eqref{eq:fully_coupled_FBSDE} rather than only the auxiliary decoupling system~\eqref{eq:fp_step_p}.
The pathwise consistency term penalizes the discrepancy between \(\widetilde Y^p\) and \(\widehat Y^p\) over \([0,T]\), with the factor \(1/T\) maintaining a comparable scale to the terminal loss as the time horizon varies.

Under gradient truncation, \(g(\widetilde X_T^p)\) serves as a detached terminal target, unlike~\eqref{eq:DeepFBSDE_Loss}, where gradients propagate through the parameter-dependent forward trajectory.
The pathwise consistency term also introduces local gradient contributions at intermediate times that bypass the full propagation chain of terminal-only backpropagation and may mitigate gradient vanishing over long time horizons; see Theorem~\ref{thm:gradient_mitigation}.

\subsection{Numerical scheme}

For \(p\ge1\), we discretize the auxiliary system~\eqref{eq:fp_step_p} and the reference-path update~\eqref{eq:fp_average_continuous}.
We denote the discrete processes by \((\widetilde X_n^p,\widetilde Y_n^p,\widetilde Z_n^p)\) and the reference path by \(\widetilde{\mathcal X}_n^p\).
For simplicity, parameter dependence is suppressed in the recursive formulas and made explicit only in the optimization objective.

For the coupled case with \(P>1\), we generate the initial reference path through a preliminary rollout at \(p=0\), starting from \(\widetilde X_0^0=x_0\) and \(\widetilde Y_0^0=\theta_{u_0}^{\,0}\).
This rollout follows the discrete recursion below with the current trajectory \(\widetilde X^0\) in place of the reference path, and we set \(\widetilde{\mathcal X}_n^0:=\widetilde X_n^0\) for \(n=0,\dots,N\).

For \(p\ge1\), given the reference path \(\widetilde{\mathcal X}^{p-1}\), we set
\(\widetilde X_0^p=x_0\) and \(\widetilde Y_0^p=\theta_{u_0}^{\,p-1}\).
The nonlinear Feynman--Kac reconstruction is implemented at the discrete time levels as
\begin{equation}\label{eq:structure_preserving_yz_discrete}
    \widehat Y_n^p=\phi_n(\widetilde{\mathcal{X}}_n^{p-1};\theta_n^{\,p-1}),
    \quad
    \widetilde Z_n^p
    =
    \mathrm{AD}(\widehat Y_n^p)\,
    \sigma(t_n,\widetilde{\mathcal{X}}_n^{p-1},\widehat Y_n^p),
    \quad n=0,\dots,N-1.
\end{equation}
Throughout this paper, the networks \(\phi_n\) are taken to be standard fully connected feedforward neural networks, without additional architectural techniques.

The time-discrete counterpart of~\eqref{eq:fp_step_p} is written in forward recursive form as
\begin{equation}\label{eq:fbsde_Y_discrete}
    \widetilde{Y}_{n+1}^p
    =
    \widetilde{Y}_n^p
    -
    f(t_n,\widetilde{\mathcal{X}}_n^{p-1},\widetilde{Y}_n^p,\widetilde{Z}_n^p)\Delta t_n
    +
    \widetilde{Z}_n^p\Delta W_n,
    \qquad n=0,\dots,N-1,
\end{equation}
and
\begin{equation}\label{eq:fbsde_X_discrete}
\widetilde{X}_{n+1}^p
=
\widetilde{X}_n^p
+
b(t_n,\widetilde{X}_n^p,\mathrm{sg}(\widetilde Y_n^p),\mathrm{sg}(\widetilde Z_n^p))\Delta t_n
+
\sigma(t_n,\widetilde{X}_n^p,\mathrm{sg}(\widetilde Y_n^p))\Delta W_n,
\ n=0,\dots,N-1.
\end{equation}
Here, \(\mathrm{sg}(U(\Theta))\) denotes the stop-gradient operator, which preserves the numerical value of \(U(\Theta)\) while setting its derivative with respect to \(\Theta\) to zero.
This operation affects only parameter differentiation and leaves the simulated dynamics unchanged.
The forward equation~\eqref{eq:fbsde_X_discrete} is discretized by the Euler--Maruyama scheme, while the backward equation~\eqref{eq:fbsde_Y_discrete} is evaluated recursively forward in time.

\begin{figure}[tbp]
    \centering
    \resizebox{\linewidth}{!}{\begin{tikzpicture}[
            >={Stealth[length=3mm, width=2mm]},
            var_base/.style={
                thick,
                rounded corners=4pt,
                minimum width=2.25cm,
                minimum height=1.0cm,
                align=center,
                font=\Large,
                inner sep=2pt
            },
            cyan_node/.style={
                var_base,
                draw=cyan!80!blue,
                fill=cyan!10,
                text=black
            },
            purple_node/.style={
                var_base,
                draw=blue!70!magenta,
                fill=blue!10,
                text=black
            },
            xprev_node/.style={
                var_base,
                draw=blue!90!black,
                fill=blue!10,
                text=black
            },
            green_node/.style={
                var_base,
                draw=green!60!black,
                fill=green!10,
                text=black
            },
            peach_node/.style={
                var_base,
                draw=magenta!80!red,
                fill=magenta!10,
                text=black
            },
            orange_node/.style={
                var_base,
                draw=orange!90!red,
                fill=orange!10,
                text=black
            },
            gray_node/.style={
                var_base,
                draw=black!50,
                fill=black!10,
                text=black
            },
            empty_node/.style={
                font=\Huge,
                text=black,
                inner sep=2pt
            },
            transparent_node/.style={
                var_base,
                draw=none,
                fill=none,
                text=black,
                font=\Huge
            },
            cyan_path/.style={
                draw=cyan!80!blue,
                thick
            },
            purple_path/.style={
                draw=blue!70!magenta,
                thick
            },
            orange_path/.style={
                draw=orange!90!red,
                thick
            },
            green_path/.style={
                draw=green!60!black,
                thick
            },
            main_cyan/.style={
                draw=cyan!80!blue,
                thick
            },
            main_purple/.style={
                draw=blue!70!magenta,
                thick
            }
        ]
        \def\yY{0}
        \def\ySgY{-1.6}
        \def\yYhat{-3.2}
        \def\yZ{-3.2}
        \def\yPhi{-4.8}
        \def\ySgZ{-4.8}
        \def\yXprev{-6.4}
        \def\yX{-8.0}
        \def\ydW{-9.6}
        \def\busshift{0.25}

        \newcommand{\drawstep}[3]{\def\x{#2}\def\n{#3}\node[cyan_node] (Y#1) at (\x,\yY)
                {$\widetilde{Y}_{\n}^p$};
            \node[cyan_node] (sgY#1) at (\x,\ySgY)
                {$\mathrm{sg}(\widetilde{Y}_{\n}^p)$};
            \node[peach_node] (Yhat#1) at (\x,\yYhat)
                {$\widehat{Y}_{\n}^p$};
            \node[green_node] (Phi#1) at (\x,\yPhi)
                {$\phi_{\n}$};
            \node[xprev_node] (Xprev#1) at (\x,\yXprev)
                {$\widetilde{\mathcal{X}}_{\n}^{p-1}$};
            \node[purple_node] (X#1) at (\x,\yX)
                {$\widetilde{X}_{\n}^p$};
            \node[orange_node] (Z#1) at (\x+3.0,\yZ)
                {$\widetilde{Z}_{\n}^p$};
            \node[orange_node] (sgZ#1) at (\x+3.0,\ySgZ)
                {$\mathrm{sg}(\widetilde{Z}_{\n}^p)$};
            \draw[->,cyan_path]
                (Y#1) -- (sgY#1);
            \draw[->,orange_path]
                (Z#1) -- (sgZ#1);
            \draw[->,green_path]
                (Xprev#1) -- (Phi#1);
            \draw[->,green_path]
                (Phi#1) -- (Yhat#1);
            \draw[->,orange_path]
                (Yhat#1.east) -- (Z#1.west);
            \draw[->,orange_path,rounded corners=4pt]
                ($(Xprev#1.east)+(0,0.15)$)
                -- (\x+1.45,\yXprev+0.15)
                |- ($(Z#1.west)+(0,-0.20)$);
        }

        \newcommand{\drawstepdots}[2]{\def\x{#2}\node[transparent_node] (Y#1) at (\x,\yY)
                {$\cdots$};
            \node[transparent_node] (sgY#1) at (\x,\ySgY)
                {$\cdots$};
            \node[transparent_node] (Yhat#1) at (\x,\yYhat)
                {$\cdots$};
            \node[transparent_node] (Phi#1) at (\x,\yPhi)
                {$\cdots$};
            \node[transparent_node] (Xprev#1) at (\x,\yXprev)
                {$\cdots$};
            \node[transparent_node] (X#1) at (\x,\yX)
                {$\cdots$};
            \node[transparent_node] (Z#1) at (\x+3.0,\yZ)
                {$\cdots$};
            \node[transparent_node] (sgZ#1) at (\x+3.0,\ySgZ)
                {$\cdots$};
            \draw[->,cyan_path]
                (Y#1) -- (sgY#1);
            \draw[->,orange_path]
                (Z#1) -- (sgZ#1);
            \draw[->,green_path]
                (Xprev#1) -- (Phi#1);
            \draw[->,green_path]
                (Phi#1) -- (Yhat#1);
            \draw[->,orange_path]
                (Yhat#1.east) -- (Z#1.west);
            \draw[->,orange_path,rounded corners=4pt]
                ($(Xprev#1.east)+(0,0.15)$)
                -- (\x+1.45,\yXprev+0.15)
                |- ($(Z#1.west)+(0,-0.20)$);
        }

        \newcommand{\drawtrans}[5]{\def\basex{#4}\def\vxSgYtoX{\basex+4.1+\busshift}\def\vxZ{\basex+4.3+\busshift}\def\vxXprevtoY{\basex+4.5+\busshift}\def\vxDWtoY{\basex+4.7+\busshift}\def\vxDWtoX{\basex+4.90+\busshift}\node[#5] (dW#1) at (\basex+6.7,\ydW)
                {#3};
            \draw[->,main_cyan]
                ($(Y#1.east)+(0,0.15)$)
                -- ($(Y#2.west)+(0,0.15)$);
            \draw[->,main_purple]
                (X#1.east) -- (X#2.west);
            \draw[->,purple_path,rounded corners=4pt]
                (sgY#1.east)
                -- (\vxSgYtoX,\ySgY)
                |- ($(X#2.west)+(0,0.15)$);
            \draw[->,cyan_path,rounded corners=4pt]
                ($(Z#1.east)+(0,0.15)$)
                -- (\vxZ,\yZ+0.15)
                |- (Y#2.west);
            \draw[->,purple_path,rounded corners=4pt]
                (sgZ#1.east)
                -- (\vxZ,\ySgZ)
                |- ($(X#2.west)+(0,0.30)$);
            \draw[->,cyan_path,rounded corners=4pt]
                ($(Xprev#1.east)+(0,-0.15)$)
                -- (\vxXprevtoY,\yXprev-0.15)
                |- ($(Y#2.west)+(0,-0.15)$);
            \draw[->,cyan_path,rounded corners=4pt]
                ($(dW#1.west)+(0,-0.15)$)
                -- (\vxDWtoY,\ydW-0.15)
                |- ($(Y#2.west)+(0,-0.30)$);
            \draw[->,purple_path,rounded corners=4pt]
                ($(dW#1.west)+(0,0.05)$)
                -- (\vxDWtoX,\ydW+0.05)
                |- ($(X#2.west)+(0,-0.15)$);
        }

        \drawstep{0}{0}{0}
        \drawstep{1}{6.7}{1}
        \drawstepdots{dots}{13.4}
        \drawstep{Nm1}{20.1}{N-1}

        \node[cyan_node] (YN) at (26.8,\yY)
            {$\widetilde{Y}_N^p$};
        \node[purple_node] (XN) at (26.8,\yX)
            {$\widetilde{X}_N^p$};

        \drawtrans{0}{1}{$\Delta W_0$}{0}{gray_node}
        \drawtrans{1}{dots}{$\cdots$}{6.7}{transparent_node}
        \drawtrans{dots}{Nm1}{$\Delta W_{N-2}$}{13.4}{gray_node}
        \drawtrans{Nm1}{N}{$\Delta W_{N-1}$}{20.1}{gray_node}
        \end{tikzpicture}}
    \caption[Architecture of the deep truncated FBSDE method]{Architecture of the deep truncated FBSDE method at the \(p\)-th iterative decoupling step.
    }
    \label{Fig:Architecture}
\end{figure}
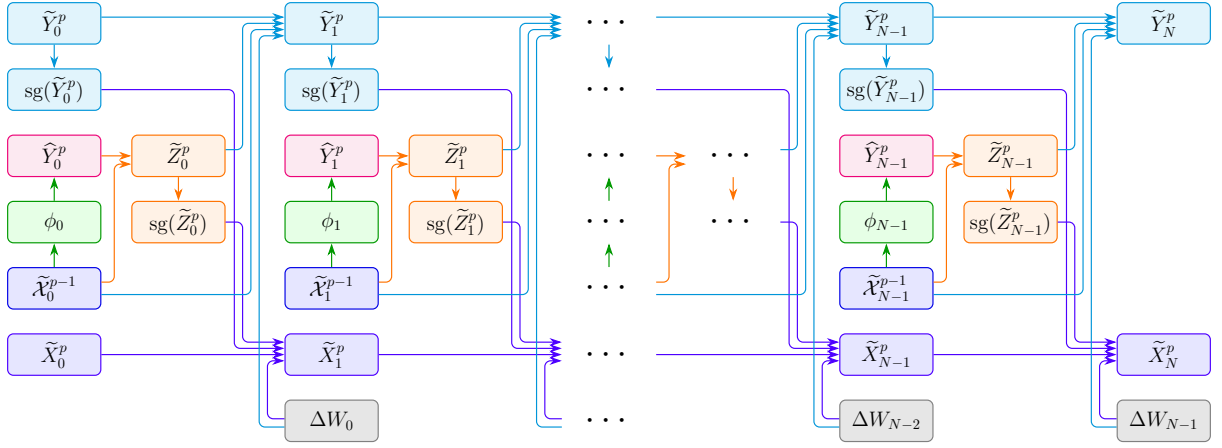

Once \(\widetilde X^p\) has been computed, the reference path is updated only when the iterative decoupling is continued to the next step.
In that case, we use the discrete fictitious-play average
\begin{equation}\label{eq:fictitious_play_X}
    \widetilde{\mathcal X}_n^p
    :=
    \frac{p}{p+1}\widetilde{\mathcal X}_n^{p-1}
    +
    \frac{1}{p+1}\widetilde X_n^p,
    \qquad n=0,\dots,N.
\end{equation}
This is the discrete counterpart of the continuous averaging rule in~\eqref{eq:fp_average_continuous}.
Figure~\ref{Fig:Architecture} illustrates the computational architecture of the discrete scheme at a generic iterative decoupling step.

Let \(\Theta:=\{\theta_{u_0},\theta_0,\ldots,\theta_{N-1}\}\) denote the full set of trainable parameters.
Define \(\Theta_0:=\{\theta_{u_0}\}\) and
\(\Theta_n:=\{\theta_{u_0},\theta_0,\ldots,\theta_{n-1}\}\) for \(n=1,\dots,N-1\).
Thus, \(\widetilde Y_n\) depends on the parameters accumulated up to time level \(t_n\), while
\(\widehat Y_n\) depends only on the network parameter \(\theta_n\) at the same time level.
The continuous optimization problem~\eqref{eq:pathwise_loss} is discretized at the \(p\)-th iterative decoupling step as
\begin{equation}
\label{eq:pathwise_loss_discrete}
\Theta^{p,*}
\in
\operatorname*{argmin}_{\Theta}
\mathbb{E}\bigl[
|
g(\widetilde X_N^p)
-
\widetilde Y_N^p(\Theta)
|^2
+
\frac1T
\sum_{n=0}^{N-1}
|
\widetilde Y_n^p(\Theta_n)
-
\widehat Y_n^p(\theta_n)
|^2
\Delta t_n
\bigr].
\end{equation}

\begin{algorithm}[tbp]
\caption{Deep truncated FBSDE method}
\label{Algo:Algorithm-12}
\small
\begin{algorithmic}[1]
    \Input Initial parameters $\Theta^0$; batch size $B$; maximum number $K$ of outer iterations; maximum number $P$ of iterative decoupling steps; time partition $\pi:0=t_0<t_1<\cdots<t_N=T$; initial point $x_0$; terminal condition $g(\cdot)$; tolerance $\delta>0$.
    \Output Optimized parameters $\Theta^K=\{\theta_{u_0}^K,\theta_0^K,\dots,\theta_{N-1}^K\}$.

    \For{$k=1$ to $K$}
        \State Sample a batch of Brownian increments \(\Delta W^{k,i}\); the sample index \(i\) is suppressed below.

        \State \textbf{Initialization:} set \(p\gets 0\),
        \(\Theta^{k,0}\gets \Theta^{k-1}\),
        \(\widetilde X_0^{k,0}\gets x_0\),
        \(\widetilde Y_0^{k,0}\gets \theta_{u_0}^{k,0}\).
        \For{$n=0$ to $N-1$}

            \State $\widehat Y_n^{k,0} \gets \phi_n(\widetilde X_n^{k,0};\theta_n^{\,k,0})$;
            \State $\widetilde Z_n^{k,0} \gets \mathrm{AD}(\widehat Y_n^{k,0})\,\sigma(t_n,\widetilde X_n^{k,0},\widehat Y_n^{k,0})$;
            \State $\widetilde Y_{n+1}^{k,0} \gets \widetilde Y_n^{k,0}
            -f(t_n,\widetilde X_n^{k,0},\widetilde Y_n^{k,0},\widetilde Z_n^{k,0}) \Delta t_n
            +\widetilde Z_n^{k,0} \Delta W_n^{k}$;
            \State $\widetilde X_{n+1}^{k,0}\gets \widetilde X_n^{k,0}
            +b(t_n,\widetilde X_n^{k,0},\mathrm{sg}(\widetilde Y_n^{k,0}),\mathrm{sg}(\widetilde Z_n^{k,0}))\Delta t_n
             +\sigma(t_n,\widetilde X_n^{k,0},\mathrm{sg}(\widetilde Y_n^{k,0}))\Delta W_n^{k};$
        \EndFor

        \State \textbf{Iterative decoupling:} set $p\gets 1$ and $\widetilde{\mathcal{X}}_n^{k,0}\gets \widetilde X_n^{k,0}$ for $n=0,\dots,N$;
        \While{$p\le P$}

            \State Initialize $\widetilde X_0^{k,p}\gets x_0,\ \widetilde Y_0^{k,p}\gets \theta_{u_0}^{\,k,p-1}$;
            \For{$n=0$ to $N-1$}

                \State $\widehat Y_n^{k,p} \gets \phi_n(\widetilde{\mathcal{X}}_n^{k,p-1};\theta_n^{\,k,p-1})$;
                \State $\widetilde Z_n^{k,p} \gets \mathrm{AD}(\widehat Y_n^{k,p})\,\sigma(t_n,\widetilde{\mathcal{X}}_n^{k,p-1},\widehat Y_n^{k,p})$;
                \State $\widetilde Y_{n+1}^{k,p} \gets
                \begin{aligned}[t]
                    & \widetilde Y_n^{k,p}
                    -f(t_n,\widetilde{\mathcal{X}}_n^{k,p-1},\widetilde Y_n^{k,p},\widetilde Z_n^{k,p}) \Delta t_n + \widetilde Z_n^{k,p} \Delta W_n^{k};
                \end{aligned}$
                \State $\widetilde X_{n+1}^{k,p}\gets \widetilde X_n^{k,p}
                +b(t_n,\widetilde X_n^{k,p},\mathrm{sg}(\widetilde Y_n^{k,p}),\mathrm{sg}(\widetilde Z_n^{k,p}))\Delta t_n
                 +\sigma(t_n,\widetilde X_n^{k,p},\mathrm{sg}(\widetilde Y_n^{k,p}))\Delta W_n^{k};$
            \EndFor

            \If{$p>1$}
                \State Compute \(\Delta_X^{k,p}\) and \(\Delta_Y^{k,p}\) by \eqref{eq:delta_discrete}; \textbf{break} if \(\max\{\Delta_X^{k,p},\Delta_Y^{k,p}\}<\delta\).
            \EndIf

            \State Set
            \(\displaystyle
            \widetilde{\mathcal X}^{k,p}
            \gets
            \frac{p}{p+1}\widetilde{\mathcal X}^{k,p-1}
            +
            \frac{1}{p+1}\widetilde X^{k,p}
            \)
            when \(p<P\).

            \State Compute the empirical loss $\widetilde L^{k,p}$ by \eqref{eq:empirical_pathwise_loss}.
            \State Update \(\Theta^{k,p-1}\) to \(\Theta^{k,p}\) by applying a stochastic gradient descent (SGD)-type optimizer to \(\widetilde L^{k,p}\).
            \State \(p\gets p+1\).
        \EndWhile

        \State $\Theta^k\gets \Theta^{k,p-1}$;
    \EndFor

    \State \textbf{return} $\Theta^K$;
\end{algorithmic}
\end{algorithm}

In practice, the expectation in~\eqref{eq:pathwise_loss_discrete} is approximated by Monte Carlo sampling over simulated Brownian trajectories.
At each outer optimization iteration \(k\), a batch of Brownian increments \(\{\Delta W_n^{k,i}\}_{n=0,\dots,N-1;\,i=1,\dots,B}\) is sampled and kept fixed throughout the subsequent iterative decoupling steps.
For the \(i\)-th trajectory, the corresponding discrete variables are denoted by
\(\widetilde X_n^{k,p,i}\), \(\widetilde Y_n^{k,p,i}\), \(\widehat Y_n^{k,p,i}\), and \(\widetilde Z_n^{k,p,i}\).

Since the same Brownian batch is used throughout the iterative decoupling steps, consecutive trajectories can be compared samplewise.
For \(p>1\), we define the normalized pathwise deviations by
\begin{equation}\label{eq:delta_discrete}
    \Delta_X^{k,p}
    :=
    \max_{1\le i\le B}
    \max_{0\le n\le N}
    \frac{|\widetilde X_n^{k,p,i}-\widetilde X_n^{k,p-1,i}|}{\sqrt d},
    \qquad
    \Delta_Y^{k,p}
    :=
    \max_{1\le i\le B}
    \max_{0\le n\le N}
    \frac{|\widetilde Y_n^{k,p,i}-\widetilde Y_n^{k,p-1,i}|}{\sqrt m}.
\end{equation}
These quantities measure the largest normalized variations of the forward and backward trajectories between two consecutive decoupling steps.
The inner iteration is terminated when both \(\Delta_X^{k,p}<\delta\) and \(\Delta_Y^{k,p}<\delta\), avoiding premature stopping based on only one component and unnecessary updates after both have stabilized, which may help limit overfitting to the fixed batch.

Whenever \(p=1\), or \(p>1\) and the stopping criterion is not satisfied, the current trajectories are used for the parameter update.
For the \(i\)-th trajectory, the samplewise loss is defined by
\begin{equation*}
\widetilde L^{k,p,i}(\Theta^{k,p-1})
:=
\bigl|g(\widetilde X_N^{k,p,i})-\widetilde Y_N^{k,p,i}(\Theta^{k,p-1})\bigr|^2
+\frac1T\sum_{n=0}^{N-1}
\bigl|\widetilde Y_n^{k,p,i}(\Theta_n^{k,p-1})
-\widehat Y_n^{k,p,i}(\theta_n^{\,k,p-1})\bigr|^2\Delta t_n.
\end{equation*}
The empirical loss used for the parameter update is
\begin{equation}\label{eq:empirical_pathwise_loss}
    \widetilde L^{k,p}(\Theta^{k,p-1})
    :=
    \frac{1}{B}\sum_{i=1}^B
    \widetilde L^{k,p,i}(\Theta^{k,p-1}).
\end{equation}

Algorithm~\ref{Algo:Algorithm-12} summarizes the implementation of our method for \(P>1\).
For notational simplicity, the sample index \(i\) is suppressed in the discrete recursions.
The same Brownian batch is reused throughout the inner decoupling loop, and one parameter update is performed at each completed decoupling step unless the stopping criterion is satisfied.
The choice of \(P\) reflects a trade-off between the decoupling error and computational cost.
A larger \(P\) allows further refinement of the iterative decoupling procedure, but requires additional forward--backward updates.

\begin{remark}[Decoupled case]
For a decoupled FBSDE, the forward dynamics are independent of the backward variables, so we set \(P=1\) and omit the initialization step \(p=0\).
The method then performs a single FBSDE simulation and parameter update per Brownian batch, without gradient truncation or fictitious-play averaging.
\end{remark}

\section{Gradient propagation and gradient shortcuts}
\label{sec:analysis_gradient}

In this section, we analyze the gradient structure induced by the proposed method and the gradient shortcuts generated by the pathwise consistency term.

\subsection{Assumptions and main theorem}

Fix a Brownian sample and a decoupling step \(p\ge1\), suppress these indices, and consider derivatives with respect to \(\theta_i\), \(0\le i\le N-2\).
Derivatives are taken under the stop-gradient convention.
Since the reference path is frozen during differentiation, the identities below hold pathwise and hence also after empirical averaging.
For derivatives of vector- or matrix-valued maps, the notation \(\|\cdot\|_2\) is understood as the induced norm associated with the norms introduced in Section~\ref{subsec:notation}.

\begin{assumption}[Coefficient regularity]
\label{assum:coeff_reg}
For every \((t,x)\in[0,T]\times\mathbb{R}^d\), it holds that
\(f(t,x,\cdot,\cdot)\in C^1(\mathbb{R}^m\times\mathbb{R}^{m\times d};\mathbb{R}^m)\) and
\(\sigma(t,x,\cdot)\in C^1(\mathbb{R}^m;\mathbb{R}^{d\times d})\).
Moreover, there exist constants \(C_{f,z},C_{\sigma,y},C_\sigma>0\) such that
\[
    \sup_{t,x,y,z}\|D_z f(t,x,y,z)\|_2\le C_{f,z},\qquad
    \sup_{t,x,y}\|D_y\sigma(t,x,y)\|_2\le C_{\sigma,y},\qquad
    \sup_{t,x,y}\|\sigma(t,x,y)\|_2\le C_\sigma.
\]
\end{assumption}

\begin{assumption}[Neural network differentiability]
\label{assum:nn_diff}
For each \(n=0,\ldots,N-1\), the mapping
\((x,\theta_n)\mapsto\phi_n(x;\theta_n)\) is differentiable in \((x,\theta_n)\), and the mixed derivative
\(\partial_{\theta_n}\partial_x\phi_n\) exists.
\end{assumption}

In all experiments, the subnetworks \(\phi_n\) use \(\tanh\) activations and thus satisfy Assumption~\ref{assum:nn_diff}.
For \(n=0,\ldots,N-1\), define the propagation matrix
\(
H_n:=I_m-D_yf(t_n,\widetilde{\mathcal X}_n,\widetilde Y_n,\widetilde Z_n)\Delta t_n.
\)
The total loss \(\mathcal L\) consists of the terminal loss \(\mathcal L_{T}\) and the pathwise consistency loss \(\mathcal L_{C}\), where \(\mathcal L_{T}:=|\widetilde Y_N-g(\widetilde X_N)|^2\) and \(\mathcal L_{C}:=\frac1T\sum_{n=0}^{N-1}|\widetilde Y_n-\widehat Y_n|^2\Delta t_n\).
The main result is as follows.

\begin{theorem}[Gradient decomposition with local shortcuts]
\label{thm:gradient_mitigation}
Under Assumptions~\ref{assum:coeff_reg}--\ref{assum:nn_diff}, for every \(i\in\{0,\ldots,N-2\}\), the gradient of the total loss admits the decomposition
\begin{equation}
\label{eq:total_grad_explicit}
    \nabla_{\theta_i}\mathcal L
    =
    \frac{2\Delta t_i}{T}
    (\partial_{\theta_i}\phi_i(\widetilde{\mathcal{X}}_i;\theta_i))^\top
    (\widehat Y_i-\widetilde Y_i)
    +
    \frac{2\Delta t_{i+1}}{T}
    (\partial_{\theta_i}\widetilde Y_{i+1})^\top
    (\widetilde Y_{i+1}-\widehat Y_{i+1})
    +R_{i+1},
\end{equation}
where
\begin{equation}\label{eq:R_nplus1_def}
R_{i+1}:=2(\partial_{\theta_i}\widetilde Y_{i+1})^\top
\Bigl[
\sum_{j=i+2}^{N-1}\frac{\Delta t_j}{T}
(\prod_{n=i+1}^{j-1}H_n)^\top
(\widetilde Y_j-\widehat Y_j)
+(\prod_{n=i+1}^{N-1}H_n)^\top
(\widetilde Y_N-g(\widetilde X_N))
\Bigr].
\end{equation}

\end{theorem}

The first two terms in~\eqref{eq:total_grad_explicit} are local contributions generated by the pathwise consistency loss and do not involve the full propagation chain.
The remainder \(R_{i+1}\) collects the contributions propagated from later time levels and the terminal loss.

\subsection{Gradient propagation and attenuation}

We first characterize the gradient propagation mechanism associated with the terminal loss.
The following two assumptions are required only for the attenuation estimate in Lemma~\ref{lem:terminal_gradient_attenuation}.

\begin{assumption}[Neural network boundedness]
\label{assum:nn_bound}
There exists a constant \(C_\phi>0\) such that
\[
    \sup_{0\le n\le N-1}\sup_{x,\theta_n}
    \left(
        \|\partial_x\phi_n(x;\theta_n)\|_2
        \vee
        \|\partial_{\theta_n}\phi_n(x;\theta_n)\|_2
        \vee
        \|\partial_{\theta_n}\partial_x\phi_n(x;\theta_n)\|_2
    \right)
    \le C_\phi.
\]
\end{assumption}

\begin{assumption}[Propagation damping]
\label{assum:propagation_damping}
There exists a constant \(\eta>0\) such that, for every \(i\in\{0,\ldots,N-2\}\),
\[
\bigl\|\prod_{n=i+1}^{N-1}H_n\bigr\|_2
\le e^{-\eta(T-t_{i+1})}.
\]
Here \(\prod_{n=a}^{b}H_n:=H_bH_{b-1}\cdots H_a\) for \(a\le b\), and empty products are understood as identity matrices.
\end{assumption}

\begin{lemma}[Attenuation of the terminal-loss gradient]
\label{lem:terminal_gradient_attenuation}
Under Assumptions~\ref{assum:coeff_reg} and~\ref{assum:nn_diff}, for every
\(i\in\{0,\ldots,N-2\}\), the terminal-loss gradient satisfies
\begin{equation}
\label{eq:grad_LT_unrolled}
    \nabla_{\theta_i}\mathcal L_{T}
    =
    2(\partial_{\theta_i}\widetilde Y_{i+1})^\top
    \bigl(\prod_{n=i+1}^{N-1}H_n\bigr)^\top
    (\widetilde Y_N-g(\widetilde X_N)).
\end{equation}
If, in addition, Assumptions~\ref{assum:nn_bound} and~\ref{assum:propagation_damping} hold, the terminal-loss gradient satisfies
\begin{equation}
\label{eq:terminal_gradient_bound}
    |\nabla_{\theta_i}\mathcal L_{T}|
    \le
    2C_{\phi,i}|\widetilde Y_N-g(\widetilde X_N)|e^{-\eta(T-t_{i+1})},
\end{equation}
where
\(
    C_{\phi,i}:=
    (C_\phi C_\sigma+C_{\sigma,y}C_\phi^2)
    (C_{f,z}\Delta t_i+|\Delta W_i|).
\)
\end{lemma}

\begin{proof}
Fix \(i\in\{0,\ldots,N-2\}\).
By the stop-gradient treatment of the forward variables, we have
\(\partial_{\theta_i}\widetilde X_N=0\).
Hence
\(
    \nabla_{\theta_i}\mathcal L_{T}
    =
    2(\partial_{\theta_i}\widetilde Y_N)^\top
    (\widetilde Y_N-g(\widetilde X_N)).
\)

We next derive \(\partial_{\theta_i}\widetilde Y_N\).
For \(n>i\), the parameter \(\theta_i\) does not enter the subnetwork \(\phi_n\), and the reference path \(\widetilde{\mathcal{X}}_n\) is frozen when differentiating with respect to \(\theta_i\).
Therefore
\(
\partial_{\theta_i}\widetilde Z_n=0
\),
and differentiation of the backward recursion~\eqref{eq:fbsde_Y_discrete} gives
\[
    \partial_{\theta_i}\widetilde Y_{n+1}
    =
    \bigl(I_m-D_yf_n\Delta t_n\bigr)\partial_{\theta_i}\widetilde Y_n
    =
    H_n\partial_{\theta_i}\widetilde Y_n,
    \qquad n=i+1,\ldots,N-1,
\]
where
\(D_yf_n:=D_yf(t_n,\widetilde{\mathcal X}_n,\widetilde Y_n,\widetilde Z_n)\).
Consequently,
\(
    \partial_{\theta_i}\widetilde Y_N
    =
    (\prod_{n=i+1}^{N-1}H_n)
    \partial_{\theta_i}\widetilde Y_{i+1},
\)
which yields~\eqref{eq:grad_LT_unrolled}.

It remains to estimate \(\partial_{\theta_i}\widetilde Y_{i+1}\).
Since \(\theta_i\) first enters the recursion at time \(t_i\),
\(
\partial_{\theta_i}\widetilde Y_i=0
\),
and hence
\[
    \partial_{\theta_i}\widetilde Y_{i+1}
    =
    -D_zf_i\,\partial_{\theta_i}\widetilde Z_i\,\Delta t_i
    +
    (\partial_{\theta_i}\widetilde Z_i)\Delta W_i,
\]
where
\(
D_zf_i
=
D_zf(t_i,\widetilde{\mathcal X}_i,\widetilde Y_i,\widetilde Z_i)
\).
From~\eqref{eq:structure_preserving_yz_discrete}, we have \(\widetilde Z_i=\partial_x\phi_i(\widetilde{\mathcal X}_i;\theta_i)\,\sigma(t_i,\widetilde{\mathcal X}_i,\widehat Y_i)\) and \(\widehat Y_i=\phi_i(\widetilde{\mathcal X}_i;\theta_i)\),
so that
\[
    \partial_{\theta_i}\widetilde Z_i
    =
    \partial_{\theta_i}\partial_x\phi_i(\widetilde{\mathcal X}_i;\theta_i)
    \sigma(t_i,\widetilde{\mathcal X}_i,\widehat Y_i)
    +
    \partial_x\phi_i(\widetilde{\mathcal X}_i;\theta_i)
    D_y\sigma(t_i,\widetilde{\mathcal X}_i,\widehat Y_i)
    \bigl[\partial_{\theta_i}\phi_i(\widetilde{\mathcal X}_i;\theta_i)\bigr].
\]
Assumptions~\ref{assum:coeff_reg} and~\ref{assum:nn_bound} therefore imply
\(
    \|\partial_{\theta_i}\widetilde Z_i\|_2
    \le
    C_\phi C_\sigma+C_{\sigma,y}C_\phi^2.
\)
Using this bound together with \(\|D_zf_i\|_2\le C_{f,z}\), we obtain
\[
    \|\partial_{\theta_i}\widetilde Y_{i+1}\|_2
    \le
    (C_\phi C_\sigma+C_{\sigma,y}C_\phi^2)
    (C_{f,z}\Delta t_i+|\Delta W_i|)
    =
    C_{\phi,i}.
\]
Taking norms in~\eqref{eq:grad_LT_unrolled}, using submultiplicativity, and combining the above estimate with Assumption~\ref{assum:propagation_damping} yield~\eqref{eq:terminal_gradient_bound}.
\end{proof}

Lemma~\ref{lem:terminal_gradient_attenuation} shows that the terminal-loss gradient for an early-time parameter is propagated through \(\prod_{n=i+1}^{N-1}H_n\), whose norm is exponentially bounded with respect to the effective horizon \(T-t_{i+1}\) under Assumption~\ref{assum:propagation_damping}.
Thus, terminal-loss gradients may become weak over long horizons, motivating the local contributions introduced by the pathwise consistency loss.

\subsection{Local gradient shortcuts from pathwise consistency}

We now show that the pathwise consistency loss produces additional local contributions that do not involve the full propagation chain.

\begin{lemma}[Gradient shortcuts via pathwise consistency loss]
\label{lem:gradient_shortcut}
Under Assumptions~\ref{assum:coeff_reg} and~\ref{assum:nn_diff}, for every \(i\in\{0,\ldots,N-2\}\), the pathwise consistency loss satisfies
\begin{equation}\label{eq:grad_LC_explicit_simple}
\begin{aligned}
    \nabla_{\theta_i}\mathcal L_{C}
    &=
    \frac{2\Delta t_i}{T}
    (\partial_{\theta_i}\phi_i(\widetilde{\mathcal{X}}_i;\theta_i))^\top
    (\widehat Y_i-\widetilde Y_i)
    +
    \frac{2\Delta t_{i+1}}{T}
    (\partial_{\theta_i}\widetilde Y_{i+1})^\top
    (\widetilde Y_{i+1}-\widehat Y_{i+1})                               \\
    &\quad
    +
    \sum_{j=i+2}^{N-1}
    \frac{2\Delta t_j}{T}
    (\partial_{\theta_i}\widetilde Y_{i+1})^\top
    \bigl(\prod_{n=i+1}^{j-1}H_n\bigr)^\top
    (\widetilde Y_j-\widehat Y_j).
\end{aligned}
\end{equation}
\end{lemma}

\begin{proof}
Applying the chain rule to \(\mathcal L_{C}\) with respect to \(\theta_i\) gives
\(
    \nabla_{\theta_i}\mathcal L_{C}
    =
    \sum_{j=0}^{N-1}
    \frac{2\Delta t_j}{T}
    (
    \partial_{\theta_i}\widetilde Y_j
    -
    \partial_{\theta_i}\widehat Y_j
    )^\top
    (\widetilde Y_j-\widehat Y_j).
\)

For \(j<i\), neither \(\widetilde Y_j\) nor \(\widehat Y_j\) depends on \(\theta_i\), so their derivatives with respect to \(\theta_i\) vanish.
At \(j=i\), we have
\(\partial_{\theta_i}\widetilde Y_i=0\).
Since the reference path is frozen during differentiation,
\(\partial_{\theta_i}\widehat Y_i
=
\partial_{\theta_i}\phi_i(\widetilde{\mathcal X}_i;\theta_i)\),
which gives the first term in~\eqref{eq:grad_LC_explicit_simple}.

For \(j>i\), the parameter \(\theta_i\) does not enter the subnetwork \(\phi_j\), and the reference path is frozen during differentiation.
Hence
\(\partial_{\theta_i}\widehat Y_j=0\).
In particular, \(j=i+1\) gives the second term in~\eqref{eq:grad_LC_explicit_simple}.

For \(j\ge i+2\), the sensitivity recursion established in the proof of Lemma~\ref{lem:terminal_gradient_attenuation} gives \(\partial_{\theta_i}\widetilde Y_j=(\prod_{n=i+1}^{j-1}H_n)\partial_{\theta_i}\widetilde Y_{i+1}\).
Substituting this relation into the remaining terms gives the last term in~\eqref{eq:grad_LC_explicit_simple}, which completes the proof.
\end{proof}

The first two terms in~\eqref{eq:grad_LC_explicit_simple} are local contributions at \(t_i\) and \(t_{i+1}\), respectively, and neither contains the full propagation chain \(\prod_{n=i+1}^{N-1}H_n\).
Thus, the pathwise consistency loss introduces local gradient contributions that complement the terminal-loss gradient.

Theorem~\ref{thm:gradient_mitigation} now follows directly from Lemmas~\ref{lem:terminal_gradient_attenuation} and~\ref{lem:gradient_shortcut}.

\begin{proof}[Proof of Theorem~\ref{thm:gradient_mitigation}]
Since
\(
\nabla_{\theta_i}\mathcal L
=
\nabla_{\theta_i}\mathcal L_{T}
+
\nabla_{\theta_i}\mathcal L_{C},
\)
combining~\eqref{eq:grad_LT_unrolled} and~\eqref{eq:grad_LC_explicit_simple} gives the decomposition~\eqref{eq:total_grad_explicit}.
The contributions at \(t_i\) and \(t_{i+1}\) form the first two terms, while the remaining propagated consistency terms and the terminal-loss contribution are collected in \(R_{i+1}\) as defined in~\eqref{eq:R_nplus1_def}.
Since the first two terms do not contain the full propagation chain \(\prod_{n=i+1}^{N-1}H_n\), they provide the local gradient contributions stated in the theorem.
\end{proof}

\section{Error estimates and convergence}
\label{sec:error_analysis}

In this section, we derive a residual-based error estimate and establish conditional convergence for the fully discrete FBSDE scheme~\eqref{eq:structure_preserving_yz_discrete}--\eqref{eq:fbsde_X_discrete}.
The neural-network approximations and initial backward value are treated as given deterministic quantities, and the resulting discrete process is compared with the original coupled FBSDE solution~\eqref{eq:fully_coupled_FBSDE}.
Only the terminal loss enters the analysis, with exact expectations and without considering Monte Carlo sampling errors or convergence of the stochastic optimization procedure.
The convergence result requires weak coupling, while strongly coupled problems are investigated numerically in Section~\ref{sec:numerical_experiments}.

\subsection{Notations, assumptions, and main results}
\label{subsec:error_assumptions_main_results}

Let \(|\pi|:=\max_{0\le n<N}\Delta t_n\) and \(\tau(t):=t_n\) for \(t\in[t_n,t_{n+1})\), with \(\tau(T)=T\).
For \(p\ge0\), define the piecewise-constant interpolations
\(
(X_t^{\pi,p},Y_t^{\pi,p},Z_t^{\pi,p})
:=
(\widetilde X_n^p,\widetilde Y_n^p,\widetilde Z_n^p)
\)
for \( t\in[t_n,t_{n+1})\), with \(X_T^{\pi,p}:=\widetilde X_N^p\) and \(Y_T^{\pi,p}:=\widetilde Y_N^p\).
Similarly, for \(p\ge1\), set \(\mathcal X_t^{\pi,p-1}:=\widetilde{\mathcal X}_n^{p-1}\) on \([t_n,t_{n+1})\), with \(\mathcal X_T^{\pi,p-1}:=\widetilde{\mathcal X}_N^{p-1}\).

For an adapted triple \((U,V,W)\), define
\begin{equation}
\label{eq:error_norm}
    \|(U,V,W)\|_{\mathcal E}^2
    :=
    \sup_{0\le t\le T}\mathbb{E}\bigl[|U_t|^2+|V_t|^2\bigr]
    +
    \mathbb{E}\int_0^T \|W_t\|^2\,\mathrm{d}t.
\end{equation}
For \(p\ge1\), define the decoupling error
\begin{equation}
\label{eq:error_decoupling_error}
    \mathcal E_{dec}^p
    :=
    \|(X^p-X,Y^p-Y,Z^p-Z)\|_{\mathcal E}^2
\end{equation}
and the numerical error with respect to the \(p\)-th auxiliary decoupling system~\eqref{eq:fp_step_p}
\begin{equation}
\label{eq:error_reference_numerical_error}
    \mathcal D_p^\pi
    :=
    \|(X^{\pi,p}-X^p,Y^{\pi,p}-Y^p,Z^{\pi,p}-Z^p)\|_{\mathcal E}^2.
\end{equation}
Its forward and backward components are
\[
\mathfrak X_p^\pi
:=
\sup_{0\le t\le T}\mathbb{E}|X_t^{\pi,p}-X_t^p|^2,
\qquad
\mathfrak B_p^\pi
:=
\sup_{0\le t\le T}\mathbb{E}|Y_t^{\pi,p}-Y_t^p|^2
+
\mathbb{E}\int_0^T\|Z_t^{\pi,p}-Z_t^p\|^2\,\mathrm{d}t.
\]
For \(p=0\), set
\(
\mathfrak X_0^\pi
:=
\sup_{0\le t\le T}
\mathbb{E} |\mathcal X_t^{\pi,0}-\mathcal X_t^0|^2
=
\sup_{0\le t\le T}
\mathbb{E} |X_t^{\pi,0}-\mathcal X_t^0 |^2,
\)
where \(\mathcal X^{\pi,0}=X^{\pi,0}\) by construction.

We next introduce the terminal loss and the residual quantities used in the error estimate.
For \(p\ge1\), define the terminal loss
\begin{equation}
\label{eq:error_training_loss}
    \mathcal L_{T}^{\pi,p}
    :=
    \mathbb{E}\bigl[|g(X_T^{\pi,p})-Y_T^{\pi,p}|^2\bigr],
\end{equation}
the terminal reference mismatch
\begin{equation*}
    \mathcal R_{\rm term}^{\pi,p}
    :=
    \mathbb{E}\bigl[
        |g(X_T^{\pi,p})-g(\mathcal X_T^{\pi,p-1})|^2
    \bigr],
\end{equation*}
and the reference-path discrepancy
\begin{equation}
\label{eq:error_reference_path}
    \mathcal R_{\rm path}^{\pi,p}
    :=
    \sup_{0\le t\le T}
    \mathbb{E}\bigl[
        |\mathcal X_t^{\pi,p-1}-\mathcal X_t^{p-1}|^2
    \bigr].
\end{equation}
The term \(\mathcal R_{\rm term}^{\pi,p}\) measures the mismatch between the training target and the auxiliary terminal condition evaluated along the discrete reference path, while \(\mathcal R_{\rm path}^{\pi,p}\) measures the discrepancy between the discrete and continuous reference paths.

We now state the assumptions used in the analysis.

\begin{assumption}[Coefficient regularity]
\label{ass:error_coefficients}
There exist nonnegative constants \(L_b^x\), \(L_b^y\), \(L_b^z\), \(L_\sigma^x\), \(L_\sigma^y\), \(L_f^x\), \(L_f^y\), \(L_f^z\), \(L_g\), \(L_b^t\), \(L_\sigma^t\), and \(L_f^t\), and a constant \(L_0>0\) such that, for all \(t,s\in[0,T]\), \(x,x'\in\mathbb{R}^d\), \(y,y'\in\mathbb{R}^m\), and \(z,z'\in\mathbb{R}^{m\times d}\),
\begin{equation}
\label{eq:error_coeff_spatial_lipschitz}
\begin{aligned}
|b(t,x,y,z)-b(t,x',y',z')|
&\le
L_b^x|x-x'|+L_b^y|y-y'|+L_b^z\|z-z'\|,
\\
\|\sigma(t,x,y)-\sigma(t,x',y')\|
&\le
L_\sigma^x|x-x'|+L_\sigma^y|y-y'|,
\\
|f(t,x,y,z)-f(t,x',y',z')|
&\le
L_f^x|x-x'|+L_f^y|y-y'|+L_f^z\|z-z'\|,
\\
|g(x)-g(x')|
&\le
L_g|x-x'|.
\end{aligned}
\end{equation}
Moreover,
\begin{equation}
\label{eq:error_coeff_origin_bound}
\sup_{0\le t\le T}
\left(
|b(t,0,0,0)|
+
\|\sigma(t,0,0)\|
+
|f(t,0,0,0)|
\right)
+
|g(0)|
\le
L_0,
\end{equation}
\begin{equation*}
\begin{aligned}
|b(t,x,y,z)-b(s,x,y,z)|
&\le
L_b^t(1+|x|+|y|+\|z\|)|t-s|^{1/2},
\\
\|\sigma(t,x,y)-\sigma(s,x,y)\|
&\le
L_\sigma^t(1+|x|+|y|)|t-s|^{1/2},
\\
|f(t,x,y,z)-f(s,x,y,z)|
&\le
L_f^t(1+|x|+|y|+\|z\|)|t-s|^{1/2}.
\end{aligned}
\end{equation*}
\end{assumption}

\begin{assumption}[Well-posedness]
\label{ass:error_wellposedness}
The FBSDE~\eqref{eq:fully_coupled_FBSDE} and, for every \(p\ge1\), the auxiliary decoupling system~\eqref{eq:fp_step_p} admit unique adapted solutions \((X,Y,Z)\) and \((X^p,Y^p,Z^p)\), respectively.
Moreover,
\(
\mathcal X^0\in\mathcal S_{\mathbb F}^2(0,T;\mathbb{R}^d),
\)
and all these solutions belong to
\(
\mathcal S_{\mathbb F}^2(0,T;\mathbb{R}^d)
\times
\mathcal S_{\mathbb F}^2(0,T;\mathbb{R}^m)
\times
\mathcal L_{\mathbb F}^2(0,T;\mathbb{R}^{m\times d}).
\)
\end{assumption}

\begin{assumption}[Uniform neural bounds]
\label{ass:error_uniform_neural_bounds}
There exist constants \(K_{\rm nn},K_\phi>0\), independent of
\(\pi\), \(p\), and \(n\), such that, for every \(p\ge0\),
\[
    |\theta_{u_0}^{\,p}|
    +
    \sup_{0\le n\le N-1}
    |\phi_n(0;\theta_n^{\,p})|
    \le K_{\rm nn},
    \qquad
    \sup_{0\le n\le N-1}
    \sup_{x\in\mathbb{R}^d}
    \|\partial_x\phi_n(x;\theta_n^{\,p})\|_2
    \le K_\phi,
\]
where \(x\mapsto\phi_n(x;\theta_n^{\,p})\) is differentiable for every \(n=0,\ldots,N-1\).
\end{assumption}

For later use, define the constants
\begin{equation}\label{eq:error_stability_constants}
\begin{aligned}
\kappa_x&:=12T(L_b^x)^2+8(L_\sigma^x)^2,
\quad
\kappa_y:=12T(L_b^y)^2+8(L_\sigma^y)^2,
\quad
\kappa_z:=12T(L_b^z)^2,
\\
\beta_f&:=2L_f^y+2(L_f^z)^2+2,
\quad
K_F:=\exp(\kappa_xT)\max\{\kappa_yT,\kappa_z\},
\quad
K_B:=3\exp(\beta_fT)[L_g^2+T(L_f^x)^2].
\end{aligned}
\end{equation}

\begin{theorem}[Residual-based error estimate]
\label{thm:error_estimate}
Let Assumptions~\ref{ass:error_coefficients}, \ref{ass:error_wellposedness}, and \ref{ass:error_uniform_neural_bounds} hold.
Then, for every fixed \(p\ge1\), there exists a constant \(\Lambda_p>0\), independent of \(\pi\), such that
\begin{equation}
\label{eq:error_main_bound}
\|(X^{\pi,p}-X,Y^{\pi,p}-Y,Z^{\pi,p}-Z)\|_{\mathcal E}^2
\le
\Lambda_p
(
|\pi|
+
\mathcal L_{T}^{\pi,p}
+
\mathcal R_{\rm term}^{\pi,p}
+
\mathcal R_{\rm path}^{\pi,p}
)
+
2\mathcal E_{dec}^p.
\end{equation}
Moreover, for every \(\zeta>0\), there exist constants \(A_{p,\zeta}>0\), independent of \(\pi\), such that
\begin{equation}\label{eq:error_refined_bound}
\|(X^{\pi,p}-X,Y^{\pi,p}-Y,Z^{\pi,p}-Z)\|_{\mathcal E}^2
\le A_{p,\zeta}
\bigl(|\pi|+\mathcal L_T^{\pi,p}+\mathcal R_{\rm term}^{\pi,p}\bigr)
+\frac{\Gamma_\zeta}{p}\sum_{j=0}^{p-1}\mathfrak X_j^\pi
+2\mathcal E_{dec}^p,
\end{equation}
where \(\Gamma_\zeta:=2(1+\zeta)K_B(1+K_F)\).
\end{theorem}

For the convergence analysis, define
\begin{align}
\kappa_1
&:=
2K_\phi^2
(
L_\sigma^x+L_\sigma^yK_\phi
)^2,
\label{eq:error_uniform_moment_kappa}
\\
G^{\rm fwd}
&:=
\max\left\{
12T(L_b^x)^2+9(L_\sigma^x)^2,\,
12T(L_b^y)^2+9(L_\sigma^y)^2
\right\},
\qquad
G^{\rm bwd}
:=
12T(L_f^y)^2,
\label{eq:error_uniform_moment_growth}
\\
\nu_1
&:=
\exp[
(
G^{\rm fwd}+G^{\rm bwd}
)T
]
\left\{
12T^2
[
(L_f^x)^2
+
(
(L_b^z)^2+(L_f^z)^2
)\kappa_1
]
+
3T\kappa_1
\right\}.
\label{eq:error_uniform_moment_constant}
\end{align}

The following assumptions are used only for the convergence result.

\begin{assumption}[Uniform moment]
\label{ass:error_uniform_moment_stability}
The constant \(\nu_1\) in~\eqref{eq:error_uniform_moment_constant} satisfies \(\nu_1<1\).
\end{assumption}

\begin{assumption}[Weak coupling]
\label{ass:error_reference_contraction}
The constants \(K_F,K_B\) in~\eqref{eq:error_stability_constants} satisfy \(K_FK_B<1\).
\end{assumption}

\begin{assumption}[Residual consistency]
\label{ass:error_residual_consistency}
The terminal loss and terminal reference mismatch satisfy
\(\lim_{p\to\infty,\,|\pi|\to0}
\bigl(\mathcal L_T^{\pi,p}+\mathcal R_{\rm term}^{\pi,p}\bigr)=0\).
\end{assumption}

\begin{theorem}[Conditional convergence]
\label{thm:error_convergence}
Let the assumptions of Theorem~\ref{thm:error_estimate} hold.
Suppose further that Assumptions~\ref{ass:error_uniform_moment_stability}, \ref{ass:error_reference_contraction}, and \ref{ass:error_residual_consistency} hold.
Then
\begin{equation}
\label{eq:error_convergence_result}
\lim_{\substack{p\to\infty\\|\pi|\to0}}
\|(X^{\pi,p}-X,Y^{\pi,p}-Y,Z^{\pi,p}-Z)\|_{\mathcal E}^2
=
0.
\end{equation}
\end{theorem}

\subsection{Proof of the error estimate}
\label{subsec:error_proof_error_estimate}

The spatial Lipschitz bounds~\eqref{eq:error_coeff_spatial_lipschitz} and the origin bound~\eqref{eq:error_coeff_origin_bound} imply
\begin{equation}\label{eq:error_coefficient_linear_growth}
\begin{aligned}
|b(t,x,y,z)|
&\le L_0+L_b^x|x|+L_b^y|y|+L_b^z\|z\|,
\quad
\|\sigma(t,x,y)\|\le L_0+L_\sigma^x|x|+L_\sigma^y|y|,\\
|f(t,x,y,z)|
&\le L_0+L_f^x|x|+L_f^y|y|+L_f^z\|z\|,
\quad
|g(x)|\le L_0+L_g|x|.
\end{aligned}
\end{equation}

We first establish growth estimates for the neural quantities in the discrete scheme.

\begin{lemma}[Growth of the neural quantities]
\label{lem:error_neural_output_growth}
Let Assumptions~\ref{ass:error_coefficients} and~\ref{ass:error_uniform_neural_bounds} hold.
Then, for every \(p\ge1\) and \(0\le n<N\),
\begin{equation}
\label{eq:error_neural_output_growth}
|\widehat Y_n^p|
\le
K_{\rm nn}+K_\phi|\widetilde{\mathcal X}_n^{p-1}|,
\qquad
\|\widetilde Z_n^p\|
\le
K_\phi[
L_0+L_\sigma^yK_{\rm nn}
+\left(L_\sigma^x+L_\sigma^yK_\phi\right)
|\widetilde{\mathcal X}_n^{p-1}|
].
\end{equation}
At the initialization step \(p=0\), for \(0\le n\le N-1\), the corresponding estimates are
\begin{equation}
\label{eq:error_neural_output_growth_initial}
|\widehat Y_n^0|
\le
K_{\rm nn}+K_\phi|\widetilde X_n^0|,
\qquad
\|\widetilde Z_n^0\|
\le
K_\phi[
L_0+L_\sigma^yK_{\rm nn}
+\left(L_\sigma^x+L_\sigma^yK_\phi\right)
|\widetilde X_n^0|
].
\end{equation}
\end{lemma}

\begin{proof}
Fix \(p\ge1\) and \(0\le n\le N-1\).
By~\eqref{eq:structure_preserving_yz_discrete}, the mean value inequality, and Assumption~\ref{ass:error_uniform_neural_bounds},
\[
|\widehat Y_n^p|
=
|\phi_n(\widetilde{\mathcal X}_n^{p-1};\theta_n^{\,p-1})|
\le
|\phi_n(0;\theta_n^{\,p-1})|
+
K_\phi|\widetilde{\mathcal X}_n^{p-1}|
\le
K_{\rm nn}
+
K_\phi|\widetilde{\mathcal X}_n^{p-1}|.
\]
Again by~\eqref{eq:structure_preserving_yz_discrete},
\(
\widetilde Z_n^p
=
\partial_x\phi_n(\widetilde{\mathcal X}_n^{p-1};\theta_n^{\,p-1})
\sigma(t_n,\widetilde{\mathcal X}_n^{p-1},\widehat Y_n^p).
\)
Hence, Assumption~\ref{ass:error_uniform_neural_bounds} and the linear growth estimate~\eqref{eq:error_coefficient_linear_growth} give
\[
\begin{aligned}
\|\widetilde Z_n^p\|
&\le
K_\phi
(
L_0
+
L_\sigma^x|\widetilde{\mathcal X}_n^{p-1}|
+
L_\sigma^y|\widehat Y_n^p|
)
\le
K_\phi
[
L_0
+
L_\sigma^yK_{\rm nn}
+
(
L_\sigma^x+L_\sigma^yK_\phi
)
|\widetilde{\mathcal X}_n^{p-1}|
],
\end{aligned}
\]
which proves~\eqref{eq:error_neural_output_growth}.
At \(p=0\), the same argument applied to the initialization rollout in Algorithm~\ref{Algo:Algorithm-12}, with \(\widetilde{\mathcal X}_n^{p-1}\) replaced by \(\widetilde X_n^0\), gives~\eqref{eq:error_neural_output_growth_initial}.
\end{proof}

The preceding growth estimate allows us to derive moment bounds that are uniform with respect to the time mesh.

\begin{lemma}[Moment estimate for the discrete process]
\label{lem:error_numerical_moment}
Suppose that Assumptions~\ref{ass:error_coefficients} and \ref{ass:error_uniform_neural_bounds} hold.
For every fixed \(p\ge0\), there exists a constant \(C_p>0\),
independent of \(\pi\) and \(n\), such that
\begin{equation}
\label{eq:error_numerical_moment_bound}
\max_{0\le n\le N}
\mathbb{E}[
|\widetilde X_n^p|^2+|\widetilde Y_n^p|^2
]
+
\max_{0\le n\le N-1}
\mathbb{E}\|\widetilde Z_n^p\|^2
+
\sum_{n=0}^{N-1}
\mathbb{E}\|\widetilde Z_n^p\|^2\Delta t_n
\le C_p.
\end{equation}
Moreover, for every \(p\ge1\),
\begin{equation}
\label{eq:error_reference_pi_bound}
\max_{0\le n\le N}
\mathbb{E}|\widetilde{\mathcal X}_n^{p-1}|^2
\le C_p.
\end{equation}
\end{lemma}

\begin{proof}
We proceed by induction over \(p\).
Set
\(
\kappa_0
:=
2K_\phi^2
\left(
L_0+L_\sigma^yK_{\rm nn}
\right)^2,
\)
and recall \(\kappa_1\) from~\eqref{eq:error_uniform_moment_kappa}.
The linear growth estimates~\eqref{eq:error_coefficient_linear_growth} imply
\begin{align}
|b(t,x,y,z)|^2
&\le
4L_0^2
+
4(L_b^x)^2|x|^2
+
4(L_b^y)^2|y|^2
+
4(L_b^z)^2\|z\|^2,
\label{eq:error_b_squared_growth}
\\
\|\sigma(t,x,y)\|^2
&\le
3L_0^2
+
3(L_\sigma^x)^2|x|^2
+
3(L_\sigma^y)^2|y|^2,
\label{eq:error_sigma_squared_growth}
\\
|f(t,x,y,z)|^2
&\le
4L_0^2
+
4(L_f^x)^2|x|^2
+
4(L_f^y)^2|y|^2
+
4(L_f^z)^2\|z\|^2.
\label{eq:error_f_squared_growth}
\end{align}

We first consider \(p=0\).
By~\eqref{eq:error_neural_output_growth_initial}, \(\|\widetilde Z_n^0\|^2 \le \kappa_0+\kappa_1|\widetilde X_n^0|^2\) for \(n=0,\ldots,N-1\).
Summing the forward and backward recursions of the \(p=0\) initialization rollout in Algorithm~\ref{Algo:Algorithm-12} from \(\ell=0\) to \(n-1\), and applying the Cauchy--Schwarz inequality and the discrete It\^o isometry, together with \(|\theta_{u_0}^{\,0}|\le K_{\rm nn}\), yields, for \(n=0,\ldots,N\),
\[
\begin{aligned}
\mathbb{E}|\widetilde X_n^0|^2
&\le
3|x_0|^2
+
3T\sum_{\ell=0}^{n-1}
\mathbb{E}|b(t_\ell,\widetilde X_\ell^0,
\widetilde Y_\ell^0,\widetilde Z_\ell^0)|^2
\Delta t_\ell
+
3\sum_{\ell=0}^{n-1}
\mathbb{E}\|\sigma(t_\ell,\widetilde X_\ell^0,
\widetilde Y_\ell^0)\|^2
\Delta t_\ell,
\\
\mathbb{E}|\widetilde Y_n^0|^2
&\le
3K_{\rm nn}^2
+
3T\sum_{\ell=0}^{n-1}
\mathbb{E}|f(t_\ell,\widetilde X_\ell^0,
\widetilde Y_\ell^0,\widetilde Z_\ell^0)|^2
\Delta t_\ell
+
3\sum_{\ell=0}^{n-1}
\mathbb{E}\|\widetilde Z_\ell^0\|^2
\Delta t_\ell.
\end{aligned}
\]
Using~\eqref{eq:error_b_squared_growth}--\eqref{eq:error_f_squared_growth}, the preceding bound for \(\widetilde Z^0\), and \(\sum_{\ell=0}^{n-1}\Delta t_\ell\le T\), the above estimates imply
\[
\mathbb{E}[
|\widetilde X_n^0|^2+|\widetilde Y_n^0|^2
]
\le
K_0
+
G_0
\sum_{\ell=0}^{n-1}
\mathbb{E}[
|\widetilde X_\ell^0|^2+|\widetilde Y_\ell^0|^2
]\Delta t_\ell,
\]
for some constants \(K_0,G_0>0\), depending only on the coefficient bounds, \(K_{\rm nn}\), \(K_\phi\), \(x_0\), and \(T\), and independent of \(\pi\) and \(n\).
The discrete Gronwall inequality therefore yields a constant \(D_0>0\), independent of \(\pi\) and \(n\), such that
\[
\max_{0\le n\le N}
\mathbb{E}[
|\widetilde X_n^0|^2+|\widetilde Y_n^0|^2
]
\le D_0.
\]
Consequently,
\[
\max_{0\le n\le N-1}
\mathbb{E}\|\widetilde Z_n^0\|^2
+
\sum_{n=0}^{N-1}
\mathbb{E}\|\widetilde Z_n^0\|^2\Delta t_n
\le
(1+T)(\kappa_0+\kappa_1D_0),
\]
so~\eqref{eq:error_numerical_moment_bound} holds for \(p=0\) with
\(
C_0
:=
D_0+(1+T)(\kappa_0+\kappa_1D_0).
\)

Let \(p\ge1\), and suppose that for each \(j=0,\ldots,p-1\), there exist
constants \(D_j,C_j>0\), independent of \(\pi\) and \(n\), such that
\[
\max_{0\le n\le N}
\mathbb{E}[
|\widetilde X_n^j|^2+|\widetilde Y_n^j|^2
]
\le D_j,
\]
and~\eqref{eq:error_numerical_moment_bound} holds with \(p=j\) and constant \(C_j\).
Set
\(
\overline D_{p-1}
:=
\frac1p\sum_{j=0}^{p-1}D_j.
\)
By~\eqref{eq:fictitious_play_X} and
Jensen's inequality,
\begin{equation}
\label{eq:error_reference_moment_step_p}
\max_{0\le n\le N}
\mathbb{E}|\widetilde{\mathcal X}_n^{p-1}|^2
\le
\max_{0\le n\le N}
\frac1p\sum_{j=0}^{p-1}
\mathbb{E}|\widetilde X_n^j|^2
\le
\overline D_{p-1}.
\end{equation}
Combining this estimate with~\eqref{eq:error_neural_output_growth} gives \(\mathbb{E}\|\widetilde Z_n^p\|^2\le\kappa_0+\kappa_1\overline D_{p-1}\) for \(0\le n\le N-1\), and hence
\begin{equation}
\label{eq:error_Z_moment_step_p}
\max_{0\le n\le N-1}
\mathbb{E}\|\widetilde Z_n^p\|^2
+
\sum_{n=0}^{N-1}
\mathbb{E}\|\widetilde Z_n^p\|^2\Delta t_n
\le
(1+T)\left(
\kappa_0+\kappa_1\overline D_{p-1}
\right).
\end{equation}

We next estimate \(\widetilde X^p\) and \(\widetilde Y^p\).
Applying the same estimates as in the case \(p=0\) to~\eqref{eq:fbsde_X_discrete} and~\eqref{eq:fbsde_Y_discrete}, together with~\eqref{eq:error_reference_moment_step_p} and~\eqref{eq:error_Z_moment_step_p}, and the definitions of \(G^{\rm fwd}\) and \(G^{\rm bwd}\) in~\eqref{eq:error_uniform_moment_growth}, yields
\[
\mathbb{E}|\widetilde X_n^p|^2
\le
K_p^{\rm fwd}
+
G^{\rm fwd}
\sum_{\ell=0}^{n-1}
\mathbb{E}[
|\widetilde X_\ell^p|^2+|\widetilde Y_\ell^p|^2
]\Delta t_\ell,
\]
where
\(
K_p^{\rm fwd}
:=
3|x_0|^2
+
12T^2
[
L_0^2
+
(L_b^z)^2
(
\kappa_0+\kappa_1\overline D_{p-1}
)
]
+
9TL_0^2,
\)
and
\[
\mathbb{E}|\widetilde Y_n^p|^2
\le
K_p^{\rm bwd}
+
G^{\rm bwd}
\sum_{\ell=0}^{n-1}
\mathbb{E}|\widetilde Y_\ell^p|^2
\Delta t_\ell,
\]
where
\(
K_p^{\rm bwd}
:=
3K_{\rm nn}^2
+
12T^2
[
L_0^2
+
(L_f^x)^2\overline D_{p-1}
+
(L_f^z)^2
(
\kappa_0+\kappa_1\overline D_{p-1}
)
]
+
3T
(
\kappa_0+\kappa_1\overline D_{p-1}
).
\)
Set
\(
\Xi:=\exp[(G^{\rm fwd}+G^{\rm bwd})T].
\)
Adding these inequalities and applying the discrete Gronwall inequality gives
\[
\max_{0\le n\le N}
\mathbb{E}[
|\widetilde X_n^p|^2+|\widetilde Y_n^p|^2
]
\le
(
K_p^{\rm fwd}+K_p^{\rm bwd}
)\Xi.
\]

We now make the dependence on
\(\overline D_{p-1}\) explicit.
By the definitions of \(K_p^{\mathrm{fwd}}\) and
\(K_p^{\mathrm{bwd}}\), we have
\(
K_p^{\mathrm{fwd}}+K_p^{\mathrm{bwd}}
=
\mu_0+\mu_1\overline D_{p-1},
\)
where \(\mu_0>0\) is independent of \(p\), \(\pi\), and \(n\), and
\(
\mu_1
:=
12T^2
[
(L_f^x)^2
+
(
(L_b^z)^2+(L_f^z)^2
)\kappa_1
]
+
3T\kappa_1.
\)
Set \(\nu_0:=\Xi\mu_0\).
By~\eqref{eq:error_uniform_moment_constant},
\(\nu_1=\Xi\mu_1\).
Thus, defining
\begin{equation}
\label{eq:error_XY_moment_constant_recursion}
D_p
:=
\nu_0+\nu_1\overline D_{p-1},
\end{equation}
we obtain
\begin{equation}
\label{eq:error_XY_moment_bound_p}
\max_{0\le n\le N}
\mathbb{E}[
|\widetilde X_n^p|^2+|\widetilde Y_n^p|^2
]
\le
D_p.
\end{equation}

Finally, define
\begin{equation}
\label{eq:error_moment_constant_recursion}
C_p
:=
\max\left\{
\overline D_{p-1},\,
D_p
+
(1+T)
\left(
\kappa_0+\kappa_1\overline D_{p-1}
\right)
\right\}.
\end{equation}
Combining~\eqref{eq:error_XY_moment_bound_p} and
\eqref{eq:error_Z_moment_step_p} gives~\eqref{eq:error_numerical_moment_bound}.
Moreover,~\eqref{eq:error_reference_moment_step_p} and
\eqref{eq:error_moment_constant_recursion} give
\[
\max_{0\le n\le N}
\mathbb{E}|\widetilde{\mathcal X}_n^{p-1}|^2
\le
\overline D_{p-1}
\le
C_p,
\]
which proves~\eqref{eq:error_reference_pi_bound}.
The constants \(\kappa_0\), \(\kappa_1\), \(\Xi\), \(\mu_0\), \(\mu_1\),
\(\nu_0\), and \(\nu_1\) are independent of \(p\), \(\pi\), and \(n\).
Since \(D_0,\ldots,D_{p-1}\) are independent of \(\pi\) and \(n\), so are
\(\overline D_{p-1}\), \(D_p\), and \(C_p\).
This completes the induction.
\end{proof}

We now introduce continuous-time extensions of the discrete forward and backward processes.

\begin{lemma}[Continuous-time extension estimate]
\label{lem:error_numerical_extension}
Let Assumptions~\ref{ass:error_coefficients} and~\ref{ass:error_uniform_neural_bounds} hold, and fix \(p\ge1\).
Define the continuous extensions by
\begin{equation}
\label{eq:error_continuous_extension}
\begin{aligned}
\bar X_t^{\pi,p}
&:=
x_0
+
\int_0^t b(\tau(s),X_s^{\pi,p},Y_s^{\pi,p},Z_s^{\pi,p})\,\mathrm{d}s
+
\int_0^t\sigma(\tau(s),X_s^{\pi,p},Y_s^{\pi,p})\,\mathrm{d}W_s,
\\
\bar Y_t^{\pi,p}
&:=
\theta_{u_0}^{\,p-1}
-
\int_0^t f(\tau(s),\mathcal X_s^{\pi,p-1},
Y_s^{\pi,p},Z_s^{\pi,p})\,\mathrm{d}s
+
\int_0^t Z_s^{\pi,p}\,\mathrm{d}W_s.
\end{aligned}
\end{equation}
Then there are constants \(K_p^{X,{\rm ext}},K_p^{Y,{\rm ext}}>0\), independent of \(\pi\), such that
\begin{align}
\sup_{0\le t\le T}\mathbb E|\bar X_t^{\pi,p}-X_t^{\pi,p}|^2
&\le K_p^{X,{\rm ext}}|\pi|,\label{eq:error_extension_X_bound}\\
\sup_{0\le t\le T}\mathbb E|\bar Y_t^{\pi,p}-Y_t^{\pi,p}|^2
&\le K_p^{Y,{\rm ext}}|\pi|.\label{eq:error_extension_Y_bound}
\end{align}
\end{lemma}

\begin{proof}
Let \(C_p\) be the constant in
Lemma~\ref{lem:error_numerical_moment}, and set
\begin{equation}
\label{eq:error_extension_integrand_bounds}
\begin{aligned}
K_p^b
&:=
4L_0^2
+
4[(L_b^x)^2+(L_b^y)^2+(L_b^z)^2]C_p,
\\
K_p^\sigma
&:=
3L_0^2
+
3[(L_\sigma^x)^2+(L_\sigma^y)^2]C_p,
\\
K_p^f
&:=
4L_0^2
+
4[(L_f^x)^2+(L_f^y)^2+(L_f^z)^2]C_p.
\end{aligned}
\end{equation}
By~\eqref{eq:error_b_squared_growth}--\eqref{eq:error_f_squared_growth}
and Lemma~\ref{lem:error_numerical_moment}, these constants are independent of \(\pi\), and for \(0\le n\le N-1\),
\[
\mathbb{E}|b(t_n,\widetilde X_n^p,\widetilde Y_n^p,\widetilde Z_n^p)|^2\le K_p^b,\qquad
\mathbb{E}\|\sigma(t_n,\widetilde X_n^p,\widetilde Y_n^p)\|^2\le K_p^\sigma,\qquad
\mathbb{E}|f(t_n,\widetilde{\mathcal X}_n^{p-1},\widetilde Y_n^p,\widetilde Z_n^p)|^2\le K_p^f.
\]

We first consider the forward component.
By~\eqref{eq:error_continuous_extension} and
\eqref{eq:fbsde_X_discrete},
\(
\bar X_{t_n}^{\pi,p}=\widetilde X_n^p
\)
for \(0\le n\le N - 1\).
Hence, for \(t\in[t_n,t_{n+1})\),
\(
X_t^{\pi,p}
=
\widetilde X_n^p
=
\bar X_{t_n}^{\pi,p}.
\)
Using \eqref{eq:error_continuous_extension} on \([t_n,t]\), we get
\[
\bar X_t^{\pi,p}-X_t^{\pi,p}
=
\int_{t_n}^t
b(t_n,\widetilde X_n^p,\widetilde Y_n^p,\widetilde Z_n^p)\,\mathrm{d}s
+
\int_{t_n}^t
\sigma(t_n,\widetilde X_n^p,\widetilde Y_n^p)\,\mathrm{d}W_s.
\]
The Cauchy--Schwarz inequality, It\^o's isometry, and \(0\le t-t_n\le|\pi|\) give
\[
\mathbb{E}|\bar X_t^{\pi,p}-X_t^{\pi,p}|^2
\le
2(t-t_n)^2
\mathbb{E}|b(t_n,\widetilde X_n^p,
\widetilde Y_n^p,\widetilde Z_n^p)|^2
+
2(t-t_n)
\mathbb{E}\|\sigma(t_n,\widetilde X_n^p,\widetilde Y_n^p)\|^2
\le
\left(
2TK_p^b+2K_p^\sigma
\right)|\pi|.
\]
At \(t=T\), summing~\eqref{eq:fbsde_X_discrete} gives
\(
\bar X_T^{\pi,p}=\widetilde X_N^p=X_T^{\pi,p}.
\)
Thus~\eqref{eq:error_extension_X_bound} holds with
\(
K_p^{X,{\rm ext}}
:=
2TK_p^b+2K_p^\sigma.
\)

Similarly, \(Y_t^{\pi,p}=\bar Y_{t_n}^{\pi,p}=\widetilde Y_n^p\) for \(t\in[t_n,t_{n+1})\), and therefore
\[
\mathbb{E}|\bar Y_t^{\pi,p}-Y_t^{\pi,p}|^2
\le
2(t-t_n)^2K_p^f+2(t-t_n)C_p
\le
\left(2TK_p^f+2C_p\right)|\pi|.
\]
Since \(\bar Y_T^{\pi,p}=Y_T^{\pi,p}\),~\eqref{eq:error_extension_Y_bound} follows with \(K_p^{Y,{\rm ext}}:=2TK_p^f+2C_p\).
\end{proof}

The next lemma provides the coefficient-dependent stability constants used in the final error estimate.

\begin{lemma}[Coefficient-dependent stability estimates]
\label{lem:error_coefficient_stability}
Under Assumptions~\ref{ass:error_coefficients} and~\ref{ass:error_wellposedness}, for every \(p\ge1\),
\begin{equation}
\label{eq:error_coefficient_backward_stability}
\sup_{0\le t\le T}\mathbb{E}|Y_t^p-Y_t|^2
+\mathbb{E}\int_0^T\|Z_t^p-Z_t\|^2\,\mathrm{d}t
\le K_B\sup_{0\le t\le T}\mathbb{E}|\mathcal X_t^{p-1}-X_t|^2.
\end{equation}
Moreover,
\begin{equation}
\label{eq:error_coefficient_forward_stability}
\sup_{0\le t\le T}\mathbb{E}|X_t^p-X_t|^2
\le K_F\bigl(\sup_{0\le t\le T}\mathbb{E}|Y_t^p-Y_t|^2
+\mathbb{E}\int_0^T\|Z_t^p-Z_t\|^2\,\mathrm{d}t\bigr).
\end{equation}
\end{lemma}

\begin{proof}
Fix \(p\ge1\), and set \(\delta X:=X^p-X\), \(\delta Y:=Y^p-Y\), \(\delta Z:=Z^p-Z\), and \(\delta\mathcal X:=\mathcal X^{p-1}-X\).
The backward difference satisfies
\[
\delta Y_t=g(\mathcal X_T^{p-1})-g(X_T)+\int_t^T\delta f_s\,\mathrm{d}s-\int_t^T\delta Z_s\,\mathrm{d}W_s,
\]
where \(\delta f_s:=f(s,\mathcal X_s^{p-1},Y_s^p,Z_s^p)-f(s,X_s,Y_s,Z_s)\). By~\eqref{eq:error_coeff_spatial_lipschitz},
\begin{equation}
\label{eq:error_coeff_stability_delta_f}
|\delta f_s|\le L_f^x|\delta\mathcal X_s|+L_f^y|\delta Y_s|+L_f^z\|\delta Z_s\|.
\end{equation}

We first estimate \(\delta Y\).
Since
\(
d\delta Y_t=-\delta f_t\,\mathrm{d}t+\delta Z_t\,\mathrm{d}W_t,
\)
It\^o's formula applied to
\(e^{\beta_f t}|\delta Y_t|^2\), followed by integration over
\([t,T]\) and taking expectations, gives
\begin{equation}
\label{eq:error_coeff_backward_ito_identity}
\begin{aligned}
&e^{\beta_f t}\mathbb{E}|\delta Y_t|^2
+
\mathbb{E}\int_t^T e^{\beta_f s}\left(\beta_f|\delta Y_s|^2+\|\delta Z_s\|^2\right)\,\mathrm{d}s
\\
&=
e^{\beta_f T}\mathbb{E}|g(\mathcal X_T^{p-1})-g(X_T)|^2
+
2\mathbb{E}\int_t^T e^{\beta_f s}\langle\delta Y_s,\delta f_s\rangle\,\mathrm{d}s.
\end{aligned}
\end{equation}
By~\eqref{eq:error_coeff_stability_delta_f} and Young's inequality,
\begin{equation}
\label{eq:error_coeff_backward_driver_bound}
\begin{aligned}
2\langle\delta Y_s,\delta f_s\rangle
&\le
2L_f^x|\delta Y_s||\delta\mathcal X_s|
+
2L_f^y|\delta Y_s|^2
+
2L_f^z|\delta Y_s|\|\delta Z_s\|
\\
&\le
[1+2L_f^y+2(L_f^z)^2]|\delta Y_s|^2
+
(L_f^x)^2|\delta\mathcal X_s|^2
+
\frac12\|\delta Z_s\|^2.
\end{aligned}
\end{equation}
Substituting~\eqref{eq:error_coeff_backward_driver_bound} into~\eqref{eq:error_coeff_backward_ito_identity}, and using \(\beta_f=2L_f^y+2(L_f^z)^2+2\) from~\eqref{eq:error_stability_constants}, gives
\begin{equation}
\label{eq:error_coeff_backward_weighted_estimate}
\begin{aligned}
e^{\beta_f t}\mathbb{E}|\delta Y_t|^2
&+
\mathbb{E}\int_t^T e^{\beta_f s}|\delta Y_s|^2\,\mathrm{d}s
+
\frac12\mathbb{E}\int_t^T e^{\beta_f s}\|\delta Z_s\|^2\,\mathrm{d}s
\\
&\le
e^{\beta_f T}\mathbb{E}|g(\mathcal X_T^{p-1})-g(X_T)|^2
+
(L_f^x)^2\mathbb{E}\int_t^T e^{\beta_f s}|\delta\mathcal X_s|^2\,\mathrm{d}s.
\end{aligned}
\end{equation}

By the Lipschitz continuity of \(g\), we have
\[
\mathbb{E}|g(\mathcal X_T^{p-1})-g(X_T)|^2\le L_g^2\sup_{0\le s\le T}\mathbb{E}|\delta\mathcal X_s|^2
\]
 and
\[
\mathbb{E}\int_t^T e^{\beta_f s}|\delta\mathcal X_s|^2\,\mathrm{d}s\le Te^{\beta_f T}\sup_{0\le s\le T}\mathbb{E}|\delta\mathcal X_s|^2.
\]
Hence~\eqref{eq:error_coeff_backward_weighted_estimate} implies
\begin{equation}
\label{eq:error_coeff_backward_rhs_bound}
e^{\beta_f t}\mathbb{E}|\delta Y_t|^2
+
\mathbb{E}\int_t^T e^{\beta_f s}|\delta Y_s|^2\,\mathrm{d}s
+
\frac12\mathbb{E}\int_t^T e^{\beta_f s}\|\delta Z_s\|^2\,\mathrm{d}s
 \le
e^{\beta_f T}[L_g^2+T(L_f^x)^2]
\sup_{0\le s\le T}\mathbb{E}|\delta\mathcal X_s|^2.
\end{equation}
Dropping the nonnegative integral terms and using
\(e^{\beta_f t}\ge1\) in~\eqref{eq:error_coeff_backward_rhs_bound}, we obtain
\[
\sup_{0\le t\le T}\mathbb{E}|\delta Y_t|^2
\le
e^{\beta_f T}
[
L_g^2+T(L_f^x)^2
]
\sup_{0\le t\le T}
\mathbb{E}|\delta\mathcal X_t|^2.
\]
Taking \(t=0\) in~\eqref{eq:error_coeff_backward_rhs_bound} and using
\(e^{\beta_f s}\ge1\) gives
\[
\mathbb{E}\int_0^T\|\delta Z_t\|^2\,\mathrm{d}t
\le
2e^{\beta_f T}
[
L_g^2+T(L_f^x)^2
]
\sup_{0\le t\le T}
\mathbb{E}|\delta\mathcal X_t|^2.
\]
Adding these inequalities and using the definition of \(K_B\) in \eqref{eq:error_stability_constants} proves \eqref{eq:error_coefficient_backward_stability}.

For the forward component,
\(
\delta X_t=\int_0^t\delta b_s\,\mathrm{d}s+\int_0^t\delta\sigma_s\,\mathrm{d}W_s,
\)
where \(\delta b_s:=b(s,X_s^p,Y_s^p,Z_s^p)-b(s,X_s,Y_s,Z_s)\) and \(\delta\sigma_s:=\sigma(s,X_s^p,Y_s^p)-\sigma(s,X_s,Y_s)\). By~\eqref{eq:error_coeff_spatial_lipschitz},
\begin{equation}
\label{eq:error_coeff_forward_integrand_bounds}
\begin{aligned}
|\delta b_s|^2
&\le
3(L_b^x)^2|\delta X_s|^2
+
3(L_b^y)^2|\delta Y_s|^2
+
3(L_b^z)^2\|\delta Z_s\|^2,
\\
\|\delta\sigma_s\|^2
&\le
2(L_\sigma^x)^2|\delta X_s|^2
+
2(L_\sigma^y)^2|\delta Y_s|^2.
\end{aligned}
\end{equation}
The Cauchy--Schwarz inequality and It\^o's isometry give
\begin{equation}
\label{eq:error_coeff_forward_basic_estimate}
\mathbb{E}|\delta X_t|^2
\le
2T\int_0^t\mathbb{E}|\delta b_s|^2\,\mathrm{d}s
+
2\int_0^t\mathbb{E}\|\delta\sigma_s\|^2\,\mathrm{d}s.
\end{equation}
Substituting~\eqref{eq:error_coeff_forward_integrand_bounds} into~\eqref{eq:error_coeff_forward_basic_estimate} and using the definitions of \(\kappa_x,\kappa_y,\kappa_z\) in~\eqref{eq:error_stability_constants} yields
\[
\begin{aligned}
\mathbb{E}|\delta X_t|^2
\le&
\kappa_x\int_0^t\mathbb{E}|\delta X_s|^2\,\mathrm{d}s
+
\kappa_y\int_0^t\mathbb{E}|\delta Y_s|^2\,\mathrm{d}s
+
\kappa_z\int_0^t\mathbb{E}\|\delta Z_s\|^2\,\mathrm{d}s.
\end{aligned}
\]
Since \(\int_0^t\mathbb{E}|\delta Y_s|^2\,\mathrm{d}s\le T\sup_{0\le s\le T}\mathbb{E}|\delta Y_s|^2\) and \(\int_0^t\mathbb{E}\|\delta Z_s\|^2\,\mathrm{d}s\le\mathbb{E}\int_0^T\|\delta Z_s\|^2\,\mathrm{d}s\), we obtain
\[
\mathbb{E}|\delta X_t|^2
\le
\kappa_x\int_0^t\mathbb{E}|\delta X_s|^2\,\mathrm{d}s
+
\max\{\kappa_yT,\kappa_z\}
\bigl(
\sup_{0\le s\le T}\mathbb{E}|\delta Y_s|^2
+
\mathbb{E}\int_0^T\|\delta Z_s\|^2\,\mathrm{d}s
\bigr).
\]
Gronwall's inequality therefore yields
\[
\sup_{0\le t\le T}\mathbb{E}|\delta X_t|^2
\le
e^{\kappa_xT}
\max\{\kappa_yT,\kappa_z\}
\bigl(
\sup_{0\le t\le T}\mathbb{E}|\delta Y_t|^2
+
\mathbb{E}\int_0^T\|\delta Z_t\|^2\,\mathrm{d}t
\bigr).
\]
Using the definition of \(K_F\) in \eqref{eq:error_stability_constants} proves \eqref{eq:error_coefficient_forward_stability}.
\end{proof}

The refined estimate requires a recursion for the discrepancy between the discrete and continuous reference paths.

\begin{lemma}[Reference-path recursion]
\label{lem:error_reference_path_recursion}
For every \(p\ge1\), the reference-path discrepancy satisfies
\begin{equation}
\label{eq:error_reference_path_recursion}
\mathcal R_{\rm path}^{\pi,p}
\le
\frac1p
\sum_{j=0}^{p-1}\mathfrak X_j^\pi.
\end{equation}
\end{lemma}

\begin{proof}
Iterating the continuous running-average rule~\eqref{eq:fp_average_continuous}, we obtain
\(
\mathcal X_t^{p-1}
=
\frac1p
(
\mathcal X_t^0
+
\sum_{j=1}^{p-1}X_t^j
).
\)
Similarly, the discrete running-average rule~\eqref{eq:fictitious_play_X} gives
\(
\mathcal X_t^{\pi,p-1}
=
\frac1p
(
\mathcal X_t^{\pi,0}
+
\sum_{j=1}^{p-1}X_t^{\pi,j}
).
\)
Hence, Jensen's inequality yields
\[
\mathbb{E}
|
\mathcal X_t^{\pi,p-1}
-
\mathcal X_t^{p-1}
|^2
\le
\frac1p
(
\mathbb{E}
|
\mathcal X_t^{\pi,0}
-
\mathcal X_t^0
|^2
+
\sum_{j=1}^{p-1}
\mathbb{E}
|
X_t^{\pi,j}-X_t^j
|^2
).
\]
Taking the supremum over \(t\in[0,T]\) proves~\eqref{eq:error_reference_path_recursion}.
\end{proof}

We are now ready to prove the error estimate.

\begin{proof}[Proof of Theorem~\ref{thm:error_estimate}]
Fix \(p\ge1\), a partition \(\pi\), and \(\eta>0\).
Let \(\bar X^{\pi,p}\) and \(\bar Y^{\pi,p}\) be the continuous-time
extensions introduced in Lemma~\ref{lem:error_numerical_extension}, and set
\(
\delta X:=\bar X^{\pi,p}-X^p
\),
\(
\delta Y:=\bar Y^{\pi,p}-Y^p
\), and
\(
\delta Z:=Z^{\pi,p}-Z^p
\).
All constants introduced below are independent of \(\pi\).

We first estimate the backward component.
Since \(Y_T^p=g(\mathcal X_T^{p-1})\) and
\(\bar Y_T^{\pi,p}=Y_T^{\pi,p}\), we have
\[
\delta Y_T
=
Y_T^{\pi,p}-g(X_T^{\pi,p})
+
g(X_T^{\pi,p})-g(\mathcal X_T^{\pi,p-1})
+
g(\mathcal X_T^{\pi,p-1})-g(\mathcal X_T^{p-1}).
\]
Grouping the first two terms and using
\(
|a+b|^2\le(1+\eta^{-1})|a|^2+(1+\eta)|b|^2
\),
we obtain
\[
\mathbb{E}|\delta Y_T|^2
\le
(1+\eta^{-1})
\mathbb{E}|
Y_T^{\pi,p}-g(X_T^{\pi,p})
+
g(X_T^{\pi,p})-g(\mathcal X_T^{\pi,p-1})
|^2
 +
(1+\eta)
\mathbb{E}|
g(\mathcal X_T^{\pi,p-1})-g(\mathcal X_T^{p-1})
|^2.
\]
Using the Lipschitz continuity of \(g\), and~\eqref{eq:error_training_loss}--\eqref{eq:error_reference_path}, we obtain
\begin{equation}
\label{eq:error_terminal_deltaY_bound}
\mathbb{E}|\delta Y_T|^2
\le
K_\eta^{\rm term}
\left(
\mathcal L_{T}^{\pi,p}
+
\mathcal R_{\rm term}^{\pi,p}
\right)
+
(1+\eta)L_g^2\mathcal R_{\rm path}^{\pi,p},
\end{equation}
where
\(
K_\eta^{\rm term}:=2(1+\eta^{-1})
\).

For \(0\le t\le T\), the backward difference satisfies
\[
\delta Y_t
=
\delta Y_T
+
\int_t^T\delta f_s^{\pi,p}\,\mathrm{d}s
-
\int_t^T\delta Z_s\,\mathrm{d}W_s,
\]
where
\(
\delta f_s^{\pi,p}
:=
f(\tau(s),\mathcal X_s^{\pi,p-1},Y_s^{\pi,p},Z_s^{\pi,p})
-
f(s,\mathcal X_s^{p-1},Y_s^p,Z_s^p).
\)
We decompose
\(
\delta f_s^{\pi,p}
=
\rho_{f,{\rm loc},s}^{\pi,p}
+
\rho_{f,{\rm path},s}^{\pi,p}
+
\widehat{\delta f}_s^{\pi,p},
\)
where
\[
\begin{aligned}
\rho_{f,{\rm loc},s}^{\pi,p}
&:=
f(\tau(s),\mathcal X_s^{\pi,p-1},Y_s^{\pi,p},Z_s^{\pi,p})
-
f(s,\mathcal X_s^{\pi,p-1},\bar Y_s^{\pi,p},Z_s^{\pi,p}),
\\
\rho_{f,{\rm path},s}^{\pi,p}
&:=
f(s,\mathcal X_s^{\pi,p-1},\bar Y_s^{\pi,p},Z_s^{\pi,p})
-
f(s,\mathcal X_s^{p-1},\bar Y_s^{\pi,p},Z_s^{\pi,p}),
\\
\widehat{\delta f}_s^{\pi,p}
&:=
f(s,\mathcal X_s^{p-1},\bar Y_s^{\pi,p},Z_s^{\pi,p})
-
f(s,\mathcal X_s^{p-1},Y_s^p,Z_s^p).
\end{aligned}
\]

To estimate \(\rho_{f,{\rm loc}}^{\pi,p}\), we add and subtract
\(f(s,\mathcal X_s^{\pi,p-1},Y_s^{\pi,p},Z_s^{\pi,p})\).
Assumption~\ref{ass:error_coefficients},
together with \(|s-\tau(s)|\le|\pi|\), gives
\[
\begin{aligned}
|\rho_{f,{\rm loc},s}^{\pi,p}|^2
\le{}&
8(L_f^t)^2
\left(
1+|\mathcal X_s^{\pi,p-1}|^2
+|Y_s^{\pi,p}|^2
+\|Z_s^{\pi,p}\|^2
\right)|\pi|
+
2(L_f^y)^2
|Y_s^{\pi,p}-\bar Y_s^{\pi,p}|^2,
\\
|\rho_{f,{\rm path},s}^{\pi,p}|^2
\le{}&
(L_f^x)^2
|\mathcal X_s^{\pi,p-1}-\mathcal X_s^{p-1}|^2.
\end{aligned}
\]
By Lemma~\ref{lem:error_numerical_moment},
\[
\mathbb{E}\int_0^T
\left(
1+|\mathcal X_s^{\pi,p-1}|^2
+|Y_s^{\pi,p}|^2
+\|Z_s^{\pi,p}\|^2
\right)\,\mathrm{d}s
\le
T+(2T+1)C_p,
\]
while~\eqref{eq:error_extension_Y_bound} gives
\(
\int_0^T
\mathbb{E}|Y_s^{\pi,p}-\bar Y_s^{\pi,p}|^2\,\mathrm{d}s
\le
TK_p^{Y,{\rm ext}}|\pi|
\).
Therefore,
\begin{equation}
\label{eq:error_rho_component_bounds}
\mathbb{E}\int_0^T|\rho_{f,{\rm loc},s}^{\pi,p}|^2\,\mathrm{d}s
\le
K_p^{f,{\rm loc}}|\pi|,
\qquad
\mathbb{E}\int_0^T|\rho_{f,{\rm path},s}^{\pi,p}|^2\,\mathrm{d}s
\le
T(L_f^x)^2\mathcal R_{\rm path}^{\pi,p},
\end{equation}
where
\(
K_p^{f,{\rm loc}}
:=
8(L_f^t)^2[T+(2T+1)C_p]
+
2T(L_f^y)^2K_p^{Y,{\rm ext}}.
\)
Set
\(
\rho_{f,s}^{\pi,p}
:=
\rho_{f,{\rm loc},s}^{\pi,p}
+
\rho_{f,{\rm path},s}^{\pi,p}
\).
Young's inequality and~\eqref{eq:error_rho_component_bounds} yield
\begin{equation}
\label{eq:error_rho_bound}
\mathbb{E}\int_0^T|\rho_{f,s}^{\pi,p}|^2\,\mathrm{d}s
\le
K_{p,\eta}^{f,\rho}|\pi|
+
(1+\eta)T(L_f^x)^2\mathcal R_{\rm path}^{\pi,p},
\end{equation}
where
\(
K_{p,\eta}^{f,\rho}
:=
(1+\eta^{-1})K_p^{f,{\rm loc}}
\).

The backward difference can now be written as
\[
\delta Y_t
=
\delta Y_T
+
\int_t^T
(
\widehat{\delta f}_s^{\pi,p}
+
\rho_{f,s}^{\pi,p}
)\,\mathrm{d}s
-
\int_t^T\delta Z_s\,\mathrm{d}W_s.
\]
By~\eqref{eq:error_coeff_spatial_lipschitz},
\(
|\widehat{\delta f}_s^{\pi,p}|
\le
L_f^y|\delta Y_s|+L_f^z\|\delta Z_s\|
\).
Applying It\^o's formula to \(e^{\beta_f t}|\delta Y_t|^2\),
integrating over \([t,T]\), and taking expectations gives
\begin{equation}
\label{eq:error_backward_forced_ito_identity}
\begin{aligned}
&e^{\beta_f t}\mathbb{E}|\delta Y_t|^2
+
\mathbb{E}\int_t^T e^{\beta_f s}
\left(
\beta_f|\delta Y_s|^2+\|\delta Z_s\|^2
\right)\,\mathrm{d}s
\\
&\quad =
e^{\beta_f T}\mathbb{E}|\delta Y_T|^2
+
2\mathbb{E}\int_t^T e^{\beta_f s}
\langle
\delta Y_s,
\widehat{\delta f}_s^{\pi,p}
+
\rho_{f,s}^{\pi,p}
\rangle\,\mathrm{d}s.
\end{aligned}
\end{equation}
The Lipschitz estimate above and Young's inequality imply
\[
\begin{aligned}
2\langle
\delta Y_s,
\widehat{\delta f}_s^{\pi,p}
+
\rho_{f,s}^{\pi,p}
\rangle
&\le
2L_f^y|\delta Y_s|^2
+
2L_f^z|\delta Y_s|\|\delta Z_s\|
+
2|\delta Y_s||\rho_{f,s}^{\pi,p}|
\\
&\le
[
2L_f^y+2(L_f^z)^2+1
]|\delta Y_s|^2
+
\frac12\|\delta Z_s\|^2
+
|\rho_{f,s}^{\pi,p}|^2.
\end{aligned}
\]
Substituting this bound into
\eqref{eq:error_backward_forced_ito_identity} and using
\(
\beta_f=2L_f^y+2(L_f^z)^2+2
\)
gives
\begin{equation}
\label{eq:error_backward_forced_weighted_estimate}
e^{\beta_f t}\mathbb{E}|\delta Y_t|^2
+
\mathbb{E}\int_t^T e^{\beta_f s}|\delta Y_s|^2\,\mathrm{d}s
+
\frac12
\mathbb{E}\int_t^T e^{\beta_f s}\|\delta Z_s\|^2\,\mathrm{d}s
 \le
e^{\beta_f T}\mathbb{E}|\delta Y_T|^2
+
\mathbb{E}\int_t^T e^{\beta_f s}
|\rho_{f,s}^{\pi,p}|^2\,\mathrm{d}s.
\end{equation}
Taking the supremum in the first term, setting \(t=0\) for the
\(Z\)-term, and using
\(1\le e^{\beta_f s}\le e^{\beta_f T}\), we obtain
\begin{equation}
\label{eq:error_backward_stability_bound}
\sup_{0\le t\le T}\mathbb{E}|\delta Y_t|^2
+
\mathbb{E}\int_0^T\|\delta Z_t\|^2\,\mathrm{d}t
\le
3e^{\beta_f T}
(
\mathbb{E}|\delta Y_T|^2
+
\mathbb{E}\int_0^T|\rho_{f,t}^{\pi,p}|^2\,\mathrm{d}t
).
\end{equation}
Combining~\eqref{eq:error_terminal_deltaY_bound}, \eqref{eq:error_rho_bound}, and \eqref{eq:error_backward_stability_bound}, and recalling \(K_B=3e^{\beta_fT}[L_g^2+T(L_f^x)^2]\), gives
\begin{equation}
\label{eq:error_backward_reference_bound}
\sup_{0\le t\le T}
\mathbb{E}|\bar Y_t^{\pi,p}-Y_t^p|^2
+
\mathbb{E}\int_0^T
\|Z_t^{\pi,p}-Z_t^p\|^2\,\mathrm{d}t
 \le
A_{p,\eta}^{\rm bwd}
\left(
|\pi|
+
\mathcal L_{T}^{\pi,p}
+
\mathcal R_{\rm term}^{\pi,p}
\right)
+
(1+\eta)K_B\mathcal R_{\rm path}^{\pi,p},
\end{equation}
where
\(
A_{p,\eta}^{\rm bwd}
:=
3e^{\beta_fT}
\max\left\{
K_\eta^{\rm term},
K_{p,\eta}^{f,\rho}
\right\}.
\)

We next estimate the forward component.
By~\eqref{eq:error_continuous_extension},
\begin{equation}
\label{eq:error_forward_difference_equation}
\delta X_t
=
\int_0^t\delta b_s^{\pi,p}\,\mathrm{d}s
+
\int_0^t\delta\sigma_s^{\pi,p}\,\mathrm{d}W_s,
\end{equation}
where
\[
\begin{aligned}
\delta b_s^{\pi,p}
&:=
b(\tau(s),X_s^{\pi,p},Y_s^{\pi,p},Z_s^{\pi,p})
-
b(s,X_s^p,Y_s^p,Z_s^p),
\\
\delta\sigma_s^{\pi,p}
&:=
\sigma(\tau(s),X_s^{\pi,p},Y_s^{\pi,p})
-
\sigma(s,X_s^p,Y_s^p).
\end{aligned}
\]
Decompose \(\delta b_s^{\pi,p}=\rho_{b,s}^{\pi,p}+\widehat{\delta b}_s^{\pi,p}\) and \(\delta\sigma_s^{\pi,p}=\rho_{\sigma,s}^{\pi,p}+\widehat{\delta\sigma}_s^{\pi,p}\),
where
\[
\begin{aligned}
\rho_{b,s}^{\pi,p}
&:=
b(\tau(s),X_s^{\pi,p},Y_s^{\pi,p},Z_s^{\pi,p})
-
b(s,\bar X_s^{\pi,p},\bar Y_s^{\pi,p},Z_s^{\pi,p}),
\\
\rho_{\sigma,s}^{\pi,p}
&:=
\sigma(\tau(s),X_s^{\pi,p},Y_s^{\pi,p})
-
\sigma(s,\bar X_s^{\pi,p},\bar Y_s^{\pi,p}),
\\
\widehat{\delta b}_s^{\pi,p}
&:=
b(s,\bar X_s^{\pi,p},\bar Y_s^{\pi,p},Z_s^{\pi,p})
-
b(s,X_s^p,Y_s^p,Z_s^p),
\\
\widehat{\delta\sigma}_s^{\pi,p}
&:=
\sigma(s,\bar X_s^{\pi,p},\bar Y_s^{\pi,p})
-
\sigma(s,X_s^p,Y_s^p).
\end{aligned}
\]

We first estimate \(\rho_b^{\pi,p}\) and \(\rho_\sigma^{\pi,p}\).
Adding and subtracting \(b(s,X_s^{\pi,p},Y_s^{\pi,p},Z_s^{\pi,p})\) and \(\sigma(s,X_s^{\pi,p},Y_s^{\pi,p})\), respectively, and applying Assumption~\ref{ass:error_coefficients}, we obtain
\[
\begin{aligned}
|\rho_{b,s}^{\pi,p}|^2
&\le
12(L_b^t)^2
\left(
1+|X_s^{\pi,p}|^2+|Y_s^{\pi,p}|^2+\|Z_s^{\pi,p}\|^2
\right)|\pi|
+
3(L_b^x)^2|X_s^{\pi,p}-\bar X_s^{\pi,p}|^2
+
3(L_b^y)^2|Y_s^{\pi,p}-\bar Y_s^{\pi,p}|^2,
\\
\|\rho_{\sigma,s}^{\pi,p}\|^2
&\le
9(L_\sigma^t)^2
\left(
1+|X_s^{\pi,p}|^2+|Y_s^{\pi,p}|^2
\right)|\pi|
+
3(L_\sigma^x)^2|X_s^{\pi,p}-\bar X_s^{\pi,p}|^2
+
3(L_\sigma^y)^2|Y_s^{\pi,p}-\bar Y_s^{\pi,p}|^2.
\end{aligned}
\]
By Lemma~\ref{lem:error_numerical_moment} and Lemma~\ref{lem:error_numerical_extension}, we obtain
\begin{equation}
\label{eq:error_forward_residual_bound}
\mathbb{E}\int_0^T
(
|\rho_{b,s}^{\pi,p}|^2
+
\|\rho_{\sigma,s}^{\pi,p}\|^2
)\,\mathrm{d}s
\le
K_p^{\rm res}|\pi|,
\end{equation}
where
\[
\begin{aligned}
K_p^{\rm res}
:={}&
12(L_b^t)^2[T+(2T+1)C_p]
+
3T[(L_b^x)^2K_p^{X,{\rm ext}}+(L_b^y)^2K_p^{Y,{\rm ext}}]
\\
&+
9T(L_\sigma^t)^2(1+2C_p)
+
3T[(L_\sigma^x)^2K_p^{X,{\rm ext}}+(L_\sigma^y)^2K_p^{Y,{\rm ext}}].
\end{aligned}
\]

By~\eqref{eq:error_forward_difference_equation}, the Cauchy--Schwarz inequality, and It\^o's isometry, we obtain
\[
\begin{aligned}
\mathbb{E}|\delta X_t|^2
&\le
2T\int_0^t
\mathbb{E}
\bigl|
\widehat{\delta b}_s^{\pi,p}
+
\rho_{b,s}^{\pi,p}
\bigr|^2\,\mathrm{d}s
+
2\int_0^t
\mathbb{E}
\bigl\|
\widehat{\delta\sigma}_s^{\pi,p}
+
\rho_{\sigma,s}^{\pi,p}
\bigr\|^2\,\mathrm{d}s
\\
&\le
4T\int_0^t
\mathbb{E}|\widehat{\delta b}_s^{\pi,p}|^2\,\mathrm{d}s
+
4\int_0^t
\mathbb{E}\|\widehat{\delta\sigma}_s^{\pi,p}\|^2\,\mathrm{d}s
+
4\max\{T,1\}
\mathbb{E}\int_0^t
\bigl(
|\rho_{b,s}^{\pi,p}|^2
+
\|\rho_{\sigma,s}^{\pi,p}\|^2
\bigr)\,\mathrm{d}s.
\end{aligned}
\]
Using the Lipschitz bounds for \(\widehat{\delta b}^{\pi,p}\) and \(\widehat{\delta\sigma}^{\pi,p}\), ~\eqref{eq:error_forward_residual_bound}, and the definitions of \(\kappa_x,\kappa_y,\kappa_z\) in~\eqref{eq:error_stability_constants}, it follows that
\begin{equation}
\label{eq:error_forward_pre_gronwall_main}
\mathbb{E}|\delta X_t|^2
\le
\kappa_x\int_0^t\mathbb{E}|\delta X_s|^2\,\mathrm{d}s
+
\kappa_y\int_0^t\mathbb{E}|\delta Y_s|^2\,\mathrm{d}s
+
\kappa_z\int_0^t\mathbb{E}\|\delta Z_s\|^2\,\mathrm{d}s
+
K_p^{\rm fwd,res}|\pi|,
\end{equation}
where
\(
K_p^{\rm fwd,res}
:=
4\max\{T,1\}K_p^{\rm res}
\).
Since
\[
\int_0^t\mathbb{E}|\delta Y_s|^2\,\mathrm{d}s
\le
T\sup_{0\le s\le T}\mathbb{E}|\delta Y_s|^2,
\qquad
\int_0^t\mathbb{E}\|\delta Z_s\|^2\,\mathrm{d}s
\le
\mathbb{E}\int_0^T\|\delta Z_s\|^2\,\mathrm{d}s,
\]
Eq.~\eqref{eq:error_forward_pre_gronwall_main} implies
\[
\mathbb{E}|\delta X_t|^2
\le
\kappa_x\int_0^t\mathbb{E}|\delta X_s|^2\,\mathrm{d}s
 +
\max\{\kappa_yT,\kappa_z\}
\bigl(
\sup_{0\le s\le T}\mathbb{E}|\delta Y_s|^2
+
\mathbb{E}\int_0^T\|\delta Z_s\|^2\,\mathrm{d}s
\bigr)
+
K_p^{\rm fwd,res}|\pi|.
\]
Gronwall's inequality and the definition of \(K_F\) therefore give
\[
\sup_{0\le t\le T}\mathbb{E}|\delta X_t|^2
\le
K_F
\bigl(
\sup_{0\le t\le T}\mathbb{E}|\delta Y_t|^2
+
\mathbb{E}\int_0^T\|\delta Z_t\|^2\,\mathrm{d}t
\bigr)
+
e^{\kappa_xT}K_p^{\rm fwd,res}|\pi|.
\]
Substituting~\eqref{eq:error_backward_reference_bound} yields
\begin{equation}
\label{eq:error_forward_reference_bound}
\sup_{0\le t\le T}
\mathbb{E}|\bar X_t^{\pi,p}-X_t^p|^2
\le
A_{p,\eta}^{\rm fwd}
\left(
|\pi|
+
\mathcal L_{T}^{\pi,p}
+
\mathcal R_{\rm term}^{\pi,p}
\right)
+
(1+\eta)K_BK_F\mathcal R_{\rm path}^{\pi,p},
\end{equation}
where
\(
A_{p,\eta}^{\rm fwd}
:=
e^{\kappa_xT}K_p^{\rm fwd,res}
+
K_FA_{p,\eta}^{\rm bwd}.
\)

We now return from the continuous extensions to the piecewise-constant processes.
Since
\(
Y^{\pi,p}-Y^p
=
(Y^{\pi,p}-\bar Y^{\pi,p})
+
(\bar Y^{\pi,p}-Y^p)
\),
Young's inequality and \(1\le1+\eta\) give
\[
\mathfrak B_p^\pi
\le
(1+\eta^{-1})
\sup_{0\le t\le T}
\mathbb{E}|Y_t^{\pi,p}-\bar Y_t^{\pi,p}|^2
 +
(1+\eta)
\bigl(
\sup_{0\le t\le T}
\mathbb{E}|\bar Y_t^{\pi,p}-Y_t^p|^2
+
\mathbb{E}\int_0^T
\|Z_t^{\pi,p}-Z_t^p\|^2\,\mathrm{d}t
\bigr).
\]
Applying~\eqref{eq:error_extension_Y_bound} and
\eqref{eq:error_backward_reference_bound}, we obtain
\begin{equation}
\label{eq:error_backward_piecewise_bound}
\mathfrak B_p^\pi
\le
\widetilde A_{p,\eta}^{\rm bwd}
\left(
|\pi|
+
\mathcal L_{T}^{\pi,p}
+
\mathcal R_{\rm term}^{\pi,p}
\right)
+
(1+\eta)^2K_B\mathcal R_{\rm path}^{\pi,p},
\end{equation}
where
\(
\widetilde A_{p,\eta}^{\rm bwd}
:=
(1+\eta)A_{p,\eta}^{\rm bwd}
+
(1+\eta^{-1})K_p^{Y,{\rm ext}}
\).

Similarly,
\(
X^{\pi,p}-X^p
=
(X^{\pi,p}-\bar X^{\pi,p})
+
(\bar X^{\pi,p}-X^p)
\).
Using Young's inequality,
\eqref{eq:error_extension_X_bound}, and
\eqref{eq:error_forward_reference_bound}, we obtain
\begin{equation}
\label{eq:error_forward_piecewise_bound}
\mathfrak X_p^\pi
\le
\widetilde A_{p,\eta}^{\rm fwd}
\left(
|\pi|
+
\mathcal L_{T}^{\pi,p}
+
\mathcal R_{\rm term}^{\pi,p}
\right)
+
(1+\eta)^2K_BK_F\mathcal R_{\rm path}^{\pi,p},
\end{equation}
where
\(
\widetilde A_{p,\eta}^{\rm fwd}
:=
(1+\eta)A_{p,\eta}^{\rm fwd}
+
(1+\eta^{-1})K_p^{X,{\rm ext}}
\).
Since
\(
\mathcal D_p^\pi
\le
\mathfrak X_p^\pi+\mathfrak B_p^\pi
\),
adding~\eqref{eq:error_backward_piecewise_bound} and
\eqref{eq:error_forward_piecewise_bound} gives
\begin{equation}
\label{eq:error_numerical_reference_bound}
\mathcal D_p^\pi
\le
A_{p,\eta}^{\rm num}
\left(
|\pi|
+
\mathcal L_{T}^{\pi,p}
+
\mathcal R_{\rm term}^{\pi,p}
\right)
+
(1+\eta)^2K_B(1+K_F)\mathcal R_{\rm path}^{\pi,p},
\end{equation}
where
\(
A_{p,\eta}^{\rm num}
:=
\widetilde A_{p,\eta}^{\rm fwd}
+
\widetilde A_{p,\eta}^{\rm bwd}
\).

By the definitions~\eqref{eq:error_decoupling_error} and \eqref{eq:error_reference_numerical_error},
\begin{equation}
\label{eq:error_total_reference_decomposition}
\|(X^{\pi,p}-X,Y^{\pi,p}-Y,Z^{\pi,p}-Z)\|_{\mathcal E}^2
\le
2\mathcal D_p^\pi
+
2\mathcal E_{dec}^p.
\end{equation}
Combining~\eqref{eq:error_numerical_reference_bound} and \eqref{eq:error_total_reference_decomposition}, we obtain
\begin{equation}
\label{eq:error_parameterized_total_bound}
\begin{aligned}
\|(X^{\pi,p}-X,Y^{\pi,p}-Y,Z^{\pi,p}-Z)\|_{\mathcal E}^2
&\le
2\mathcal E_{dec}^p
+
2A_{p,\eta}^{\rm num}
\left(
|\pi|
+
\mathcal L_{T}^{\pi,p}
+
\mathcal R_{\rm term}^{\pi,p}
\right)
\\
&\quad +
2(1+\eta)^2K_B(1+K_F)
\mathcal R_{\rm path}^{\pi,p}.
\end{aligned}
\end{equation}
Taking \(\eta=1\) in~\eqref{eq:error_parameterized_total_bound}
proves~\eqref{eq:error_main_bound} with
\(
\Lambda_p
:=
2\max\left\{
A_{p,1}^{\rm num},
4K_B(1+K_F)
\right\}.
\)

It remains to prove the refined estimate.
Fix \(\zeta>0\) and set
\(
\zeta_*:=\sqrt{1+\zeta}-1
\),
so that
\(
(1+\zeta_*)^2=1+\zeta
\).
Taking \(\eta=\zeta_*\) in~\eqref{eq:error_forward_piecewise_bound} and using Lemma~\ref{lem:error_reference_path_recursion} gives
\begin{equation}
\label{eq:error_forward_recursive_bound}
\mathfrak X_p^\pi
\le
\widetilde A_{p,\zeta_*}^{\rm fwd}
\left(
|\pi|
+
\mathcal L_{T}^{\pi,p}
+
\mathcal R_{\rm term}^{\pi,p}
\right)
+
\frac{(1+\zeta)K_BK_F}{p}
\sum_{j=0}^{p-1}\mathfrak X_j^\pi.
\end{equation}
Similarly, taking \(\eta=\zeta_*\) in \eqref{eq:error_numerical_reference_bound} and applying \eqref{eq:error_reference_path_recursion}, we obtain
\begin{equation}
\label{eq:error_refined_numerical_reference_bound}
\mathcal D_p^\pi
\le
A_{p,\zeta_*}^{\rm num}
\left(
|\pi|
+
\mathcal L_{T}^{\pi,p}
+
\mathcal R_{\rm term}^{\pi,p}
\right)
+
\frac{(1+\zeta)K_B(1+K_F)}{p}
\sum_{j=0}^{p-1}\mathfrak X_j^\pi.
\end{equation}
Combining \eqref{eq:error_refined_numerical_reference_bound} with \eqref{eq:error_total_reference_decomposition} proves \eqref{eq:error_refined_bound} with \(A_{p,\zeta}:=2A_{p,\zeta_*}^{\rm num}\) and \(\Gamma_\zeta:=2(1+\zeta)K_B(1+K_F)\).
\end{proof}

\subsection{Proof of the convergence}
\label{subsec:error_proof_convergence}

We first upgrade the moment estimate in Lemma~\ref{lem:error_numerical_moment}, which holds for fixed \(p\), to a bound that is uniform with respect to the decoupling step.

\begin{lemma}[Uniform discrete moment estimate]
\label{lem:error_uniform_discrete_moment}
Let Assumptions~\ref{ass:error_coefficients}, \ref{ass:error_uniform_neural_bounds}, and~\ref{ass:error_uniform_moment_stability} hold.
Then there exists a constant \(C_\star>0\), independent of \(p\), \(\pi\), and \(n\), such that, for every \(p\ge0\),
\begin{equation}
\label{eq:error_uniform_discrete_moment_bound}
\max_{0\le n\le N}
\mathbb{E}[
|\widetilde X_n^p|^2+|\widetilde Y_n^p|^2
]
+
\max_{0\le n\le N-1}
\mathbb{E}\|\widetilde Z_n^p\|^2
+
\sum_{n=0}^{N-1}
\mathbb{E}\|\widetilde Z_n^p\|^2\Delta t_n
\le C_\star.
\end{equation}
Moreover, for every \(p\ge1\),
\begin{equation}
\label{eq:error_uniform_reference_moment_bound}
\max_{0\le n\le N}
\mathbb{E}|\widetilde{\mathcal X}_n^{p-1}|^2
\le C_\star.
\end{equation}
\end{lemma}

\begin{proof}
Let \(D_p\) and \(C_p\) be the constants constructed in the proof of Lemma~\ref{lem:error_numerical_moment}.
By~\eqref{eq:error_XY_moment_constant_recursion}, \(D_p=\nu_0+\nu_1\overline D_{p-1}\), where \(\overline D_{p-1}=p^{-1}\sum_{j=0}^{p-1}D_j\) for \(p\ge1\).
Since \(\nu_1<1\) by Assumption~\ref{ass:error_uniform_moment_stability}, define \(D_\star:=\max\{D_0,\nu_0/(1-\nu_1)\}\).
We prove by induction that \(D_p\le D_\star\) for every \(p\ge0\).
The assertion is immediate for \(p=0\). Suppose that
\(D_j\le D_\star\) for \(0\le j\le p-1\). Then
\(\overline D_{p-1}\le D_\star\),
and hence
\[
D_p
=
\nu_0+\nu_1\overline D_{p-1}
\le
\nu_0+\nu_1D_\star
\le
D_\star,
\]
where the last inequality follows from \(D_\star\ge\nu_0/(1-\nu_1)\).
Therefore,
\begin{equation}
\label{eq:error_uniform_D_bound}
\sup_{p\ge0}D_p\le D_\star,
\qquad
\sup_{p\ge1}\overline D_{p-1}\le D_\star.
\end{equation}

Recall \(\kappa_0=
2K_\phi^2(L_0+L_\sigma^yK_{\rm nn})^2\) and define
\(
C_\star
:=
D_\star
+
(1+T)
\left(
\kappa_0+\kappa_1D_\star
\right).
\)
For \(p\ge1\), \eqref{eq:error_moment_constant_recursion} and \eqref{eq:error_uniform_D_bound} give \(C_p\le C_\star\).
The same bound holds for \(C_0\), since \(D_0\le D_\star\).
Thus
\(
\sup_{p\ge0}C_p\le C_\star
\).
Applying~\eqref{eq:error_numerical_moment_bound} and \eqref{eq:error_reference_pi_bound} proves \eqref{eq:error_uniform_discrete_moment_bound} and \eqref{eq:error_uniform_reference_moment_bound}, respectively.
\end{proof}

\begin{corollary}[Uniform refined error estimates]
\label{cor:error_uniform_refined_estimates}
Let Assumptions~\ref{ass:error_coefficients}--\ref{ass:error_uniform_moment_stability} hold.
Then, for every \(\zeta>0\), there exist constants \(\overline A_\zeta^{\rm fwd},\overline A_\zeta^{\rm tot}>0\), independent of \(p\) and \(\pi\), such that, for every \(p\ge1\) and every partition \(\pi\),
\begin{gather}
\mathfrak X_p^\pi
\le
\overline A_\zeta^{\rm fwd}
\left(
|\pi|+\mathcal L_T^{\pi,p}+\mathcal R_{\rm term}^{\pi,p}
\right)
+
\frac{(1+\zeta)K_BK_F}{p}
\sum_{j=0}^{p-1}\mathfrak X_j^\pi,
\label{eq:error_uniform_forward_recursive_bound}
\\
\|(X^{\pi,p}-X,Y^{\pi,p}-Y,Z^{\pi,p}-Z)\|_{\mathcal E}^2
\le
\overline A_\zeta^{\rm tot}
\left(
|\pi|+\mathcal L_T^{\pi,p}+\mathcal R_{\rm term}^{\pi,p}
\right)
+
\frac{\Gamma_\zeta}{p}
\sum_{j=0}^{p-1}\mathfrak X_j^\pi
+
2\mathcal E_{dec}^p,
\label{eq:error_uniform_total_recursive_bound}
\end{gather}
where \(\Gamma_\zeta=2(1+\zeta)K_B(1+K_F)\).
Moreover, the constants \(\Lambda_p\) in Theorem~\ref{thm:error_estimate} can be chosen such that \(\sup_{p\ge1}\Lambda_p<\infty\).
\end{corollary}

\begin{proof}
Fix \(\zeta>0\) and set
\(\zeta_*:=\sqrt{1+\zeta}-1\).
By Lemma~\ref{lem:error_uniform_discrete_moment},
\(C_p\le C_\star\) uniformly in \(p\).
Hence~\eqref{eq:error_extension_integrand_bounds} and the definitions in Lemma~\ref{lem:error_numerical_extension} show that \(K_p^b\), \(K_p^\sigma\), \(K_p^f\), \(K_p^{X,{\rm ext}}\), and \(K_p^{Y,{\rm ext}}\) are uniformly bounded in \(p\).

For every fixed \(\eta>0\), the definitions following \eqref{eq:error_rho_component_bounds}, \eqref{eq:error_rho_bound}, and \eqref{eq:error_backward_reference_bound} then give uniform bounds for \(K_p^{f,{\rm loc}}\), \(K_{p,\eta}^{f,\rho}\), and \(A_{p,\eta}^{\rm bwd}\).
Similarly, \eqref{eq:error_forward_residual_bound}--\eqref{eq:error_forward_reference_bound} give uniform bounds for \(K_p^{\rm res}\), \(K_p^{\rm fwd,res}\), and \(A_{p,\eta}^{\rm fwd}\).
Finally, the definitions following \eqref{eq:error_backward_piecewise_bound}, \eqref{eq:error_forward_piecewise_bound}, and \eqref{eq:error_numerical_reference_bound} show that \(\widetilde A_{p,\eta}^{\rm bwd}\), \(\widetilde A_{p,\eta}^{\rm fwd}\), and \(A_{p,\eta}^{\rm num}\) are uniformly bounded in \(p\).

Taking \(\eta=\zeta_*\), define
\(
\overline A_\zeta^{\rm fwd}
:=
\sup_{p\ge1}\widetilde A_{p,\zeta_*}^{\rm fwd}
\)
and
\(
\overline A_\zeta^{\rm tot}
:=
2\sup_{p\ge1}A_{p,\zeta_*}^{\rm num}.
\)
Then~\eqref{eq:error_forward_recursive_bound} gives \eqref{eq:error_uniform_forward_recursive_bound}, while \eqref{eq:error_refined_bound} gives \eqref{eq:error_uniform_total_recursive_bound}.

Taking \(\eta=1\) instead gives \(\sup_{p\ge1}A_{p,1}^{\rm num}<\infty\).
Since
\(
\Lambda_p
=
2\max\{A_{p,1}^{\rm num},4K_B(1+K_F)\}
\),
we also have
\(
\sup_{p\ge1}\Lambda_p<\infty
\).
\end{proof}

We next establish convergence of the exact continuous decoupling iteration.

\begin{lemma}[Convergence of the exact decoupling iteration]
\label{lem:error_exact_decoupling_convergence}
Let Assumptions~\ref{ass:error_coefficients}, \ref{ass:error_wellposedness}, and~\ref{ass:error_reference_contraction} hold.
Set \(q:=\sqrt{K_BK_F}\) and \(a_0:=(\sup_{0\le t\le T}\mathbb E|\mathcal X_t^0-X_t|^2)^{1/2}\).
Then \(q\in[0,1)\), \(a_0<\infty\), and, for every \(p\ge1\),
\begin{equation}\label{eq:error_exact_decoupling_rate}
\mathcal E_{dec}^p
\le K_B(1+K_F)a_0^2\Bigl(\frac{2}{p+1}\Bigr)^{2(1-q)}.
\end{equation}
In particular, \(\lim_{p\to\infty}\mathcal E_{dec}^p=0\).
\end{lemma}

\begin{proof}
For \(p\ge0\), set
\(
a_p
:=
\bigl(
\sup_{0\le t\le T}
\mathbb{E}|\mathcal X_t^p-X_t|^2
\bigr)^{1/2}.
\)
By Assumption~\ref{ass:error_wellposedness}, \(\mathcal X^0,X\in\mathcal S_{\mathbb F}^2\), and hence \(a_0<\infty\).
Assumption~\ref{ass:error_reference_contraction} gives \(q\in[0,1)\).

For every \(p\ge1\), the definition of \(a_{p-1}\) and \eqref{eq:error_coefficient_backward_stability} yield
\begin{equation}
\label{eq:error_exact_decoupling_backward_stability}
\sup_{0\le t\le T}\mathbb{E}|Y_t^p-Y_t|^2
+
\mathbb{E}\int_0^T\|Z_t^p-Z_t\|^2\,\mathrm{d}t
\le
K_Ba_{p-1}^2.
\end{equation}
Combining this estimate with \eqref{eq:error_coefficient_forward_stability}, we obtain
\begin{equation}
\label{eq:error_exact_decoupling_forward_contraction}
\bigl(
\sup_{0\le t\le T}
\mathbb{E}|X_t^p-X_t|^2
\bigr)^{1/2}
\le
(
K_BK_Fa_{p-1}^2
)^{1/2}
=
q\,a_{p-1}.
\end{equation}

The continuous running-average rule~\eqref{eq:fp_average_continuous} gives \(\mathcal X^p-X=\frac{p}{p+1}(\mathcal X^{p-1}-X)+\frac1{p+1}(X^p-X)\).
For each \(t\in[0,T]\), Minkowski's inequality therefore implies
\[
\left(
\mathbb{E}|\mathcal X_t^p-X_t|^2
\right)^{1/2}
\le
\frac{p}{p+1}
(
\mathbb{E}|\mathcal X_t^{p-1}-X_t|^2
)^{1/2}
+
\frac1{p+1}
(
\mathbb{E}|X_t^p-X_t|^2
)^{1/2}.
\]
Taking the supremum over \(t\in[0,T]\) and using
\eqref{eq:error_exact_decoupling_forward_contraction}, we find
\begin{equation}
\label{eq:error_exact_decoupling_average_recursion}
a_p
\le
\frac{p}{p+1}a_{p-1}
+
\frac{q}{p+1}a_{p-1}
=
(
1-\frac{1-q}{p+1}
)a_{p-1}.
\end{equation}
Iterating \eqref{eq:error_exact_decoupling_average_recursion} and using \(1-r\le e^{-r}\) for \(r\ge0\), we obtain
\[
a_p
\le
a_0
\prod_{k=1}^p
\bigl(
1-\frac{1-q}{k+1}
\bigr)
\le
a_0
\exp\bigl[
-(1-q)
\sum_{k=1}^p\frac1{k+1}
\bigr].
\]
Furthermore, using
\(
\sum_{k=1}^p(k+1)^{-1}
\ge
\log((p+2)/2)
\) gives
\begin{equation}
\label{eq:error_exact_reference_rate}
a_p
\le
a_0
\bigl(
\frac{2}{p+2}
\bigr)^{1-q}.
\end{equation}
The same bound holds trivially for \(p=0\).

We finally estimate the decoupling error~\eqref{eq:error_decoupling_error}.
By \eqref{eq:error_exact_decoupling_backward_stability} and \eqref{eq:error_exact_decoupling_forward_contraction},
\[
\begin{aligned}
\mathcal E_{dec}^p
&\le
\sup_{0\le t\le T}\mathbb{E}|X_t^p-X_t|^2
+
\sup_{0\le t\le T}\mathbb{E}|Y_t^p-Y_t|^2
+
\mathbb{E}\int_0^T\|Z_t^p-Z_t\|^2\,\mathrm{d}t
\\
&\le
K_BK_Fa_{p-1}^2+K_Ba_{p-1}^2
=
K_B(1+K_F)a_{p-1}^2.
\end{aligned}
\]
Applying \eqref{eq:error_exact_reference_rate} with \(p-1\) proves \eqref{eq:error_exact_decoupling_rate}.
Since \(q<1\), the right-hand side of~\eqref{eq:error_exact_decoupling_rate} converges to zero as \(p\to\infty\), and hence the decoupling error converges to zero.
\end{proof}

We are now ready to prove the convergence theorem.

\begin{proof}[Proof of Theorem~\ref{thm:error_convergence}]
By Assumption~\ref{ass:error_reference_contraction},
\(K_BK_F<1\). Choose \(\zeta>0\) such that
\(
q_\zeta
:=
(1+\zeta)K_BK_F
<1.
\)
For \(p\ge1\), set
\(
r_p^\pi
:=
\mathcal L_{T}^{\pi,p}
+
\mathcal R_{\rm term}^{\pi,p}
\)
and
\(
\delta_p^\pi
:=
|\pi|+r_p^\pi.
\)
Assumption~\ref{ass:error_residual_consistency} gives
\begin{equation}
\label{eq:error_convergence_joint_residual}
\lim_{\substack{p\to\infty\\|\pi|\to0}}
\delta_p^\pi
=
0.
\end{equation}

We first establish a uniform bound for the forward numerical errors.
By Lemma~\ref{lem:error_uniform_discrete_moment},
\(
\sup_{0\le t\le T}\mathbb{E}|X_t^{\pi,p}|^2
\le C_\star
\)
for every \(p\ge0\), uniformly with respect to \(\pi\).

For \(p\ge1\), the definition of \(\mathcal E_{dec}^p\) gives
\[
\sup_{0\le t\le T}\mathbb{E}|X_t^p|^2
\le
2\sup_{0\le t\le T}\mathbb{E}|X_t|^2
+
2\mathcal E_{dec}^p.
\]
Lemma~\ref{lem:error_exact_decoupling_convergence} and
Assumption~\ref{ass:error_wellposedness} therefore imply
\(
\sup_{p\ge1}
\sup_{0\le t\le T}
\mathbb{E}|X_t^p|^2
<\infty.
\)

For \(p=0\), by the definition of \(\mathfrak X_0^\pi\),
\[
\mathfrak X_0^\pi
\le
2\sup_{0\le t\le T}
\mathbb{E}|X_t^{\pi,0}|^2
+
2\sup_{0\le t\le T}
\mathbb{E}|\mathcal X_t^0|^2.
\]
Hence Lemma~\ref{lem:error_uniform_discrete_moment} and
Assumption~\ref{ass:error_wellposedness} imply
\(
\sup_\pi\mathfrak X_0^\pi<\infty.
\)
Consequently, there exists a constant \(M>0\), independent of \(p\) and \(\pi\), such that
\begin{equation}
\label{eq:error_convergence_uniform_forward_bound}
\mathfrak X_p^\pi\le M,
\qquad p\ge0.
\end{equation}
Define the joint upper limit by
\[
L
:=
\limsup_{\substack{p\to\infty\\|\pi|\to0}}
\mathfrak X_p^\pi
=
\inf_{\substack{P\ge1\\h>0}}
\sup_{\substack{p\ge P\\|\pi|\le h}}
\mathfrak X_p^\pi.
\]
By \eqref{eq:error_convergence_uniform_forward_bound}, \(L<\infty\).

We next estimate the averaged forward errors. For every fixed
\(J\ge1\) and \(p>J\),
\[
\frac1p\sum_{j=0}^{p-1}\mathfrak X_j^\pi
=
\frac1p\sum_{j=0}^{J-1}\mathfrak X_j^\pi
+
\frac1p\sum_{j=J}^{p-1}\mathfrak X_j^\pi
\le
\frac{JM}{p}
+
\sup_{j\ge J}\mathfrak X_j^\pi.
\]
Taking the joint upper limit and then the infimum over \(J\ge1\), we obtain
\begin{equation}
\label{eq:error_convergence_joint_average_bound}
\limsup_{\substack{p\to\infty\\|\pi|\to0}}
\frac1p\sum_{j=0}^{p-1}\mathfrak X_j^\pi
\le
\inf_{\substack{J\ge1\\h>0}}
\sup_{\substack{j\ge J\\|\pi|\le h}}
\mathfrak X_j^\pi
=
L.
\end{equation}
Taking the joint upper limit in \eqref{eq:error_uniform_forward_recursive_bound}, and using \eqref{eq:error_convergence_joint_residual} and \eqref{eq:error_convergence_joint_average_bound}, yields
\[
L
\le
q_\zeta
\limsup_{\substack{p\to\infty\\|\pi|\to0}}
\frac1p\sum_{j=0}^{p-1}\mathfrak X_j^\pi
\le
q_\zeta L.
\]
Since \(L\ge0\) and \(q_\zeta<1\), we have \(L=0\), and therefore
\(
\lim_{\substack{p\to\infty\\|\pi|\to0}}
\mathfrak X_p^\pi
=
0.
\)
Combining \(L=0\) with
\eqref{eq:error_convergence_joint_average_bound}, and using the nonnegativity of the averaged errors, we also obtain
\begin{equation}
\label{eq:error_convergence_joint_average_zero}
\lim_{\substack{p\to\infty\\|\pi|\to0}}
\frac1p\sum_{j=0}^{p-1}\mathfrak X_j^\pi
=
0.
\end{equation}

Finally, taking the joint upper limit in \eqref{eq:error_uniform_total_recursive_bound}, we obtain
\[
\limsup_{\substack{p\to\infty\\|\pi|\to0}}
\|(X^{\pi,p}-X,Y^{\pi,p}-Y,Z^{\pi,p}-Z)\|_{\mathcal E}^2
\le
\overline A_\zeta^{\rm tot}
\limsup_{\substack{p\to\infty\\|\pi|\to0}}
\delta_p^\pi
 +
\Gamma_\zeta
\limsup_{\substack{p\to\infty\\|\pi|\to0}}
\frac1p\sum_{j=0}^{p-1}\mathfrak X_j^\pi
+
2\lim_{p\to\infty}\mathcal E_{dec}^p.
\]
The three terms on the right-hand side vanish by \eqref{eq:error_convergence_joint_residual}, \eqref{eq:error_convergence_joint_average_zero}, and Lemma~\ref{lem:error_exact_decoupling_convergence}, respectively.
Hence~\eqref{eq:error_convergence_result} follows.
\end{proof}

\section{Numerical experiments}
\label{sec:numerical_experiments}

We evaluate the accuracy, stability, and efficiency of the deep truncated FBSDE method.
All experiments are implemented in Python with PyTorch on an Intel Core Ultra 9 275HX CPU, an NVIDIA GeForce RTX 5070 Laptop GPU (8\,GB), and 32\,GB RAM.
The reported runtimes include the entire training procedure.
Each subnetwork is fully connected with two hidden layers, using tanh activation without normalization for the proposed method.
We set the maximum number of decoupling steps to \(P=5\) for coupled FBSDEs and \(P=1\) for decoupled FBSDEs, with stopping tolerance \(\delta=10^{-5}\), throughout the experiments.
Each completed decoupling step performs one parameter update, so each outer iteration contains at most \(P\) updates.
For coupled problems, we therefore use one fifth as many outer iterations as the competing methods to ensure comparable maximum numbers of parameter updates.
The competing methods retain their original activation functions and normalization settings.
All methods use Adam~\cite{kingma2015adam} with batches of \(B=64\) stochastic paths, except for DBDP, whose additional training settings are specified in the corresponding examples due to its backward-induction structure.

\subsection{Example 1: scalar Burgers equation}

We consider the one-dimensional Burgers equation
\begin{equation}\label{eq:Burgers_1D}
\left\{
\begin{alignedat}{2}
&\textstyle u_t+uu_x=0, &&\quad \textstyle (t,x)\in(0,T]\times\mathbb R,\\
&\textstyle u(0,x)=\sin x.
\end{alignedat}
\right.
\end{equation}
To apply the proposed framework, we introduce the viscous approximation
\begin{equation*}
\left\{
\begin{alignedat}{2}
&\textstyle u_t^\varepsilon+u^\varepsilon u_x^\varepsilon
=\varepsilon u_{xx}^\varepsilon, &&\quad \textstyle (t,x)\in(0,T]\times\mathbb R,\\
&\textstyle u^\varepsilon(0,x)=\sin x,
\end{alignedat}
\right.
\end{equation*}
where \(\varepsilon>0\) and \(u^\varepsilon\) converges almost everywhere to the entropy solution as \(\varepsilon\downarrow0\)~\cite{MR47234}.
The time reversal \(v^\varepsilon(t,x):=u^\varepsilon(T-t,x)\) satisfies \(v_t^\varepsilon-v^\varepsilon v_x^\varepsilon+\varepsilon v_{xx}^\varepsilon=0\), with terminal condition \(v^\varepsilon(T,x)=\sin x\).

Keeping the nonlinear transport term in the generator gives \(b=0\), \(\sigma=\sqrt{2\varepsilon}\), and \(f(t,x,y,z)=-yz/\sqrt{2\varepsilon}\), yielding the decoupled FBSDE
\begin{equation}\label{eq:Form_Burgers_1D_decoupled}
\left\{
\begin{aligned}
X_t&=x+\int_0^t\sqrt{2\varepsilon}\,\mathrm{d}W_s,\\
Y_t&=\sin(X_T)-\int_t^T\frac{Y_sZ_s}{\sqrt{2\varepsilon}}\,\mathrm{d}s
-\int_t^TZ_s\,\mathrm{d}W_s.
\end{aligned}
\right.
\end{equation}
Alternatively, incorporating the transport term into the forward drift gives \(b=-y\), \(\sigma=\sqrt{2\varepsilon}\), and \(f=0\), yielding the coupled FBSDE
\begin{equation}\label{eq:Form_Burgers_1D_coupled}
\left\{
\begin{aligned}
X_t&=x-\int_0^tY_s\,\mathrm{d}s+\int_0^t\sqrt{2\varepsilon}\,\mathrm{d}W_s,\\
Y_t&=\sin(X_T)-\int_t^TZ_s\,\mathrm{d}W_s.
\end{aligned}
\right.
\end{equation}
Unlike the decoupled formulation, the coupled formulation avoids the singular factor \(\varepsilon^{-1/2}\) in the generator.
We therefore apply the proposed method to~\eqref{eq:Form_Burgers_1D_coupled}.
The entropy solution obtained from the characteristic representation serves as the reference, so the reported errors include both vanishing-viscosity and numerical approximation errors.

The first breaking time is \(T_b=1\)~\cite{MR1153252}.
We consider \(T=0.5,1.0,1.5\), covering the pre-shock, shock-formation, and post-shock regimes, with \(\varepsilon\in\{10^{-1},10^{-2},\ldots,10^{-8}\}\).
Although our method supports training on spatial intervals, existing methods primarily target pointwise estimation.
For a fair comparison, all methods are evaluated at \(256\) uniformly spaced points in \([0,2\pi]\), each treated as a separate initial state.

\begin{figure}[tbp]
\centering
\includegraphics[width=1.0\linewidth]{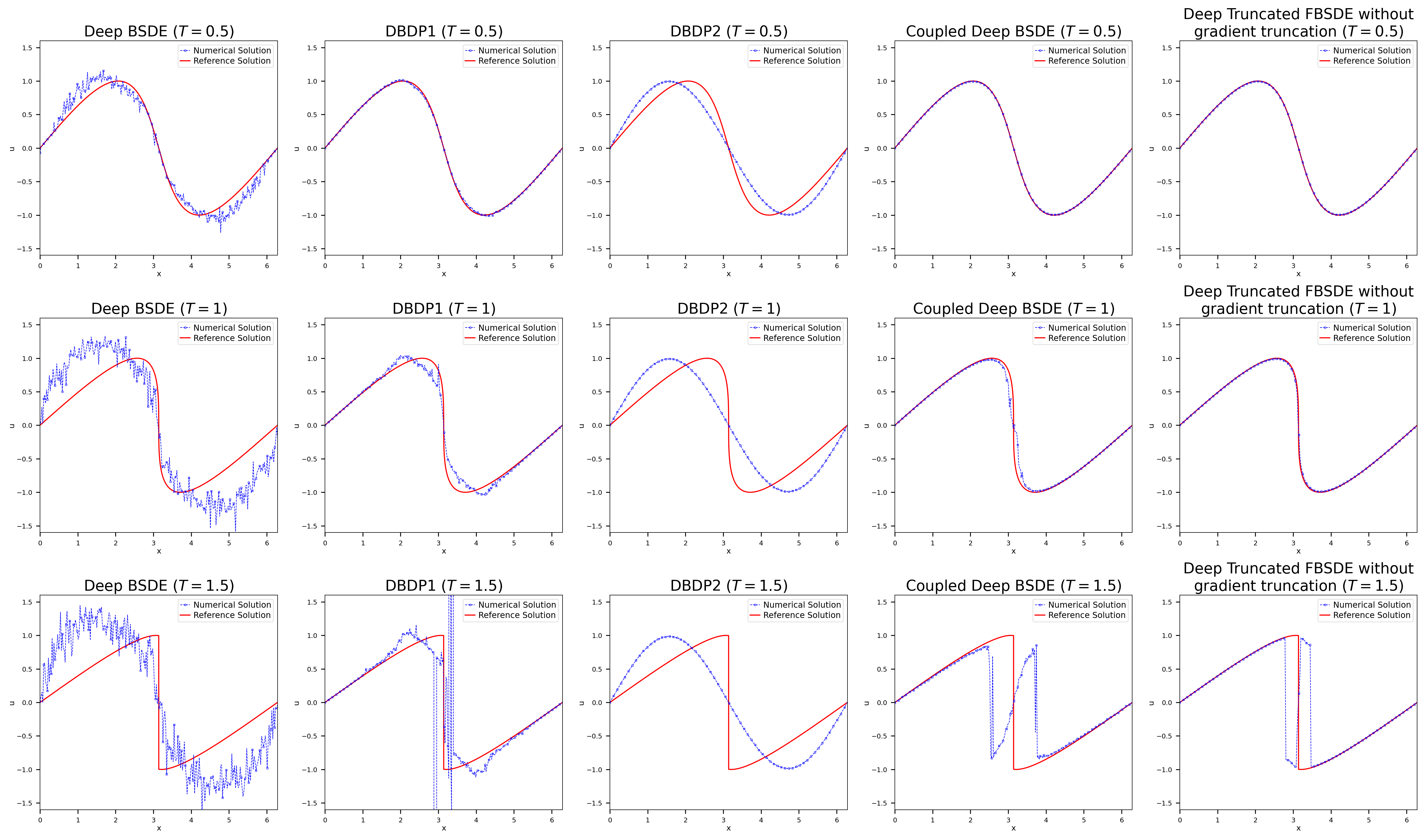}
\caption{Comparison for the viscous Burgers equation with \(\varepsilon=10^{-2}\).
The rows correspond to \(T=0.5,1.0,1.5\).
The columns show deep BSDE, DBDP1, and DBDP2 using the decoupled formulation, followed by coupled deep BSDE and the proposed method without gradient truncation using the coupled formulation.}
\label{Fig:comparison_Burgers_1D}
\end{figure}

We use \(N=10,20,30\) for \(T=0.5,1.0,1.5\), respectively, with hidden-layer width \(10\), and initialize the trainable initial value of \(Y\) according to the terminal condition.
The proposed method uses \(K=300\) outer iterations with at most \(1500\) parameter updates and Adam learning rates \(10^{-2},10^{-3},10^{-4}\) over successive blocks of \(100\) iterations.
The variant without gradient truncation uses the same settings but removes the stop-gradient operation.
The deep BSDE and coupled deep BSDE methods use the same discretization and network size, with \(1500\) iterations and the same learning-rate schedule over blocks of \(500\) iterations.
The former uses~\eqref{eq:Form_Burgers_1D_decoupled}, while the latter uses~\eqref{eq:Form_Burgers_1D_coupled}.
DBDP1 and DBDP2 use~\eqref{eq:Form_Burgers_1D_decoupled}, with \(50\,000\) fixed forward paths, batch size \(1000\) as in~\cite{MR4081911}, and \(150\) epochs per local subnetwork, corresponding to \(7500\) Adam updates, following the same three-stage learning-rate schedule.

Figure~\ref{Fig:comparison_Burgers_1D} shows that deep BSDE and DBDP1 lose accuracy and exhibit oscillations near the shock, reflecting the difficulty caused by the singular generator.
The coupled deep BSDE method also deteriorates in the post-shock regime.
For DBDP2, small \(\varepsilon\) concentrates the forward samples, limiting spatial information and potentially underestimating the reconstructed \(Z\).
The variant without gradient truncation remains accurate before and around shock formation but develops a mismatch in the post-shock regime, illustrating the stabilizing role of gradient truncation.

\begin{figure}[tbp]
\centering
\includegraphics[width=0.8\linewidth]{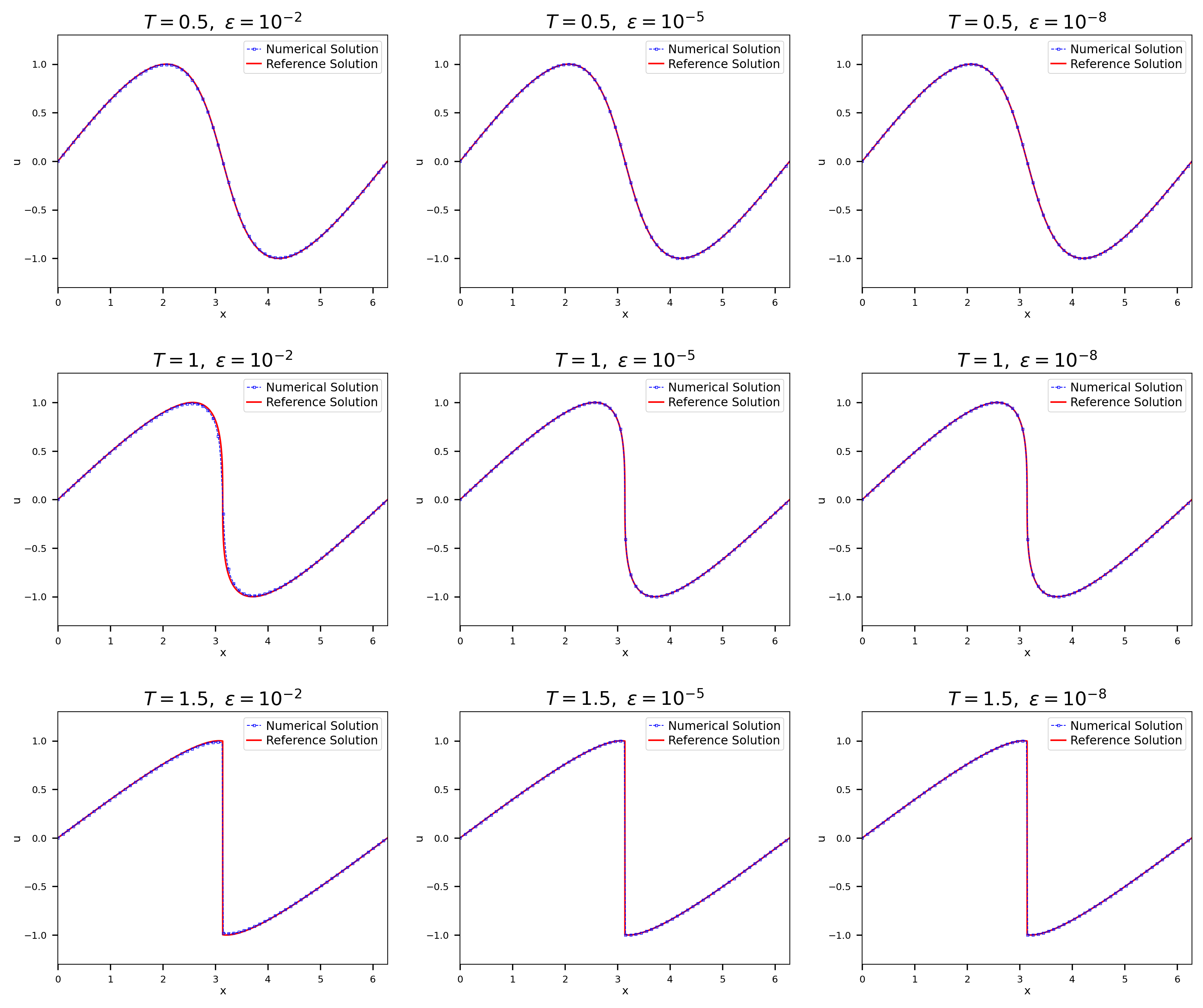}
\caption{Vanishing-viscosity approximations obtained by the proposed method using the coupled formulation. The rows correspond to \(T=0.5,1.0,1.5\), and the columns to \(\varepsilon=10^{-2},10^{-5},10^{-8}\).}
\label{Fig:Algorithm_10_Burgers_1D}
\end{figure}

\begin{table}[bp]
\centering
\caption{Discrete \(L^2\) and \(L^\infty\) errors of the proposed method relative to the entropy solution of~\eqref{eq:Burgers_1D}.}
\label{Tab:Error_eps_Algorithm_10_Burgers_1D_FBSDE}
\renewcommand{\arraystretch}{0.9}
\begin{tabular}{c ccc ccc}
\toprule
\multirow{2}{*}{$\varepsilon$}
& \multicolumn{3}{c}{$L^2$ error}
& \multicolumn{3}{c}{$L^\infty$ error}\\
\cmidrule(lr){2-4}\cmidrule(lr){5-7}
& $T=0.5$ & $T=1.0$ & $T=1.5$
& $T=0.5$ & $T=1.0$ & $T=1.5$\\
\midrule
$10^{-1}$ & 1.01E-01 & 3.43E-01 & 4.31E-01
          & 7.46E-02 & 4.80E-01 & 9.61E-01\\
$10^{-2}$ & 1.07E-02 & 7.98E-02 & 1.33E-02
          & 8.37E-03 & 2.67E-01 & 1.96E-02\\
$10^{-3}$ & 1.39E-03 & 4.78E-03 & 1.63E-03
          & 2.08E-03 & 1.92E-02 & 1.72E-03\\
$10^{-4}$ & 4.04E-04 & 6.80E-04 & 2.95E-04
          & 1.20E-03 & 3.03E-03 & 4.08E-04\\
$10^{-5}$ & 1.34E-04 & 1.89E-04 & 9.83E-05
          & 3.37E-04 & 7.53E-04 & 1.44E-04\\
$10^{-6}$ & 2.78E-05 & 7.80E-05 & 2.23E-05
          & 1.09E-04 & 4.27E-04 & 9.32E-05\\
$10^{-7}$ & 6.95E-06 & 1.62E-05 & 6.21E-06
          & 1.68E-05 & 8.02E-05 & 9.69E-06\\
$10^{-8}$ & \textbf{3.28E-06} & \textbf{7.13E-06}
          & \textbf{3.57E-06} & \textbf{3.77E-06}
          & \textbf{2.94E-05} & \textbf{3.99E-06}\\
\bottomrule
\end{tabular}

\end{table}

Figure~\ref{Fig:Algorithm_10_Burgers_1D} and Table~\ref{Tab:Error_eps_Algorithm_10_Burgers_1D_FBSDE} show that the proposed method remains accurate across all three regimes, with errors decreasing as \(\varepsilon\) decreases.
At \(T=1.5\) and \(\varepsilon=10^{-8}\), both \(L^2\) and \(L^\infty\) errors are of order \(10^{-6}\), demonstrating stable approximation in the nearly inviscid post-shock regime.

\subsection[Example 2: 100-dimensional spatially heterogeneous Burgers-type equation]{Example 2: $100$-dimensional spatially heterogeneous Burgers-type equation}

We extend the Burgers-type equation in Section~4.5 of~\cite{MR3736669} by introducing coordinate-dependent oscillations while preserving an explicit classical solution.
Let \(\Psi(x):=d^{-1}\sum_{i=1}^d(a_ix_i+b_i\sin(c_ix_i))\) and \(g(x):=[1+\exp(-T-\Psi(x))]^{-1}\).
Define \(H(x):=\sum_{i=1}^d(a_i+b_ic_i\cos(c_ix_i))\), \(M(x):=\sum_{i=1}^d(a_i+b_ic_i\cos(c_ix_i))^2\), and \(R(x):=-\sum_{i=1}^d b_ic_i^2\sin(c_ix_i)\).
Set \(C(x,y):=[M(x)y-1-dR(x)/2-M(x)/2]/H(x)\).
We consider the spatially heterogeneous Burgers-type equation
\begin{equation}
\label{eq:heterogeneous_burgers_type_pde}
\left\{
\begin{alignedat}{2}
&\textstyle u_t+\frac{d^2}{2}\Delta u+d\,C(x,u)\sum_{i=1}^d u_{x_i}=0,
&&\quad \textstyle (t,x)\in[0,T)\times\mathbb R^d,\\
&\textstyle u(T,x)=g(x),
&&\quad \textstyle x\in\mathbb R^d.
\end{alignedat}
\right.
\end{equation}
Provided \(H(x)\neq0\), an explicit classical solution is \(u(t,x)=[1+\exp(-t-\Psi(x))]^{-1}\).

For the decoupled representation, we set \(b=0\), \(\sigma=dI_d\), and \(f(t,x,y,z)=C(x,y)\sum_{i=1}^d z^i\), identifying \(z\in\mathbb R^{1\times d}\) with a vector in \(\mathbb R^d\).
The corresponding decoupled FBSDE is
\begin{equation*}
\left\{
\begin{aligned}
X_t&=x+\int_0^t dI_d\,\mathrm{d}W_s,\\
Y_t&=g(X_T)+\int_t^T C(X_s,Y_s)\sum_{i=1}^d Z_s^i\,\mathrm{d}s
-\int_t^T Z_s\,\mathrm{d}W_s.
\end{aligned}
\right.
\end{equation*}
Alternatively, incorporating the convection term into the drift gives \(b(t,x,y,z)=d\,C(x,y)\mathbf 1_d\) and \(f=0\), where \(\mathbf 1_d=(1,\ldots,1)^\top\).
The coupled FBSDE becomes
\begin{equation}\label{eq:heterogeneous_burgers_type_coupled_fbsde}
\left\{
\begin{aligned}
X_t&=x+\int_0^t d\,C(X_s,Y_s)\mathbf 1_d\,\mathrm{d}s
+\int_0^t dI_d\,\mathrm{d}W_s,\\
Y_t&=g(X_T)-\int_t^T Z_s\,\mathrm{d}W_s.
\end{aligned}
\right.
\end{equation}
The decoupled formulation introduces explicit dependence on \(\sum_{i=1}^d Z^i\) in the backward generator, potentially increasing sensitivity to errors in \(Z\) in high dimension.
The coupled formulation avoids this dependence by moving the convection term into the forward drift.
We therefore use the coupled formulation~\eqref{eq:heterogeneous_burgers_type_coupled_fbsde}.

We choose \(a_i=1+i/d\), \(b_i=i/(2d)\), and \(c_i=i/d\), which ensure \(H(x)>0\) for all \(x\in\mathbb R^d\).
We take \(d=100\), \(T=1\), \(x_0=(0,\ldots,0)\), \(N=100\), and hidden-layer width \(110\).
The trainable initial value of \(Y\) is independently sampled from the uniform distribution on \([2,4]\).
The proposed method uses \(K=2000\) outer iterations with at most \(10000\) parameter updates and Adam learning rates \(5\times10^{-3}\) and \(5\times10^{-4}\) over successive blocks of \(1000\) iterations.
We report statistics from ten independent runs.

\begin{figure}[tbp]
    \centering
    \begin{minipage}[t]{0.45\textwidth}
        \centering
        \includegraphics[width=\linewidth]
        {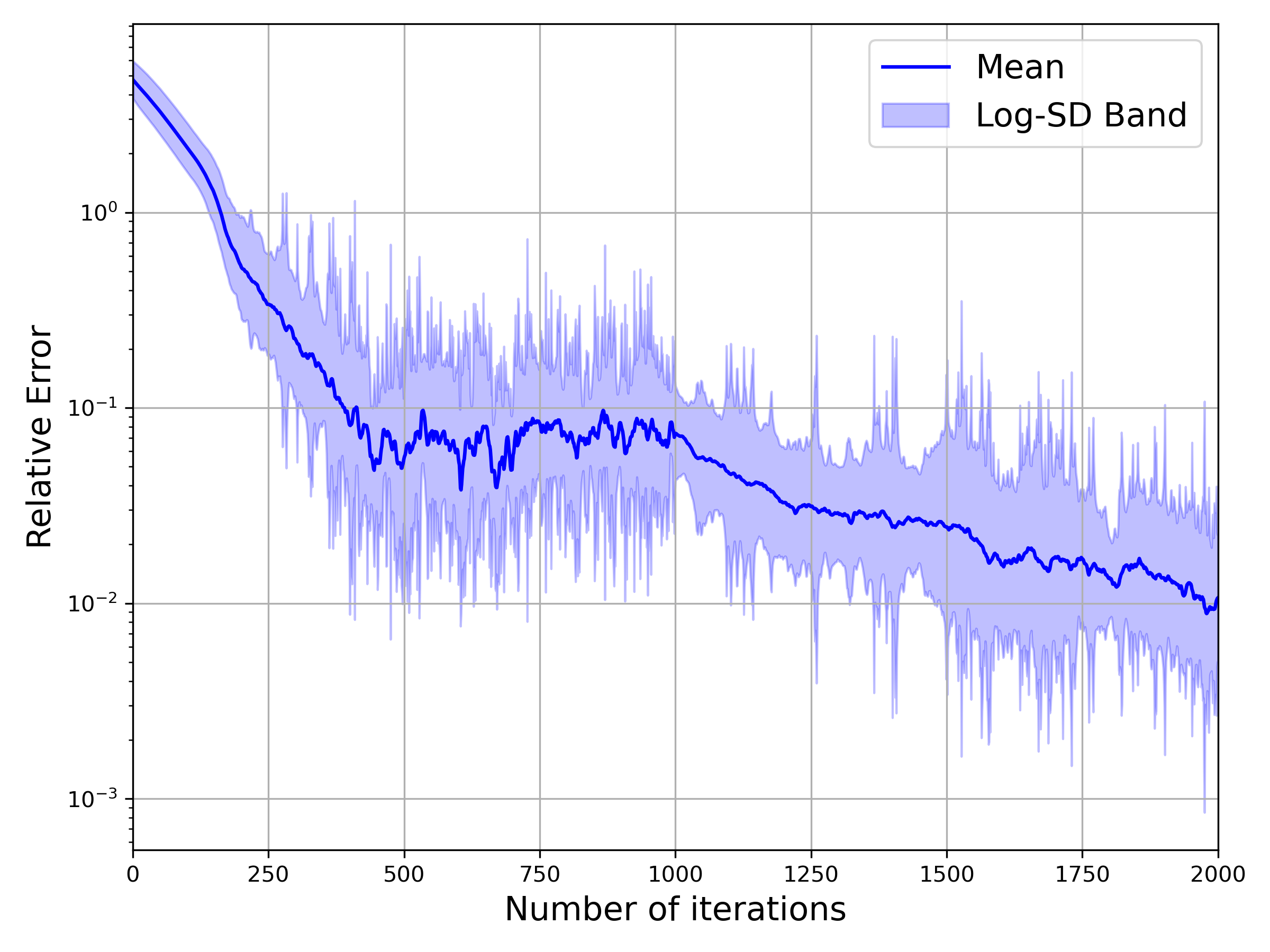}
    \end{minipage}\hspace{0.0\textwidth}\begin{minipage}[t]{0.45\textwidth}
        \centering
        \includegraphics[width=\linewidth]
        {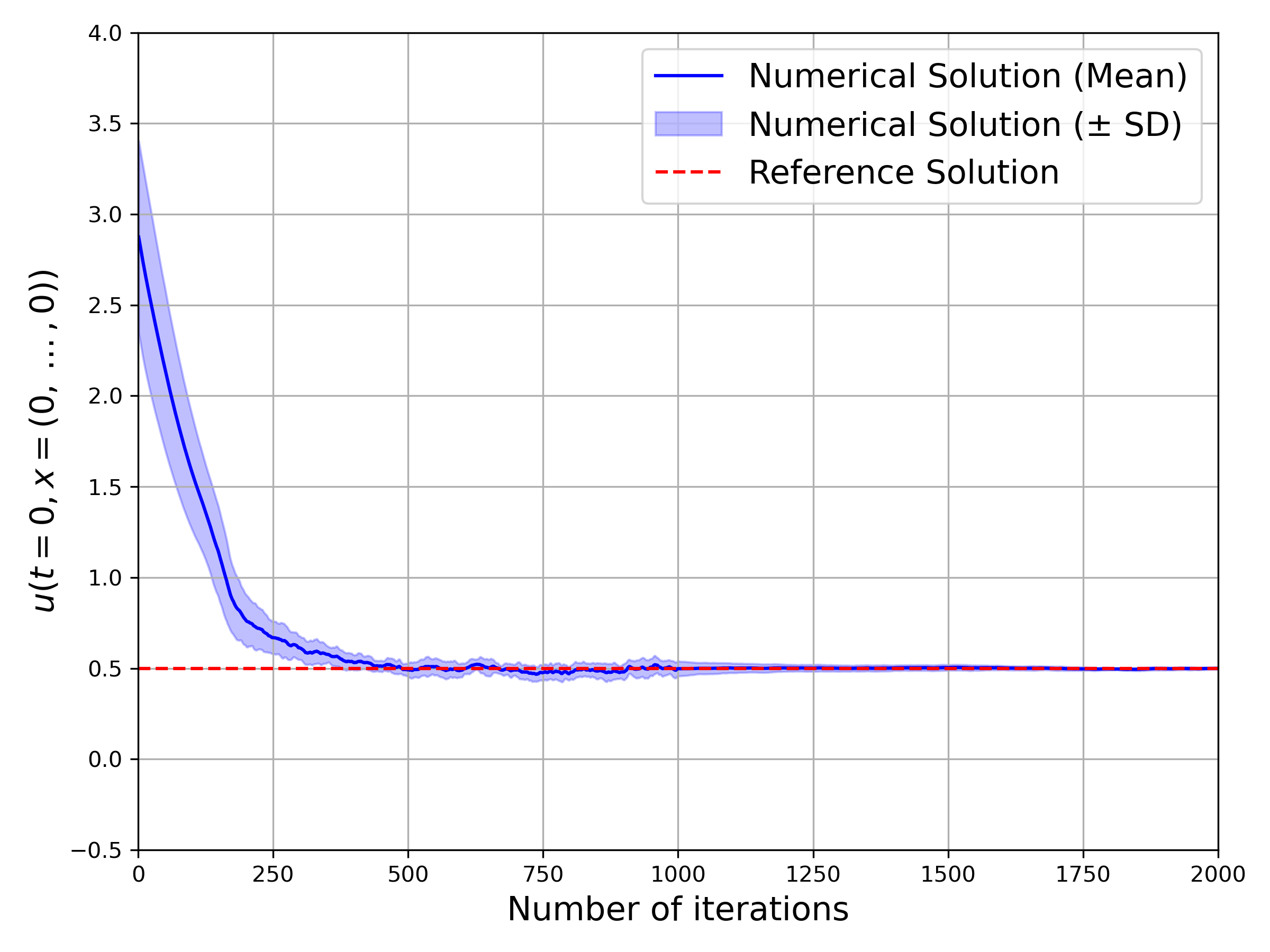}
    \end{minipage}
    \vspace{0.01cm}
    \caption{Deep truncated FBSDE method for the \(100\)-dimensional spatially heterogeneous Burgers-type equation~\eqref{eq:heterogeneous_burgers_type_pde}.
Left: relative error during training, with the Log-SD Band showing the standard deviation of log-transformed errors.
Right: convergence of the numerical solution at \(t=0,\ x=(0,\dots,0)\).
Results are based on ten independent runs.}
\label{Fig:Algorithm_10_Heterogeneous_Burgers_Type_FBSDE}
\end{figure}

Figure~\ref{Fig:Algorithm_10_Heterogeneous_Burgers_Type_FBSDE} shows a rapid initial decrease in relative error followed by sustained improvement.
Across ten runs, the final mean relative error and its standard deviation are \(1.07\times10^{-2}\) and \(7.85\times10^{-3}\), respectively.
The numerical solution has mean \(0.499310\) and standard deviation \(6.59\times10^{-3}\), compared with the exact value \(0.5\).

For comparison, deep BSDE, coupled deep BSDE, and Algorithms 2 and 3 of Ji et al.\ use the same discretization and network size, with \(10000\) iterations and Adam learning rates \(5\times10^{-3}\) and \(5\times10^{-4}\) over the two halves of training.
DBDP1 and DBDP2 use \(50\,000\) fixed forward paths, batch size \(1000\), and \(1000\) epochs per local subnetwork, corresponding to \(50\,000\) Adam updates, with the same learning-rate schedule.
The deep BSDE, coupled deep BSDE, and Ji et al.\ methods fail to provide reliable approximations under their respective formulations.
The DBDP methods, applied to the decoupled formulation, also exhibit instability despite substantially heavier training.
By contrast, the proposed method achieves a mean relative error close to \(1\%\), demonstrating improved accuracy for this high-dimensional problem.

\subsection{Example 3: two-dimensional Allen--Cahn equation}

We consider the Allen--Cahn equation on the \(2\pi\)-periodic torus \(\mathbb T^2:=[0,2\pi)^2\)
\begin{equation}\label{eq:allen-cahn-2d-forward}
\left\{
\begin{alignedat}{2}
&\textstyle u_t=\varepsilon^2\Delta u+u-u^3,
&&\quad \textstyle (t,x)\in(0,T]\times\mathbb T^2,\\
&\textstyle u(0,x)=u_0(x),
&&\quad \textstyle x\in\mathbb T^2.
\end{alignedat}
\right.
\end{equation}
The time reversal \(v(t,x):=u(T-t,x)\) satisfies \(v_t+\varepsilon^2\Delta v+v-v^3=0\), with terminal condition \(v(T,x)=u_0(x)\).
The corresponding decoupled FBSDE, with \(\sigma=\sqrt{2}\varepsilon I_2\) and \(f(t,x,y,z)=y-y^3\), is
\begin{equation*}
\left\{
\begin{aligned}
X_t&=x+\int_0^t\sqrt{2}\varepsilon I_2\,\mathrm{d}W_s,\\
Y_t&=u_0(X_T)+\int_t^T(Y_s-Y_s^3)\,\mathrm{d}s
-\int_t^TZ_s\,\mathrm{d}W_s.
\end{aligned}
\right.
\end{equation*}
All methods use this decoupled formulation.

We consider two initial conditions.
The first is a radially symmetric ring centered at \((\pi,\pi)\), defined by
\(u_0(x_1,x_2)
=\tanh\bigl((\sqrt{(x_1-\pi)^2+(x_2-\pi)^2}-2)/(\sqrt{2}\,\varepsilon)\bigr)\).
The second consists of seven compactly supported components
\begin{equation*}
    u_0(x_1,x_2)
    =
    -1+\sum_{i=1}^{7}
    f_0\,(\sqrt{(x_1-x_1^i)^2+(x_2-x_2^i)^2}-r^i),\qquad
    f_0(s)=
    \begin{cases}
        2\,e^{-\varepsilon^{2}/s^{2}}, & s<0,\\[0.25em]
        0, & s\ge 0.
    \end{cases}
\end{equation*}
The centers and radii are given in Table~\ref{tab:allen_cahn_ball}.
\begin{table}[!htbp]
\centering
\caption{Centers and radii for the seven-disk initial condition.}
\label{tab:allen_cahn_ball}
\begin{tabular}{c ccccccc}
\toprule
$i$ & $1$ & $2$ & $3$ & $4$ & $5$ & $6$ & $7$\\
\midrule
$x_1^i$ & $\pi/2$ & $\pi/4$ & $\pi/2$ & $\pi$ & $3\pi/2$ & $\pi$ & $3\pi/2$\\
$x_2^i$ & $\pi/2$ & $3\pi/4$ & $5\pi/4$ & $\pi/4$ & $\pi/4$ & $\pi$ & $3\pi/2$\\
$r^i$ & $\pi/5$ & $2\pi/15$ & $2\pi/15$ & $\pi/10$ & $\pi/10$ & $\pi/4$ & $\pi/4$\\
\bottomrule
\end{tabular}
\end{table}
\begin{figure}[tbp]
\centering
\includegraphics[width=\linewidth]{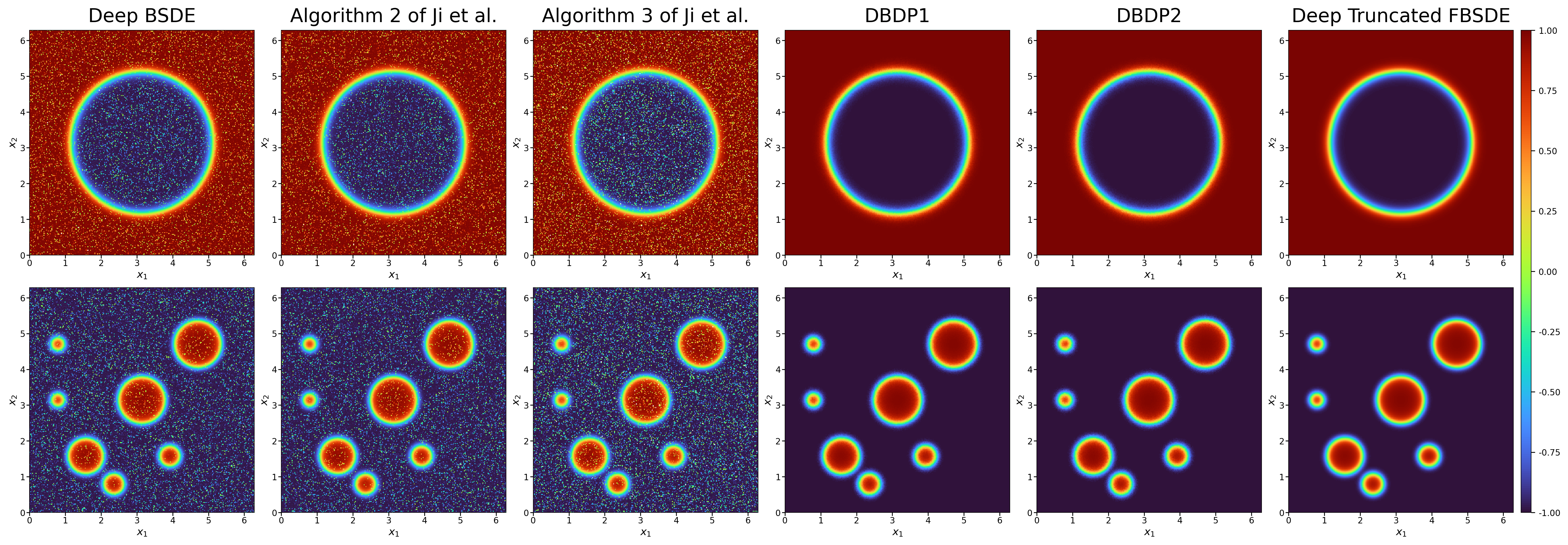}
\caption{Comparison for the 2D Allen--Cahn equation~\eqref{eq:allen-cahn-2d-forward}. The rows correspond to the ring and seven-disk initial conditions. The columns show deep BSDE, Algorithms 2 and 3 of Ji et al., DBDP1, DBDP2, and the proposed method.}
\label{Fig:Comparison_Allen_Cahn_2D}
\end{figure}

We take \(d=2\), \(\varepsilon=0.1\), \(T=0.3\), \(N=10\), and hidden-layer width \(10\), with the trainable initial value of \(Y\) initialized to \(0\).
The proposed method uses \(K=300\) iterations and Adam learning rate \(5\times10^{-3}\).
All methods are evaluated independently at the points of a uniform \(256\times256\) grid over \([0,2\pi]^2\).
The reference solution is computed using a semi-implicit Fourier spectral method.
Deep BSDE and Algorithms 2 and 3 of Ji et al.\ use the same time discretization, network size, and learning rate.
DBDP methods use \(10\,000\) forward paths, batch size \(1000\), and \(30\) epochs per local subnetwork, corresponding to \(300\) Adam updates, with learning rate \(5\times10^{-3}\).

\begin{table}[bp]
\centering
\caption{Relative errors for the ring and seven-disk initial conditions.}
\label{Tab:Comparisons_Allen_Cahn_2D_Initial_Conditions}
\begin{tabular}{lcccc}
\toprule
\multirow{2}{*}{\textbf{Method}}
& \multicolumn{2}{c}{\textbf{Ring}}
& \multicolumn{2}{c}{\textbf{Seven-disk}}\\
\cmidrule(lr){2-3}\cmidrule(lr){4-5}
& Rel. $L^2$ error & Rel. $L^\infty$ error
& Rel. $L^2$ error & Rel. $L^\infty$ error\\
\midrule
DBDP1 & 1.03E-02 & 1.36E-01 & 1.18E-02 & 1.46E-01\\
DBDP2 & 1.12E-02 & 1.45E-01 & 1.32E-02 & 1.57E-01\\
Deep truncated FBSDE
& \textbf{4.54E-03} & \textbf{4.76E-02}
& \textbf{7.44E-03} & \textbf{8.21E-02}\\
\bottomrule
\end{tabular}
\end{table}

Figure~\ref{Fig:Comparison_Allen_Cahn_2D} shows numerical instability in deep BSDE and Algorithms 2 and 3 of Ji et al., while DBDP1 and DBDP2 recover the main phase structures with heavier training.
The proposed method accurately captures the interfaces for both initial conditions and achieves the smallest relative errors among the methods reported in Table~\ref{Tab:Comparisons_Allen_Cahn_2D_Initial_Conditions}.
These results suggest that nonlinear Feynman--Kac reconstruction and pathwise consistency can improve performance even without forward--backward coupling.

\subsection[Example 4: 100-dimensional Allen--Cahn equations]{Example 4: $100$-dimensional Allen--Cahn equations}

We consider two \(100\)-dimensional Allen--Cahn equations with double-well and logarithmic potentials.
Both admit decoupled FBSDE representations with \(b=0\) and \(\sigma=\sqrt{2}\varepsilon I_d\), given by
\begin{equation*}
\left\{
\begin{aligned}
X_t&=x+\int_0^t\sqrt{2}\varepsilon I_d\,\mathrm{d}W_s,\\
Y_t&=g(X_T)+\int_t^T f(s,X_s,Y_s,Z_s)\,\mathrm{d}s
-\int_t^T Z_s\,\mathrm{d}W_s.
\end{aligned}
\right.
\end{equation*}
We use \(d=100\), \(m=1\), \(\varepsilon=0.1\), \(T=1\), \(N=100\), and hidden-layer width \(110\).
For all methods except DBDP, the trainable initial value of \(Y\) is independently initialized from the uniform distribution on \([-0.1,0.1]\) in each run.
These methods use \(K=1000\) iterations with Adam learning rates \(5\times10^{-3}\) and \(5\times10^{-4}\) over the two halves of training.
DBDP1 and DBDP2 use \(10\,000\) forward paths, batch size \(1000\), and \(100\) epochs per local subnetwork, corresponding to \(1000\) Adam updates, with the same learning-rate schedule.
All statistics are based on ten independent runs.

\begin{figure}[tbp]
    \centering

    \includegraphics[width=0.45\textwidth]
    {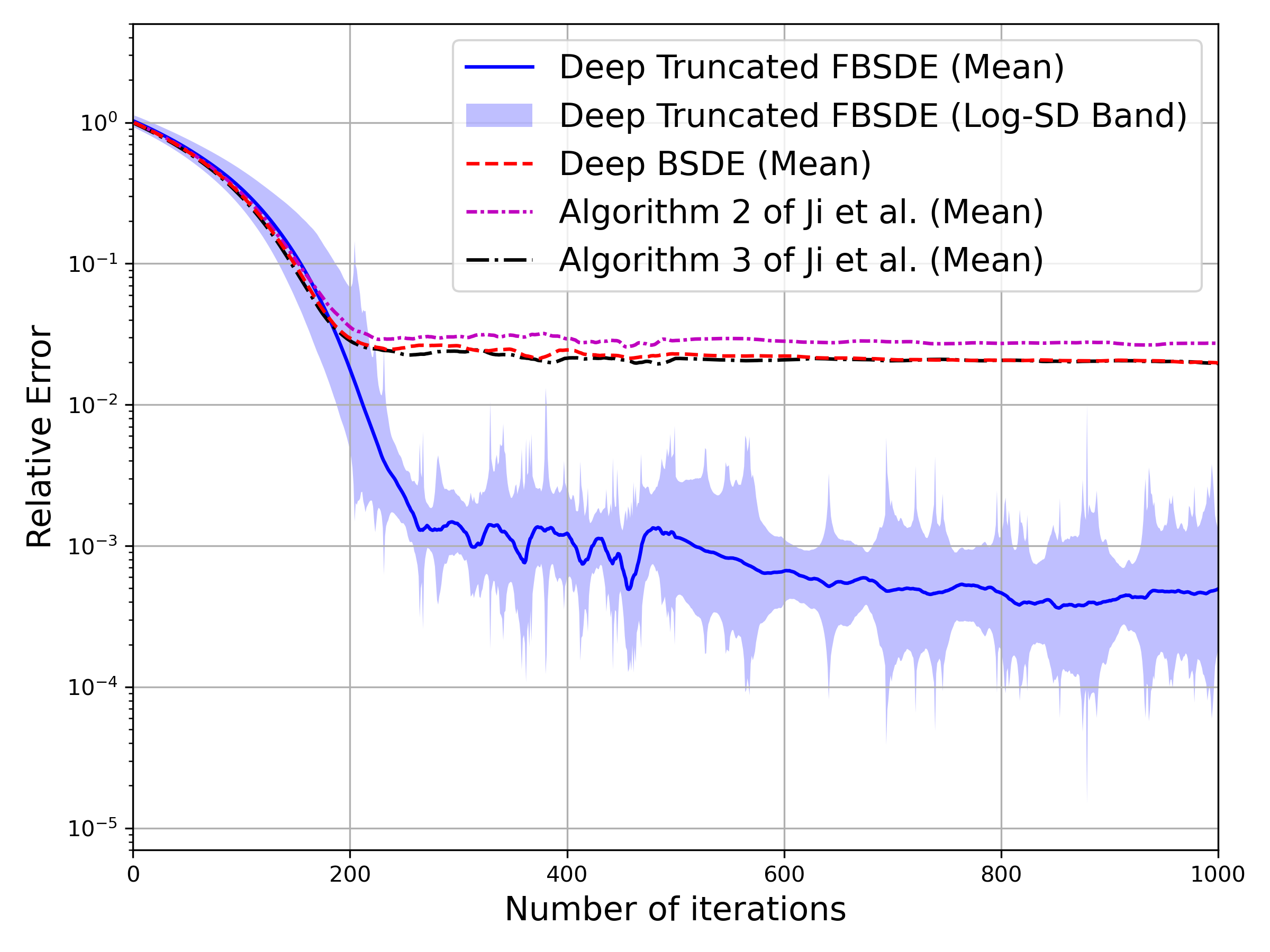}
    \hspace{0.0\textwidth}
    \includegraphics[width=0.45\textwidth]
    {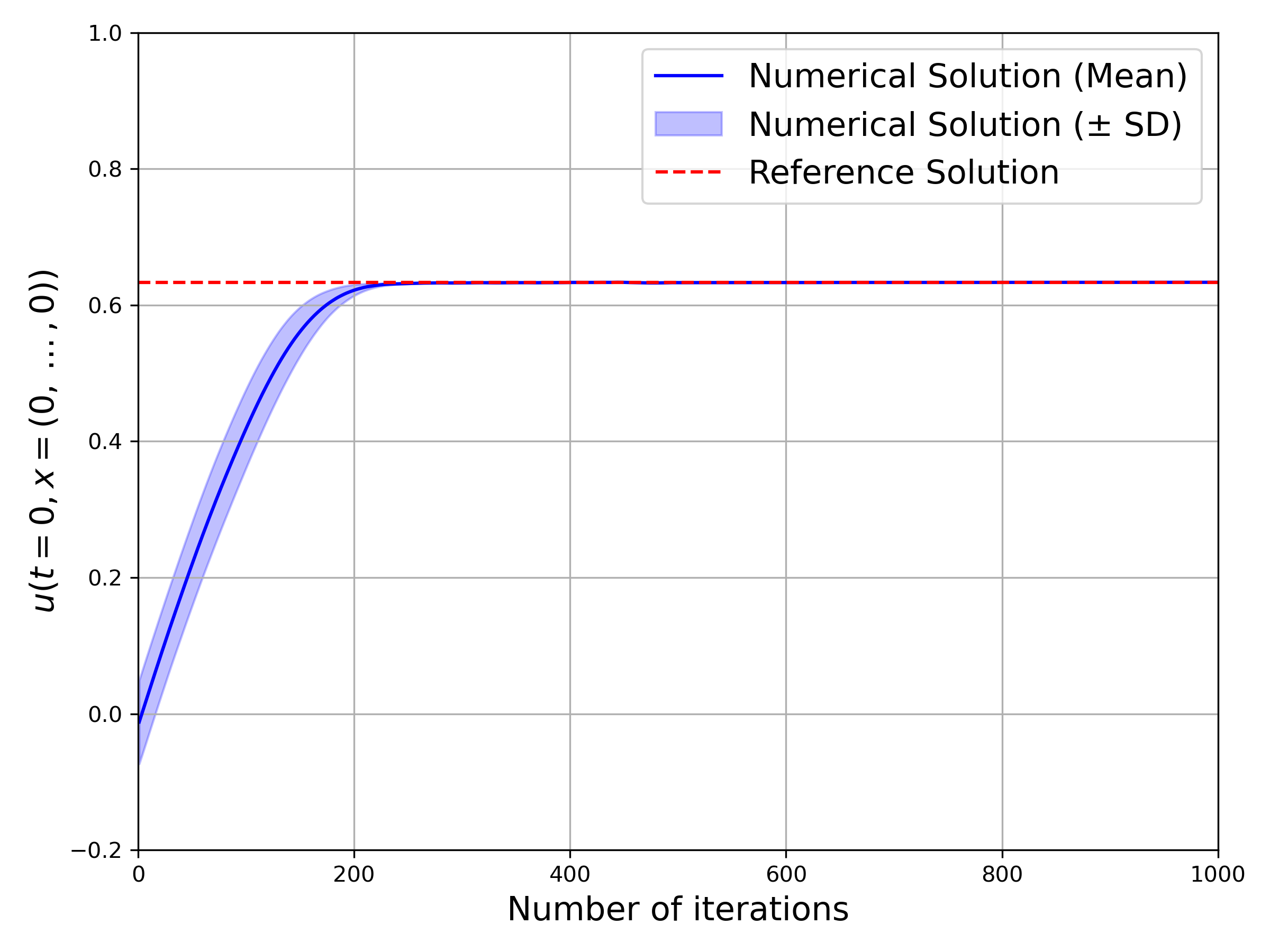}

    \vspace{0.6em}

    \includegraphics[width=0.45\textwidth]
    {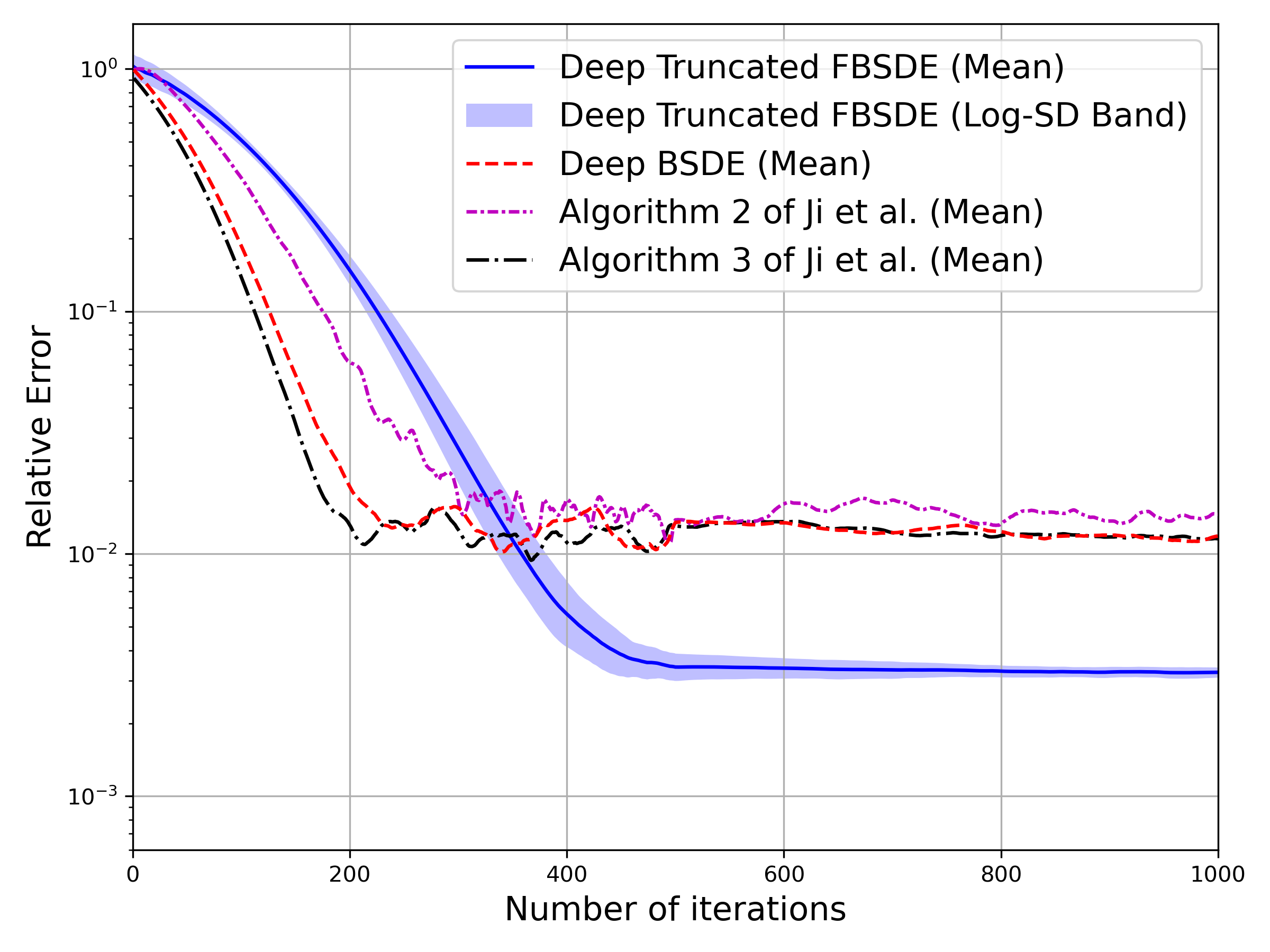}
    \hspace{0.0\textwidth}
    \includegraphics[width=0.45\textwidth]
    {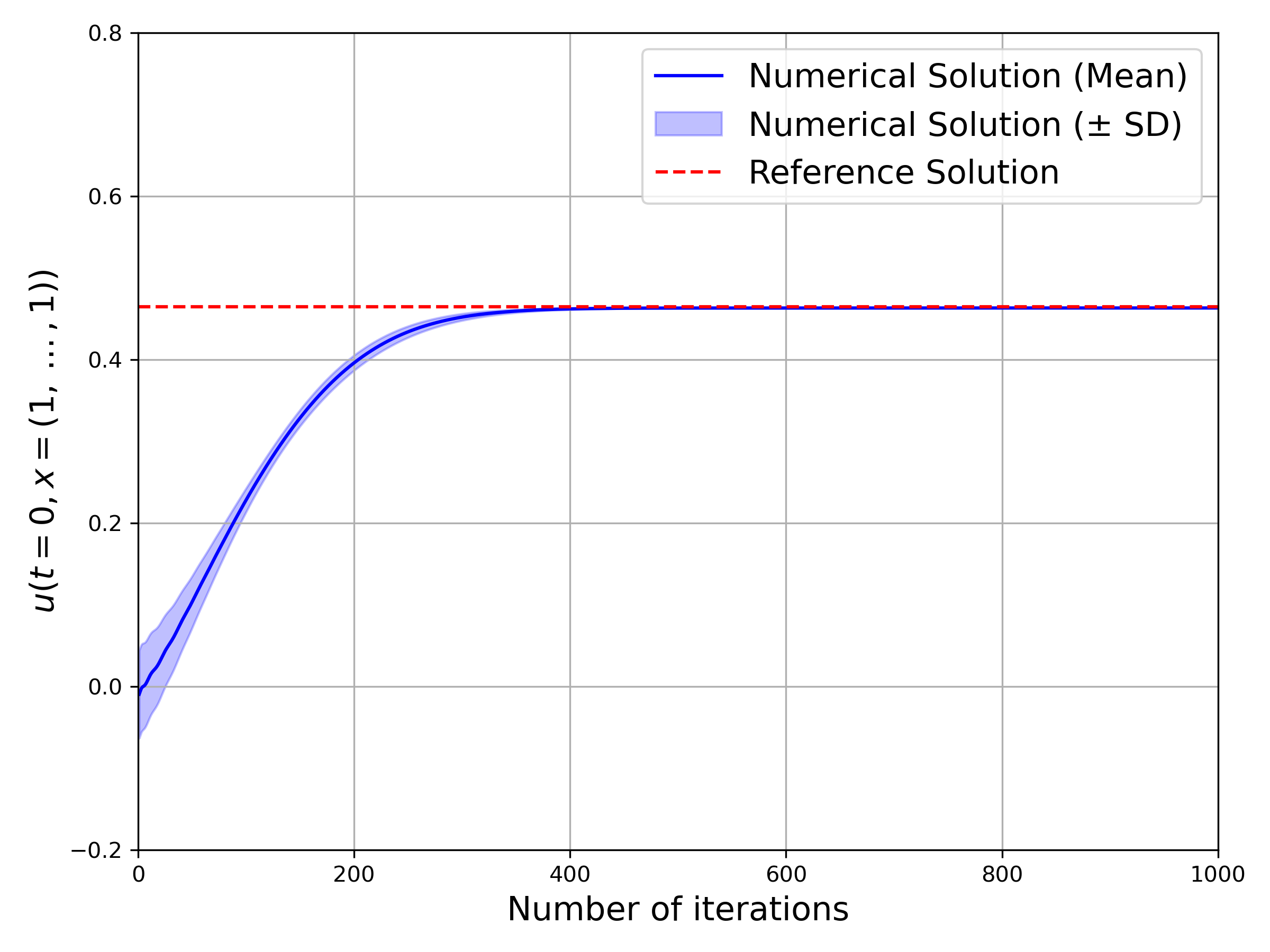}

    \caption{Comparison for the two \(100\)-dimensional Allen--Cahn equations.
    The first row corresponds to the double-well potential~\eqref{eq:allen-cahn-100d-double} at \(t=0,\ x=(0,\ldots,0)\), and the second row to the logarithmic potential~\eqref{eq:allen-cahn-100d-log} at \(t=0,\ x=(1,\ldots,1)\).
    In each row, the left panel shows the relative errors of different methods, and the right panel shows the convergence of the deep truncated FBSDE solution.}
\label{Fig:Allen_Cahn_100D}
\end{figure}

We first consider the Allen--Cahn equation with a double-well potential
\begin{equation}\label{eq:allen-cahn-100d-double}
\left\{
\begin{alignedat}{2}
&u_t+\varepsilon^2\Delta u+u-u^3=0,
&&\quad (t,x)\in[0,T)\times\mathbb R^d,\\
&u(T,x)=g(x),
&&\quad x\in\mathbb R^d.
\end{alignedat}
\right.
\end{equation}
The terminal condition is \(g(x)=(2+d^{-1}\sum_{i=1}^d a_ix_i^2)^{-1}\exp[-(d-1)^{-1}\sum_{i=1}^{d-1}b_i(x_{i+1}-x_i)^2]\), where \(a_i=i\) and \(b_i=\ln(1+i)\).
Compared with the benchmark \(g(x)=1/(2+0.4\|x\|^2)\) in~\cite{MR3847747}, this condition introduces spatial heterogeneity and interactions between neighboring coordinates.
The corresponding generator is \(f(t,x,y,z)=y-y^3\).
At \(t=0\) and \(x=(0,\ldots,0)\), the branching diffusion method~\cite{MR3138609} with \(10^9\) Monte Carlo samples gives the reference value \(0.6330577\).

We next consider the Allen--Cahn equation with logarithmic potential
\begin{equation}\label{eq:allen-cahn-100d-log}
\left\{
\begin{alignedat}{2}
&u_t+\varepsilon^2\Delta u+\theta_cu
-\frac{\theta}{2}\ln\frac{1+u}{1-u}-h(t,x)=0,
&&\quad (t,x)\in[0,T)\times\mathbb R^d,\\
&u(T,x)=g(x),
&&\quad x\in\mathbb R^d.
\end{alignedat}
\right.
\end{equation}
We assume \(0<\theta<\theta_c\).
To allow exact error evaluation, we prescribe
\(u(t,x)=\lambda\exp[\frac12\cos(T-t)-d^{-1}\sum_{i=1}^d a_ix_i^2]\,[d^{-1}\sum_{i=1}^d\cos(b_ix_i)]\).
We set \(g(x)=u(T,x)\) and determine the forcing term \(h(t,x)\) analytically by substitution.
The corresponding generator is \(f(t,x,y,z)=\theta_cy-\frac{\theta}{2}\ln\frac{1+y}{1-y}-h(t,x)\).
We choose \(\theta_c=1\), \(\theta=1/4\), \(\lambda=1/2\), \(a_i=i/(3d)\), and \(b_i=i/d\).
The solution is evaluated at \(t=0\) and \(x=(1,\ldots,1)\) to obtain a nontrivial reference value.

\begin{table}[tbp]
\centering
\caption{Comparison of the numerical solution, relative error, and training time for the two $100$-dimensional Allen--Cahn equations, based on ten independent runs.}
\label{Tab:Comparisons_Allen_Cahn_100D}
\renewcommand{\arraystretch}{0.95}
\begin{tabular}{lcccccc}
\toprule
\multirow{2}{*}{\textbf{Method}}
& \multicolumn{2}{c}{\textbf{Numerical solution}}
& \multicolumn{2}{c}{\textbf{Relative error}}
& \multicolumn{2}{c}{\textbf{Time (s)}}\\
\cmidrule(lr){2-3}\cmidrule(lr){4-5}\cmidrule(lr){6-7}
& Mean & Std. Dev.
& Mean & Std. Dev.
& Mean & Std. Dev.\\
\midrule
\multicolumn{7}{l}{\textbf{Double-well potential}}\\
Deep BSDE
& 0.620526 & 7.37E-04
& 1.98E-02 & 1.16E-03
& 258.65 & 2.57\\
Algorithm 2 of Ji et al.
& 0.615803 & 9.52E-04
& 2.73E-02 & 1.50E-03
& 393.52 & 2.11\\
Algorithm 3 of Ji et al.
& 0.620638 & 8.47E-04
& 1.96E-02 & 1.34E-03
& 235.35 & 1.75\\
DBDP1
& 0.632918 & 1.15E-03
& 1.46E-03 & 1.11E-03
& 531.68 & 0.38\\
DBDP2
& 0.634921 & 1.46E-03
& 3.15E-03 & 2.03E-03
& 733.45 & 0.31\\
Deep truncated FBSDE
& 0.633371 & \textbf{2.23E-04}
& \textbf{4.95E-04} & \textbf{3.52E-04}
& 244.08 & 2.75\\
\midrule
\multicolumn{7}{l}{\textbf{Logarithmic potential}}\\
Deep BSDE
& 0.459049 & 1.07E-03
& 1.19E-02 & 2.31E-03
& 352.93 & 2.89\\
Algorithm 2 of Ji et al.
& 0.457799 & 1.67E-03
& 1.45E-02 & 3.60E-03
& 480.91 & 6.77\\
Algorithm 3 of Ji et al.
& 0.459194 & 1.18E-03
& 1.15E-02 & 2.55E-03
& 316.54 & 3.86\\
DBDP1
& 0.455162 & 3.71E-03
& 2.02E-02 & 7.98E-03
& 688.14 & 9.54\\
DBDP2
& 0.462486 & 3.37E-03
& 7.35E-03 & 4.27E-03
& 847.79 & 6.31\\
Deep truncated FBSDE
& 0.463047 & \textbf{7.14E-05}
& \textbf{3.24E-03} & \textbf{1.54E-04}
& 339.10 & 1.88\\
\bottomrule
\end{tabular}
\end{table}

Figure~\ref{Fig:Allen_Cahn_100D} and Table~\ref{Tab:Comparisons_Allen_Cahn_100D} show that the proposed method achieves the lowest mean relative error in both tests, together with the smallest standard deviations in the numerical solution and relative error, while maintaining training times comparable to the fastest methods.
For the double-well potential, its mean relative error is \(4.95\times10^{-4}\), compared with errors of order \(10^{-2}\) for deep BSDE and Algorithms 2 and 3 of Ji et al.\ and \(10^{-3}\) for DBDP1 and DBDP2.
For the logarithmic potential, the relative error decreases more slowly initially but reaches \(3.24\times10^{-3}\), compared with \(7.35\times10^{-3}\) for DBDP2 and errors above \(10^{-2}\) for the other methods.
These results demonstrate accuracy, stability, and computational efficiency across different nonlinear potentials in high-dimensional decoupled problems.

\subsection{Example 5: gene flow model}

The gene flow model describes the evolution of a trait proportion under diffusion, reaction, and transport induced by population-density gradients~\cite{MR4697915}.
Let \(\rho(t,x)\in[0,1]\) denote the trait proportion and \(\mathcal N(x,\rho)\) the total population density, depending on space and the trait proportion.
We consider
\begin{equation}\label{eq:gene_flow_pde_forward}
\left\{
\begin{alignedat}{2}
&\rho_t-A\Delta\rho
-2A\left\langle\nabla\ln\mathcal N(x,\rho),\nabla\rho\right\rangle=r(\rho),
&&\quad (t,x)\in(0,T]\times\mathbb R^2,\\
&\rho(0,x)=g(x),
&&\quad x\in\mathbb R^2,
\end{alignedat}
\right.
\end{equation}
where \(A>0\).
For simplicity, define \(a(x,\rho):=\nabla_x\mathcal N(x,\rho)/\mathcal N(x,\rho)\) and \(q(x,\rho):=\mathcal N_\rho(x,\rho)/\mathcal N(x,\rho)\).
Applying the time reversal \(u(t,x)=\rho(T-t,x)\) and the chain rule yields
\begin{equation}\label{eq:gene_flow_pde_backward}
\left\{
\begin{alignedat}{2}
&u_t+A\Delta u+2A\langle a(x,u),\nabla u\rangle+2Aq(x,u)|\nabla u|^2+r(u)=0,
&&\ (t,x)\in[0,T)\times\mathbb R^2,\\
&u(T,x)=g(x),
&&\ x\in\mathbb R^2.
\end{alignedat}
\right.
\end{equation}

Following the Wolbachia invasion model in~\cite{MR3561787}, we apply the gene flow model to mosquito populations.
We take \(\mathcal N(x,\rho)=E(x)G(\rho)\), where \(E(x)=1+A_E\exp[-|x-c_E|^2/(2\sigma_E^2)]\) describes spatial environmental capacity and \(G(\rho)=b_u[s_h\rho^2-(s_f+s_h)\rho+1]\) models trait-dependent reproductive capacity.
Here \(b_u\) denotes the fecundity of uninfected mosquitoes, while \(s_f\) and \(s_h\) represent the fitness cost of infection and the strength of cytoplasmic incompatibility, respectively.
Thus, \(a(x,\rho)=\nabla E(x)/E(x)\) and \(q(x,\rho)=G'(\rho)/G(\rho)\).
We choose \(r(\rho)=\rho(1-\rho)(\rho-\theta)\), exhibiting the \textit{Allee effect}, where trait proportions above \(\theta\) tend to invade and those below \(\theta\) tend to decay.
Here \(\theta\in(0,1)\) is the critical invasion threshold.

\begin{table}[bp]
\centering
\caption{Release centers, amplitudes, and widths of the five localized Gaussian components.}
\label{tab:geneflow_release_parameters}
\begin{tabular}{c ccccc}
\toprule
\(i\) & 1 & 2 & 3 & 4 & 5\\
\midrule
\(c_i\)
& \((0,0)\)
& \((6,6)\)
& \((6,-6)\)
& \((-6,6)\)
& \((-6,-6)\)\\
\(\alpha_i\)
& \(4.00\)
& \(0.50\)
& \(2.00\)
& \(1.50\)
& \(1.00\)\\
\(\eta_i\)
& \(1.80\)
& \(1.40\)
& \(1.00\)
& \(0.60\)
& \(1.00\)\\
\bottomrule
\end{tabular}
\end{table}

The terminal condition is generated by localized releases, with \(\phi(x)=\sum_{i=1}^5\alpha_i\exp[-|x-c_i|^2/(2\eta_i^2)]\) and \(g(x)=\phi(x)/[1+\phi(x)]\).
Here \(\phi\) represents the ratio of released trait-carrying individuals to the remaining population, and \(g\) maps this ratio into \([0,1)\).
The release parameters are listed in Table~\ref{tab:geneflow_release_parameters}.

Since \(m=1\), we identify \(z\in\mathbb R^{1\times2}\) with \(z\in\mathbb R^2\).
The decoupled formulation of~\eqref{eq:gene_flow_pde_backward} has \(b=0\), \(\sigma=\sqrt{2A}I_2\), and \(f(t,x,y,z)=r(y)+\sqrt{2A}\langle a(x,y),z\rangle+q(x,y)|z|^2\), where \(z=\sqrt{2A}\nabla u\) by~\eqref{eq:nonlinear_feynman_kac}.
The corresponding decoupled FBSDE is
\begin{equation}\label{eq:gene_flow_decoupled_fbsde}
\left\{
\begin{aligned}
X_t&=x+\int_0^t\sqrt{2A}I_2\,\mathrm{d}W_s,\\
Y_t&=g(X_T)+\int_t^T\bigl[r(Y_s)+\sqrt{2A}\langle a(X_s,Y_s),Z_s\rangle+q(X_s,Y_s)|Z_s|^2\bigr]\,\mathrm{d}s-\int_t^T Z_s\,\mathrm{d}W_s.
\end{aligned}
\right.
\end{equation}

Alternatively, incorporating the gene flow terms into the forward drift gives \(b(t,x,y,z)=2A\,a(x,y)+\sqrt{2A}\,q(x,y)z\), \(\sigma=\sqrt{2A}I_2\), and \(f(t,x,y,z)=r(y)\).
The corresponding coupled FBSDE is
\begin{equation}\label{eq:gene_flow_coupled_fbsde}
\left\{
\begin{aligned}
X_t&=x+\int_0^t
\bigl[2A\,a(X_s,Y_s)+\sqrt{2A}\,q(X_s,Y_s)Z_s\bigr]\,\mathrm{d}s+\int_0^t\sqrt{2A}I_2\,\mathrm{d}W_s,\\
Y_t&=g(X_T)+\int_t^T r(Y_s)\,\mathrm{d}s
-\int_t^T Z_s\,\mathrm{d}W_s.
\end{aligned}
\right.
\end{equation}
The two formulations are equivalent at the PDE level.
However, the decoupled generator contains the quadratic term \(q(x,y)|z|^2\), increasing sensitivity to errors in \(Z\).
The coupled formulation removes this term from the generator and treats it through the forward dynamics.
We therefore use~\eqref{eq:gene_flow_coupled_fbsde} in the numerical experiment.

\begin{figure}[tbp]
\centering
\includegraphics[width=\linewidth]{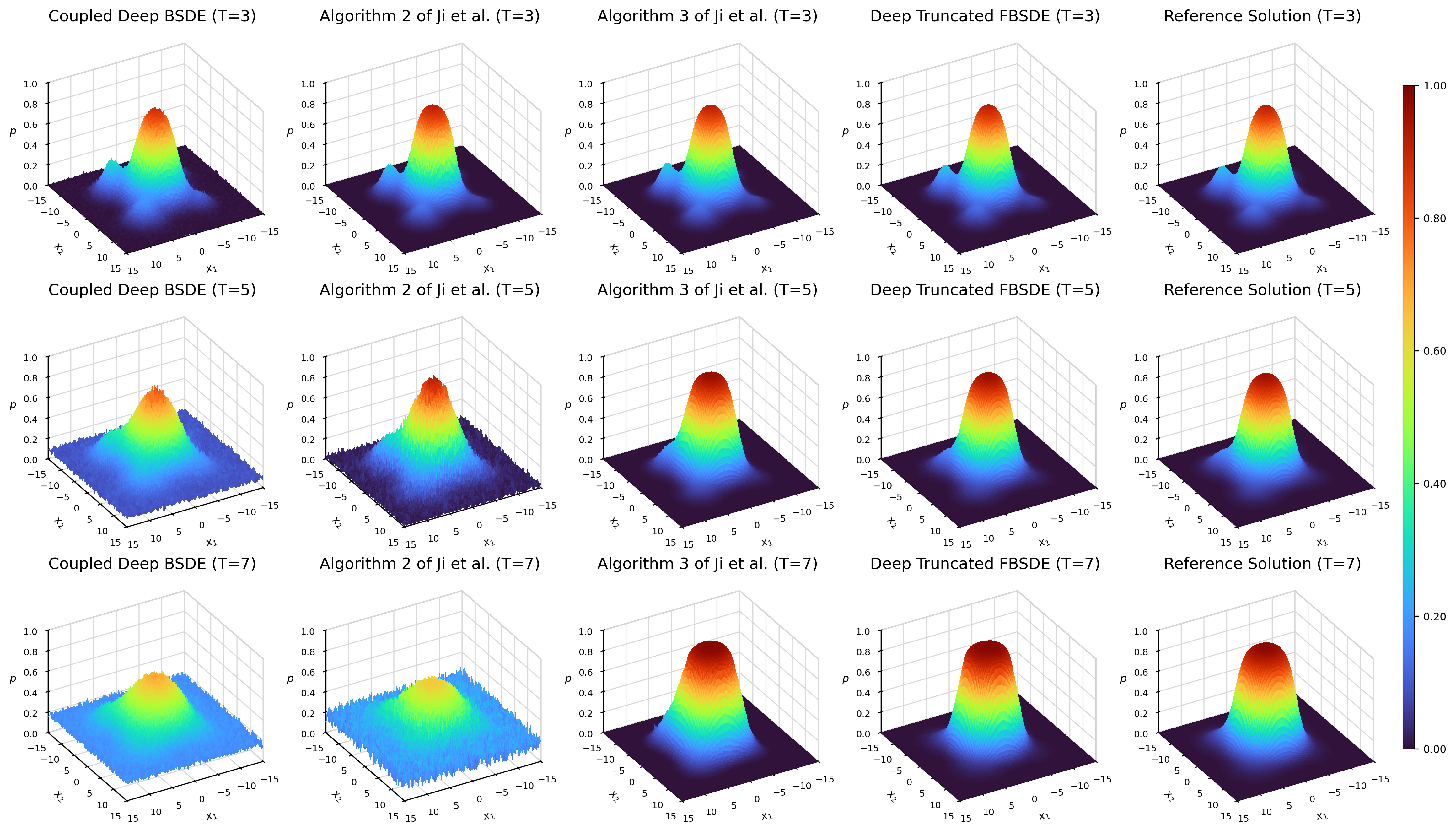}
\caption{Comparison of numerical solutions for the 2D gene flow equation~\eqref{eq:gene_flow_pde_forward} using the coupled formulation~\eqref{eq:gene_flow_coupled_fbsde}. The rows correspond to \(T=3,5,7\). The columns show coupled deep BSDE, Algorithms 2 and 3 of Ji et al., the proposed method, and the reference solution.}
\label{Fig:Comparison_GeneFlow_2D}
\end{figure}

We take \(A=0.5\), \(A_E=5.0\), \(\sigma_E=3.0\), \(c_E=(0,0)\), \(b_u=1.12\), \(s_f=0.1\), \(s_h=0.8\), and \(\theta=0.2\).
The numerical solutions are evaluated on a uniform \(101\times101\) grid over \([-15,15]^2\).
Since \(\rho\) represents a proportion, we apply the same limiter \(\Pi_{[0,1]}(\rho)=\min\{1,\max\{0,\rho\}\}\) to all methods.
The reference solution is computed using a fourth-order Runge--Kutta finite-difference method on an enlarged computational domain.
Time is measured in days.
For \(T=3,5,7\), we use \(N=30,50,70\), respectively, so that \(\Delta t=0.1\).
The hidden-layer width is \(10\), and the trainable initial value of \(Y\) is initialized to \(0\).
The proposed method uses \(K=400\) outer iterations with Adam learning rates \(5\times10^{-3}\) and \(5\times10^{-4}\) over the first and second halves, respectively.
The competing methods use \(2000\) training iterations with the same learning-rate schedule.
We use early stopping with tolerance \(10^{-4}\) and patience \(100\) Adam updates.

\begin{table}[tbp]
\centering
\caption{Relative \(L^2\) and \(L^\infty\) errors for the 2D gene flow equation at different terminal times.}
\label{Tab:Comparisons_GeneFlow_2D}
\begin{tabular}{lcccccc}
\toprule
\multirow{2}{*}{\textbf{Method}}
& \multicolumn{3}{c}{\textbf{Rel. \(L^2\) error}}
& \multicolumn{3}{c}{\textbf{Rel. \(L^\infty\) error}}\\
\cmidrule(lr){2-4}\cmidrule(lr){5-7}
& \(T=3\) & \(T=5\) & \(T=7\)
& \(T=3\) & \(T=5\) & \(T=7\)\\
\midrule
Coupled Deep BSDE
& 1.19E-01 & 4.80E-01 & 6.66E-01
& 1.24E-01 & 2.68E-01 & 3.74E-01\\
Algorithm 2 of Ji et al.
& 3.51E-02 & 2.71E-01 & 7.07E-01
& 6.00E-02 & 2.37E-01 & 4.18E-01\\
Algorithm 3 of Ji et al.
& 2.76E-02 & 5.17E-02 & 9.80E-02
& 4.12E-02 & 8.36E-02 & 1.68E-01\\
Deep truncated FBSDE
& \textbf{1.32E-02} & \textbf{1.84E-02} & \textbf{4.39E-02}
& \textbf{2.02E-02} & \textbf{2.77E-02} & \textbf{5.71E-02}\\
\bottomrule
\end{tabular}
\end{table}

Figure~\ref{Fig:Comparison_GeneFlow_2D} shows that the coupled deep BSDE method and Algorithms 2 and 3 of Ji et al.\ deteriorate as \(T\) increases, while the proposed method remains close to the reference solutions and captures the expansion and merging of the release regions.
Table~\ref{Tab:Comparisons_GeneFlow_2D} shows that the proposed method achieves the smallest relative \(L^2\) and \(L^\infty\) errors in all tested cases, with slower degradation over long horizons.
These results indicate improved long-time stability and robustness for coupled FBSDE problems.

\section{Conclusions and future work}\label{sec:conclusions}

In this paper, we developed the deep truncated FBSDE method for high-dimensional nonlinear PDEs and fully coupled FBSDEs, combining gradient-truncated iterative decoupling with fictitious-play averaging and pathwise consistency optimization to improve training stability.
We established residual-based error estimates and conditional convergence under suitable assumptions, while numerical experiments demonstrated improved accuracy and stability, including for strongly coupled and degenerate convection-dominated problems beyond the theoretical setting.
Future work will extend the convergence analysis to stronger coupling regimes, quantify neural approximation, sampling, and optimization errors, and explore broader classes of nonlinear PDEs and stochastic systems.

\section*{Acknowledgments}

The authors would like to express their sincere gratitude to Professor Enrique Zuazua for his generous support and insightful discussions.
His suggestions and encouragement have greatly contributed to the improvement of this work.

X. Cheng was supported in part by the National Natural Science Foundation of China (Grant Nos.~12401270 and 42450192), the Natural Science Foundation of Shanghai (Grant No.~24ZR1404200), and the Shanghai Magnolia Talent Plan Pujiang Project (Grant No.~24PJA007).
Y. Li was supported in part by the National Natural Science Foundation of China (Grant No.~12301566), the Science and Technology Commission of Shanghai Municipality (Grant No.~23JC1400300), and the Pujiang Program of Shanghai Magnolia Talent Plan (Grant No.~24PJD002).

\bibliographystyle{abbrv}
\begingroup
\makeatletter
\renewcommand{\url}[1]{\unskip\@ifnextchar.{\@gobble}{}}
\makeatother
\bibliography{references_arXiv}          \endgroup

\end{document}